\documentclass[11pt]{amsart}
\usepackage{amscd,amsmath,amsthm,amssymb}
\usepackage{verbatim}
\usepackage[all]{xy}
\usepackage{color}
\usepackage{hyperref}
\usepackage{tikz, float} \usetikzlibrary {positioning}

\usepackage[lined,algonl,boxed,norelsize]{algorithm2e}
\let\chapter\undefined

\def\NZQ{\mathbb}               % the font for N,Z,Q,R,C
\def\NN{{\NZQ N}}

\def\ZZ{{\NZQ Z}}

\def\frk{\mathfrak}               % font for "Fraktur"

\def\Phi{{\frk n}}
\def\Phi{{\frk N}}
\def\Ff{{\frk F}}
\def\Ef{{\frk E}}
\def\Sf{{\frk S}}
\def\Ifr{{\mathfrak I}}
\def\Bf{{\mathfrak B}}
\def\Vc{{\mathcal V}}
\def\Vc{{\mathcal V}}
\def\Xc{{\mathcal X}}
\def\Yc{{\mathcal Y}}
\def\Wc{{\mathcal W}}

\def\Gc{{\mathbf G}}

\def\Uc{{\mathcal U}}
\def\Vc{{\mathcal V}}
\def\Wc{{\mathcal W}}
\def\K{{\mathcal K}}
\def\Sc{{\mathcal S}}
\def\M{{\mathcal M}}
\def\Gb{{\mathbf G}}
\def\As{{\mathsf A}}
\def\xb{{\mathbf x}}

\def\js{{\mathsf j}}
\def\as{{\mathsf a}}
\def\opn#1#2{\def#1{\operatorname{#2}}} % to make operators
\opn\chara{char} \opn\length{\ell} \opn\pd{pd} \opn\rk{rk}
\opn\projdim{proj\,dim} \opn\injdim{inj\,dim} \opn\rank{rank}
\opn\depth{depth} \opn\grade{grade} \opn\height{height}
\opn\embdim{emb\,dim} \opn\codim{codim}

\opn\Tr{Tr} \opn\bigrank{big\,rank}
\opn\superheight{superheight}\opn\lcm{lcm}
\opn\trdeg{tr\,deg}%\emph{
\opn\reg{reg} \opn\lreg{lreg} \opn\ini{in} \opn\lpd{lpd}
\opn\size{size} \opn\sdepth{sdepth}
\opn\link{link}\opn\fdepth{fdepth}\opn\lex{lex}
\opn\LM{LM}
\opn\LC{LC}
\opn\NF{NF}
\opn\Merge{Merge}
\opn\sgn{sgn}
\opn\div{div} \opn\Div{Div} \opn\cl{cl} \opn\Pic{Pic}
\opn\Prin{Prin}
\opn\op{op}
\opn\indeg{indeg} \opn\outdeg{outdeg}
\opn\red{red}
\opn\Spec{Spec} \opn\Supp{Supp} \opn\supp{supp} \opn\Sing{Sing}
\opn\Ass{Ass} \opn\Min{Min}\opn\Mon{Mon}
\opn\Ann{Ann} \opn\Rad{Rad} \opn\Soc{Soc}
 \opn\Ker{Ker} \opn\Coker{Coker} \opn\Am{Am}
\opn\Hom{Hom} \opn\Tor{Tor} \opn\Ext{Ext} \opn\End{End}
\opn\Aut{Aut} \opn\id{id}

\opn\nat{nat}
\opn\pff{pf}%   \pf exists already
\opn\Pf{Pf} \opn\GL{GL} \opn\SL{SL} \opn\mod{mod} \opn\ord{ord}
\opn\Gin{Gin} \opn\Hilb{Hilb}\opn\sort{sort}
\opn\span{span}
\opn\Image{Image}
\opn\aff{aff} \opn\con{conv} \opn\relint{relint} \opn\st{st}
\opn\lk{lk} \opn\cn{cn} \opn\core{core} \opn\vol{vol}
\opn\link{link} \opn\star{star}\opn\lex{lex}\opn\set{set}
\opn\dist{dist}
\opn\gr{gr}

\def\pot#1#2{#1[\kern-0.28ex[#2]\kern-0.28ex]}

\opn\dirlim{\underrightarrow{\lim}}
\opn\inivlim{\underleftarrow{\lim}}
\let\to=\rightarrow

\def\Implies{\ifmmode\Longrightarrow \else
        \unskip${}\Longrightarrow{}$\ignorespaces\fi}
\def\implies{\ifmmode\Rightarrow \else
        \unskip${}\Rightarrow{}$\ignorespaces\fi}
\def\iff{\ifmmode\Longleftrightarrow \else
        \unskip${}\Longleftrightarrow{}$\ignorespaces\fi}

\let\:=\colon
\newtheorem{Theorem}{Theorem}[section]
\newtheorem{Lemma}[Theorem]{Lemma}
\newtheorem{Corollary}[Theorem]{Corollary}
\newtheorem{Proposition}[Theorem]{Proposition}

\theoremstyle{remark}
\newtheorem{Remark}[Theorem]{Remark}

\theoremstyle{definition}
\newtheorem{Example}[Theorem]{Example}

\newtheorem{Definition}[Theorem]{Definition}
\newtheorem*{Notation}{Notation}
\let\kappa=\varkappa
\def\qed{\ifhmode\textqed\fi
      \ifmmode\ifinner\quad\qedsymbol\else\dispqed\fi\fi}
\def\textqed{\unskip\nobreak\penalty50
       \hskip2em\hbox{}\nobreak\hfil\qedsymbol
       \parfillskip=0pt \finalhyphendemerits=0}
\def\dispqed{\rlap{\qquad\qedsymbol}}

\opn\dis{dis}
\def\pnt{{\raise0.5mm\hbox{\large\bf.}}}

\opn\Lex{Lex}
\opn\syz{{\rm syz}}
\opn\spoly{{\rm spoly}}
\opn\LM{{\rm LM}}
\opn\lm{{\rm lm}}
\opn\lcm{{\rm lcm}} \opn\A{\mathcal A}
\numberwithin{equation}{section}

\begin{document}

\tikzstyle{Cgray}=[draw=black, scale = .4,circle, fill = white, minimum size=3mm] \tikzstyle{Cblack}=[scale = .7,circle, fill = black, minimum size=3mm]
%\tikzstyle{Cred}=[scale = .7,circle, fill = red, minimum size=3mm]
\tikzstyle{C0}=[scale = .9,circle, fill = black!0, inner sep = 0pt, minimum
size=3mm]
\tikzstyle{C1}=[scale = .7,circle, fill = black!0, inner sep = 0pt, minimum
size=3mm]
\title {Divisors on graphs, Connected flags, and Syzygies}

%\date{December 15, 2012}


\author {Fatemeh Mohammadi}
\email{fatemeh.mohammadi716@gmail.com}
\address{Fachbereich Mathematik und Informatik, Philipps-Universit\"at, 35032 Marburg, Germany}

\author {Farbod Shokrieh}
\email{shokrieh@math.gatech.edu}
\address{Georgia Institute of Technology\\
Atlanta, Georgia 30332-0160\\
USA}


%%%%%%%%%%%%%%%%%%%%%%%%%%%%%%%%%%%%%%%%%%%%%%%%%%%%%%%%%%%%%%%%%%%%%%%%%%%%%%%%%%%%%%%%%%%%%%%%%%%%%%%%%%%%%%%
%%%%%%%%%%%%%%%%%%%%%%%%%%%%%%%%%%%%%%%%%%%%%%%%%%%%%%%%%%%%%%%%%%%%%%%%%%%%%%%%%%%%%%%%%%%%%%%%%%%%%%%%%%%%%%%

\begin{abstract}
We study the binomial and monomial ideals arising from linear equivalence of divisors on graphs from the point of view of Gr\"obner theory. We give an explicit description of a minimal Gr\"obner bases for each higher syzygy module. In each case the given minimal Gr\"obner bases is also a minimal generating set.  The Betti numbers of the binomial ideal and its natural initial ideal coincide and they correspond to the number of ``connected flags'' in the graph. In particular the Betti numbers are independent of the characteristic of the base field. For complete graphs the problem was previously studied by  Postnikov and Shapiro (\cite{PostnikovShapiro04}) 
and by Manjunath and Sturmfels 
(\cite{MadhuBernd}). 
The case of a general graph was stated as an open problem.
\end{abstract}

%%%%%%%%%%%%%%%%%%%%%%%%%%%%%%%%%%%%%%%%%%%%%%%%%%%%%%%%%%%%%%%%%%%%%%%%%%%%%%%%%%%%%%%%%%%%%%%%%%%%%%%%%%%%%%%
%%%%%%%%%%%%%%%%%%%%%%%%%%%%%%%%%%%%%%%%%%%%%%%%%%%%%%%%%%%%%%%%%%%%%%%%%%%%%%%%%%%%%%%%%%%%%%%%%%%%%%%%%%%%%%%

\maketitle

%\tableofcontents

%%%%%%%%%%%%%%%%%%%%%%%%%%%%%%%%%%%%%%%%%%%%%%%%%%%%%%%%%%%%%%%%%%%%%%%%%%%%%%%%%%%%%%%%%%%%%%%%%%%%%%%%%%%%%%%
%%%%%%%%%%%%%%%%%%%%%%%%%%%%%%%%%%%%%%%%%%%%%%%%%%%%%%%%%%%%%%%%%%%%%%%%%%%%%%%%%%%%%%%%%%%%%%%%%%%%%%%%%%%%%%%

\section{Introduction}
\label{sec:Introduction}

The theory of divisors on finite graphs can be viewed as a discrete version of the analogous theory on Riemann surfaces. This notion arises in different fields of research including the study of ``abelian sandpiles" (\cite{Dhar90,Gabrielov93}), the study of component groups of N{\'e}ron models of Jacobians of algebraic curves (\cite{RaynaudPicard, Lorenzini89}), and the theory of chip-firing games on graphs (\cite{Biggs97}). Riemann-Roch theory for finite graphs (and generalizations to tropical curves) is developed in this setting (\cite{BN1, GathmannKerber, MK08}).

\medskip

We are interested in the linear equivalence of divisors on graphs from the point of view of commutative algebra. Associated to every graph $G$ there is a canonical binomial ideal $I_G$ which encodes the linear equivalences of divisors on $G$. Let $R$ denote the polynomial ring with one variable associated to each vertex. For any two effective divisors $D_1 \sim D_2$ one can write a binomial $\xb^{D_1}-\xb^{D_2}$. The ideal $I_G \subset R$ is generated by all such binomials. Two effective divisors are linearly equivalent if and only if their associated monomials are equal in $R / I_G$. This ideal is implicitly defined in Dhar's seminal statistical physics paper \cite{Dhar90}; $R / I_G$ is the ``operator algebra'' defined there. To our knowledge, this ideal (more precisely, an affine piece of it) was first introduced in \cite{CoriRossinSalvy02} to address computational questions in chip-firing dynamics using Gr\"obner bases. From a purely computational point of view, there are now much more efficient methods available (see, e.g., \cite{FarbodMatt12} and references therein). However, this ideal seems to encode a lot of interesting information about $G$ and its linear systems. Some of the algebraic properties of $I_G$ (and its generalization for directed graphs) are studied in \cite{perkinson}.  Manjunath and Sturmfels \cite{MadhuBernd} relate Riemann-Roch theory for finite graphs to Alexander duality in commutative algebra using this ideal.

\medskip

In this paper, we study the syzygies and free resolutions of the ideals $I_G$ and $\ini(I_G)$ from the point of view of Gr\"obner theory. Here $\ini(I_G)$ denotes the initial ideal with respect to a natural term order which is defined after distinguishing a vertex $q$ (see Definition~\ref{Def:MonomialOrder}). When $G$ is a complete graph, the syzygies and Betti numbers of the ideal $\ini(I_G)$ are studied by Postnikov and Shapiro  \cite{PostnikovShapiro04}. Again for complete graphs, Manjunath and Sturmfels \cite{MadhuBernd} study the ideal $I_G$ and show that the Betti numbers coincide with the Betti numbers of $\ini(I_G)$. Finding minimal free resolutions for a general graph $G$ was stated as an open problem in both \cite{PostnikovShapiro04} and \cite{MadhuBernd} (also in \cite{Wilmes,perkinson}, where a conjecture is formulated). In particular, it was not known whether the Betti numbers for a general graph depend on the characteristic of the base field or not.

\medskip

We explicitly construct free resolutions for both $\ini(I_G)$ and $I_G$ for a general graph $G$ using Schreyer's algorithm. We remark that Schreyer's algorithm ``almost never'' gives a {\em minimal} free resolution. However in our situation we are able to carefully order our combinatorial objects to enforce the minimality. As a result we describe, combinatorially, the minimal Gr\"obner bases for all higher syzygy modules of $I_G$ and $\ini(I_G)$. In each case the minimal Gr\"obner bases is also a minimal generating set and the given resolution is minimal. This is shown by explicitly describing the differential maps in the constructed resolutions.
In particular, the Betti numbers of $\ini(I_G)$ and $I_G$ coincide. In other words, we have a positive answer to \cite[Question~1.1]{conca} for $I_G$ (see \cite{Fatemeh} and references therein, for other such examples).
For a complete graph the minimal free resolution for $\ini(I_G)$ is nicely structured by a Scarf complex. The resolution for $I_G$ when $G$ is a tree is given by a Koszul complex since $I_G$ is a complete intersection.

\medskip

The description of the generating sets and the Betti numbers is in terms of the ``connected flags'' of $G$. Fix a vertex $q \in V(G)$ and an integer $k$. A {\em connected $k$-flag} of $G$  (based at $q$) is a strictly increasing sequence $U_1 \subsetneq U_2 \subsetneq \cdots \subsetneq U_k =V(G)$ such that $q \in U_1$ and all induced subgraphs on vertex sets $U_{i}$ and $U_{i+1} \backslash U_{i}$ are connected. Associated to any connected $k$-flag one can assign a ``partial orientation'' on $G$ (Definition~\ref{def:Du}). Two connected $k$-flags are considered equivalent if the associated partially oriented graphs coincide. The Betti numbers correspond to the numbers of connected flags up to this equivalence. We give a bijective map between the connected flags of $G$ and minimal Gr\"obner bases for higher syzygy modules of $I_G$ and $\ini(I_G)$. For a complete graph all flags are connected and all distinct flags are inequivalent. So in this case the Betti numbers are simply the face numbers of the order complex of the poset of those subsets of $V(G)$ that contain $q$ (ordered by inclusion). These numbers can be described using classical Stirling numbers (see Example~\ref{exam:complete}). Hence our results directly generalize the analogous results in \cite{PostnikovShapiro04} and \cite{MadhuBernd}.

\medskip

The paper is structured as follows. In \S\ref{sec:Background} we fix our notation and provide the necessary background from the theory of divisors on graphs. We also define the ideal $I_G$ and the natural $\Pic(G)$-grading and a term order $<$ on the polynomial ring relevant to our setting. In \S\ref{sec:comalg} we quickly recall some basic notions from commutative algebra. Our main goal is to fix our notation for Schreyer's algorithm for computing higher syzygies, which is slightly different from what appears in the existing literature but is more convenient for our application. Also, to our knowledge Theorem~\ref{thm:GBini}, which gives a general sufficient criterion for an ideal to have the same graded Betti numbers as its initial ideal, has not appeared in the literature. In \S\ref{sec:Flag} we define connected flags and their equivalence relation. Basic properties of connected flags (up to equivalence) are studied in \S\ref{sec:technical}.
In \S\ref{sec:Syzygy} we study the free resolution and higher syzygies of our ideals from the point of view of Gr\"obner theory, and in \S\ref{sec:minimality} we show that the constructed free resolutions are  minimal. As a corollary we give our description of the graded Betti numbers in \S\ref{sec:betti} and we describe some connections with the theory of reduced divisors.

\medskip

Analogous results were obtained simultaneously (and independently) by Manjunath, Schreyer, and Wilmes in \cite{Madhu} using different techniques. 
Mania \cite{Horia} gives an alternate proof for the expression of the first Betti number (see Remark~\ref{rmk:mania}(ii)).
The constructed minimal free resolutions in this paper are in fact supported on certain cellular complexes. In \cite{FarbodFatemeh2} we describe this geometric picture for both $I_G$ and $\ini(I_G)$, making precise connections with Lawrence and oriented matroid ideals of \cite{popescu, novik}. 
Dochtermann and Sanyal have recently (independently) worked out this geometric picture in \cite{Anton} for the monomial ideal $\ini(I_G)$.

\medskip

{\bf Acknowledgements.} The first author acknowledges support from the Mathematical Sciences Research Institute and the Alexander von Humboldt Foundation. She is grateful to H${\rm \acute{e}}$l${\rm \grave{e}}$ne Barcelo and Volkmar Welker for numerous comments and many helpful conversations. She would like to express her gratitude to Bernd Sturmfels for introducing her to this concept. The second author's work was partially supported by NSF grants DMS-0901487. Part of this work was done while the second author was visiting UC Berkeley, which he would like to thank for the hospitality. We thank the referees for their very helpful suggestions and remarks. We would like to thank Lukas Katth\"an for pointing out the reference \cite{Kamoi}.

%%%%%%%%%%%%%%%%%%%%%%%%%%%%%%%%%%%%%%%%%%%%%%%%%%%%%%%%%%%%%%%%%%%%%%%%%%%%%%%%%%%%%%%%%%%%%%%%%%%%%%%%%%%%%%%
%%%%%%%%%%%%%%%%%%%%%%%%%%%%%%%%%%%%%%%%%%%%%%%%%%%%%%%%%%%%%%%%%%%%%%%%%%%%%%%%%%%%%%%%%%%%%%%%%%%%%%%%%%%%%%%

%%%%%%%%%%%%%%%%%%%%%%%%%%%%%%%%%%%%%%%%%%%%%%%%%%%%%%%%%%%%%%%%%%%%%%%%%%%%%%%%%%%%%%%%%%%%%%%%%%%%%%%%%%%%%%%
%%%%%%%%%%%%%%%%%%%%%%%%%%%%%%%%%%%%%%%%%%%%%%%%%%%%%%%%%%%%%%%%%%%%%%%%%%%%%%%%%%%%%%%%%%%%%%%%%%%%%%%%%%%%%%%

\section{Definitions and background}
\label{sec:Background}

\subsection{Graphs and divisors}
Throughout this paper, a {\em graph} means a finite, connected, unweighted multigraph with no loops. As usual, the set of vertices and edges of a graph $G$ are denoted by $V(G)$ and $E(G)$. We set $n=|V(G)|$. A {\em pointed graph} $(G,q)$ is a graph together with a choice of a distinguished vertex $q \in V(G)$.

\medskip

For a subset $S \subseteq V(G)$, we denote by $G[S]$ the induced subgraph of $G$ with the vertex set $S$; the edges of $G[S]$ are exactly the edges that appear in $G$ over the set $S$. We use ``$S$ is connected'' and ``$G[S]$ is connected'' interchangeably.

\medskip

Let $\Div(G)$ be the free abelian group generated by $V(G)$. An element  of $\Div(G)$ is written as
$\sum_{v \in V(G)} a_v (v)$
 and is called a {\em divisor} on $G$. The coefficient $a_v$ in $D$ is also denoted by $D(v)$.  A divisor $D$ is called {\em effective} if $D(v) \geq 0$ for all $v\in V(G)$. The set of effective divisors is denoted by $\Div_{+}(G)$. We write $D \leq E$ if $E-D \in \Div_{+}(G)$. For $D \in \Div(G)$, let $\deg(D) = \sum_{v \in V(G)} D(v)$. For $D_1, D_2 \in \Div(G)$, the divisor $E=\max(D_1,D_2)$ is defined by $E(v)=\max(D_1(v),D_2(v))$ for $v \in V(G)$.

\medskip

We denote by $\M(G)$ the group of integer-valued functions on the vertices. For $A \subseteq V(G)$, $\chi_A \in \M(G)$ denotes the $\{0,1\}$-valued characteristic function of $A$. The {\em Laplacian operator} $\Delta : \M(G) \to \Div(G)$ is defined by

\[\Delta(f) = \sum_{v \in V(G)} \sum_{\{v,w\} \in E(G)} (f(v) - f(w)) (v) \  .\]

The group of {\em principal divisors} is defined as the image of the Laplacian operator and is denoted by $\Prin(G)$. It is easy to check that $\Prin(G) \subseteq \Div^0(G)$, where $\Div^0(G)$ denotes the subgroup consisting of divisors of degree
zero. The quotient
$\Pic^0(G) = \Div^0(G) / \Prin(G)$
is a finite group whose cardinality is the number of spanning trees of $G$ (see, e.g., \cite{FarbodMatt12} and references therein). The full {\em Picard group} of $G$ is defined as
\[
\Pic(G) = \Div(G) / \Prin(G)
\]
which is isomorphic to $\ZZ \oplus \Pic^0(G)$. Since principal divisors have degree zero, the map $\deg: \Div(G) \rightarrow \ZZ$ descends to a well-defined map $\deg: \Pic(G) \rightarrow \ZZ$. Two divisors $D_1$ and $D_2$ are called {\em linearly equivalent} if they become equal in  $\Pic(G)$. In this case we write $D_1 \sim D_2$. The {\em linear system} $|D|$ of $D$ is defined as the set of effective divisors that are linearly equivalent to $D$.

\medskip

To an ordered pair of {\em disjoint} subsets $A, B \subseteq V(G)$  we assign an effective divisor
\begin{equation}\label{DAB}
D(A,B)=\sum_{v \in A }{|\{w \in B : \, \{v,w\} \in E(G)\}|(v)} \ .
\end{equation}
In other words, the support of $D(A,B)$ is a subset of $A$ and for $v \in A$ the coefficient of $(v)$ in $D(A,B)$ is the number of edges between $v$ and $B$. We define
\[d(A,B)=\deg(D(A,B))\]
 which is the number of edges of $G$ with one end in $A$ and the other end in $B$. Although in general $D(A,B) \neq D(B,A)$ we always have $d(A,B)=d(B,A)$.

\subsection{Divisors on graphs and the polynomial ring}
Let $K$ be a field and let $R=K[\xb]$ be the polynomial ring in the $n$ variables $\{x_v: v \in V(G)\}$. Any effective divisor $D$ gives rise to a monomial
\[
\xb^D := \prod_{v \in V(G)}{x_{v}^{D(v)}} \ .
\]

\subsubsection{Gradings} \label{sec:grade}
For an abelian group $\As$, the polynomial ring $R$ is said to be $\As$-graded (or graded by $\As$) if it is endowed with an $\As$-valued degree homomorphism $\deg_{\As} : \Div(G) \rightarrow \As$. This is equivalent to fixing a semigroup homomorphism $\deg_{\As} : \Div_{+}(G) \rightarrow \As$. Let $\deg_{\As}(\xb^D)=\deg_{\As}(D)$. For $\as \in \As$ let $R_{\as}$ denote the $K$-vector space of homogeneous polynomials of degree $\as$. If there is no homogeneous polynomial of degree $\as$ we let $R_{\as}=0$.

\medskip

There are three natural gradings of $R$ in our setting:
\begin{itemize}
\item[(i)] $\As=\ZZ$ and $\deg_{\As}(\xb^D)=\deg(D)$. This is the coarse $\ZZ$-grading of $R$.
\item[(ii)] $\As=\Div(G)$ and $\deg_{\As}(\xb^D)=D$. This is the fine $\ZZ^n$-grading of $R$.
\item[(iii)]  $\As=\Pic(G)$ and $\deg_{\As}(\xb^D)=[D]$, where $[D]$ denotes the equivalence class of $D$ in $\Pic(G)$.
\end{itemize}

\medskip

Gradings (i) and (ii) are, of course, well known. For the grading in (iii) (which is finer than the grading in (i) and is coarser than the grading in (ii)) we have the following lemma.

\begin{Lemma}\label{lem:pic1}
Let $\As=\Pic(G)$ and $\deg_{\As}(\xb^D)=[D]$ as above. Then $R_0=K$ and, for each $\as \in \Pic(G)$, the graded piece $R_{\as}$ is finite-dimensional.
\end{Lemma}
We remark that by \cite[Theorem 8.6]{MillerSturmfels} the two conclusions in this lemma are, in general, equivalent. 
\begin{proof}
For each $[D] \in \Pic(G)$, the graded piece $R_{[D]}$ is spanned (as a $K$-vector space) by $\{\xb^E : \, E \in |D|\}$ which is a finite set. This is because if $E \in |D|$, then, in particular, $\deg(E)=\deg(D)$.
 If $D \sim 0$ then $\deg(D)=0$. So if $D$ is effective we get $D=0$. This means $R_0=K$.
\end{proof}

Let $R=\bigoplus_{\as \in \Pic(G)}R_{\as}$ and $\mathfrak{m}=\bigoplus_{0 \ne \as \in \Pic(G)}R_{\as}$. It follows from Lemma~\ref{lem:pic1} that $\mathfrak{m}$ is a maximal ideal of $R$.
Consider the map $u:=\deg \circ \deg_{\As}:\Div(G) \rightarrow \ZZ$ sending $D$ to $\deg([D])=\deg(D)$. It is clear that $u$ takes every nonzero element of $\Div_{+}(G) $ to a strictly positive integer. Equivalently $u(r) \geq 1$ for every nonzero monomial. We can use this observation to prove that Nakayama's lemma holds for $R$ with respect to the $\Pic(G)$-grading (see, e.g., \cite[Proposition~1.4]{Kamoi}).

\begin{Lemma}\label{lem:picgrad}
Let $R$ and $\mathfrak{m}$ be as above. Then for every finitely generated $\Pic(G)$-graded module $M$ such that $\mathfrak{m}M=M$ we have $M=0$.
\end{Lemma}
\begin{proof}
Suppose $M \ne 0$. Write $M=\bigoplus_{\as \in \Pic(G)}M_{\as}$. For any graded piece $M_{\as}$ let $u(M_{\as})$ denote the integer $u(m_{\as})$ for any $m_{\as} \in M_{\as}$. Let $u(M)=\min_{\as \in \Pic(G)}{u(M_{\as})}$. Since $M$ is assumed to be finitely generated $u(M) > -\infty$. Since $u(r) \geq 1$ for all $r \in \mathfrak{m}$ we have $u(\mathfrak{m}M) > u(M)$ and therefore $\mathfrak{m}M \ne M$.
\end{proof}

Because of Lemma~\ref{lem:picgrad} the notions of minimal generating set for (finitely generated) modules, minimal free resolution, and graded Betti numbers all make sense for the $\Pic(G)$-grading. Note that we need Lemma~\ref{lem:picgrad} since $\Pic(G)$ is {\em not} in general torsion-free, and the $\Pic(G)$-grading of $R$ is not a ``positive multigrading'' in the sense of \cite[Definition 8.7]{MillerSturmfels}.

\subsubsection{The binomial ideal $I_G$}
Associated to every graph $G$ there is a canonical ideal which encodes the linear equivalences of divisors on $G$. This ideal is implicitly defined in Dhar's seminal paper \cite{Dhar90}. The ideal was introduced in \cite{CoriRossinSalvy02} to address computational questions in chip-firing dynamics using Gr\"obner bases.

\begin{Definition}\label{Def:BinomialIdeal}
\[
\begin{aligned}
I_G&:= \langle \xb^{D_1} - \xb^{D_2} : \, D_1 \sim D_2 \text{ both effective divisors}\rangle\\ \
& =\span_{K}\{\xb^{D_1} - \xb^{D_2} : \, D_1 \sim D_2 \text{ both effective divisors}\} \ .
\end{aligned}
\]
\end{Definition}

Clearly this ideal is graded (or homogeneous) with respect to the $\ZZ$ and $\Pic(G)$ gradings described in \S\ref{sec:grade} ((i) and (iii)). It follows that the quotient $R/I_G$ is both $\ZZ$-graded and $\Pic(G)$-graded as an $R$-module.

\subsubsection{A natural term order}

Once we fix a vertex $q$, there is a family of natural monomial orders that gives rise to a particularly nice  Gr\"obner bases for $I_G$. This term order was first introduced in \cite{CoriRossinSalvy02}.

Fix a pointed graph $(G,q)$. Consider a total ordering of the set of variables $\{x_v: v \in V(G)\}$ compatible with the distances
 of vertices from $q$ in $G$:
\begin{equation}\label{eq:dist}
\dist(w,q) < \dist(v,q) \, \implies \,  x_w < x_v   \ .
\end{equation}
Here, the distance between two vertices in a graph is the number of edges in a shortest path connecting them. The above ordering can be thought of an ordering on vertices induced by running the breadth-first search algorithm starting at the root vertex $q$.

\begin{Definition} \label{Def:MonomialOrder}
We denote by $<$ the degree reverse lexicographic ordering on $R=K[\xb]$ induced by the total ordering on the variables given in \eqref{eq:dist}.
\end{Definition}

We remark that the choice of the vertex $q$ is implicit in this notation.

\begin{Remark}\label{rmk:potential}
The ``total potential'' functional $b_q(\cdot)$ from \cite{FarbodMatt12} is in the Gr\"obner cone of $<$. In fact it corresponds to the barycenter of this cone (see \cite[\S{3.3} and \S{3.4}]{FarbodFatemeh2}).
\end{Remark}

\medskip

Throughout this paper $\ini(I_G)$ denotes the initial ideal of $I_G$ with respect to this term order. Note that $\ini(I_G)$ is denoted by $M_G$ in \cite{PostnikovShapiro04}.

%%%%%%%%%%%%%%%%%%%%%%%%%%%%%%%%%%%%%%%%%%%%%%%%%%%%%%%%%%%%%%%%%%%%%%%%%%%%%%%%%%%%%%%%%%%%%%%%%%%%%%%%%%%%%%%
%%%%%%%%%%%%%%%%%%%%%%%%%%%%%%%%%%%%%%%%%%%%%%%%%%%%%%%%%%%%%%%%%%%%%%%%%%%%%%%%%%%%%%%%%%%%%%%%%%%%%%%%%%%%%%%
\section{Commutative algebra: syzygies and Betti numbers} \label{sec:comalg}

In this section, we quickly recall some basic notions from commutative algebra. Our main goal is to fix our notation. A secondary goal is to keep the paper self-contained. Most of the material here is well known and we refer to standard books (e.g. \cite{Eisenbud, Singular}) for proofs and more details. To our knowledge Theorem~\ref{thm:GBini}, which gives a general sufficient criterion for an ideal to have the same Betti numbers as its initial ideal, has not appeared in the literature.

\medskip

Let $K$ be any field and let $R=K[\xb]$ be the polynomial ring in $n$ variables graded by an abelian group $\As$. The degree map will be denoted by $\deg$. Whenever we talk about notions like minimal generating sets, minimal free resolutions, or graded Betti numbers, we further assume that the grading is ``nice'' in the sense that Nakayama's lemma holds and these notions are well defined. In this case we let $\mathfrak{m}$ denote the corresponding maximal ideal of $R$ consisting of nonunit elements. Examples of such nice gradings are all the gradings in \S\ref{sec:grade} as well as ``positive multigradings'' in the sense of \cite[Definition 8.7]{MillerSturmfels} (which generalizes the $\ZZ$ and $\Div(G)$ gradings, but not the $\Pic(G)$-grading).

\subsection{Syzygies}

 Let $F_{-1}$ be the free $R$-module generated by a finite set $E$. Elements of $F_{-1}$ will be written as formal sums (with coefficients in $R$) of symbols $[e]$ (one symbol $[e]$ for each $e \in E$). Fix a {\em module ordering} $<_0$  on $F_{-1}$ extending the monomial ordering $<$ on $R$. Recall that a module ordering on $F_{-1}$ is a total ordering on the set of ``monomials'' $\xb^{\alpha}[e]$ (for $\alpha \in \NN^n$ and $e \in E$) extending a monomial ordering on $R$ and compatible with the $R$-module structure. As usual $\LM$ will denote the {\em leading monomial} with respect to the associated ordering on monomials.

\medskip

Let $M$ be a graded submodule of $F_{-1}$. Assume that the finite totally ordered set $(\Gc , \prec )$ forms a Gr\"obner bases for $(M,<_0)$ consisting of homogeneous elements.  Let $F_0$ be the free module generated by $\Gc$. For $g \in \Gc$ we let the formal symbol $[g]$ denote the corresponding generator for $F_0$; each element of $F_0$ can be written as a sum of these formal symbols with coefficients in $R$.

\medskip

There is a natural surjective homomorphism
\[\varphi_0: F_0\longrightarrow M \subseteq F_{-1} \]
sending $[g]$ to $g$ for each $g \in \Gc$. Moreover, we force this homomorphism to be graded (or homogeneous of degree $0$) by defining
\[
\deg([g]):=\deg(g) \quad \text{for all} \quad g \in \Gc \ .
\]

By definition the syzygy module of $M$ with respect to $\Gc$, denoted by $\syz(\Gc)$, is the kernel of this map. Let $\syz_0(\Gc):=M$ and $\syz_1(\Gc):=\syz(\Gc)$. For $i>1$ the higher syzygy modules are defined as $\syz_i(\Gc):= \syz(\syz_{i-1}(\Gc))$.

\begin{Remark}
Since $R$ is a graded ring, if $\Gc$ is a {\em minimal} set of homogeneous generators of $M$ then $\syz(M):= \syz(\Gc)$ is well defined (i.e. independent of the choice of the generating set $\Gc$) up to a graded isomorphism.
\end{Remark}

\subsection{Gr\"obner bases for syzygy modules}

We now discuss a method to compute a Gr\"obner bases for $\syz(\Gc)$.

\medskip

One can ``pull back'' the module ordering $<_0$ from $F_{-1}$ along $\varphi_0$ to get a compatible module ordering $<_1$ on $F_0$; for $f,h \in \Gc$ define
\begin{equation}\label{orderpull}
\xb^{\beta} [h]<_{1} \xb^{\alpha}[f] \Leftrightarrow
\begin{cases}
\LM(\xb^{\beta}h) <_0 \LM(\xb^{\alpha}f) &\\
\text{or} &\\
\LM(\xb^{\beta}h) = \LM(\xb^{\alpha}f) \quad \text{ and } \quad f \prec h.
\end{cases}
\end{equation}

Note that both $<_0$ and $<_1$ extend the same monomial ordering $<$ on $R$. Also, the module ordering $<_1$ on $F_0$ depends on both $<_0$ on $F_{-1}$ and on the totally ordered set $(\Gc , \prec )$.

\medskip

To simplify the notation we assume the leading coefficients of all elements of $\Gc$ are $1$. Suppose we are given a pair of elements $f \prec h$ of $\Gc$ such that
\[\LM(f)=\xb^{\alpha(f)} [e] \quad \text{and} \quad \LM(h)=\xb^{\alpha(h)} [e]\]
for some $e \in E$. Since $\Gc$ is a Gr\"obner bases, setting $\gamma(f,h):=\max(\alpha(f),\alpha(h))$ (the entry-wise maximum), we have the ``standard representation'':
\begin{equation}\label{eq:standard}
\spoly(f,h)=\xb^{\gamma(f,h)-\alpha(f)}f- \xb^{\gamma(f,h)-\alpha(h)} h=\sum_{g \in \Gc} a_g^{(f,h)} g
\end{equation}
for some polynomials $a_g^{(f,h)} \in R$. We set
\begin{equation}\label{eq:s(f,g)}
s(f,h)=\xb^{\gamma(f,h)-\alpha(f)}[f]- \xb^{\gamma(f,h)-\alpha(h)}[h]-\sum_{g \in \Gc} a_g^{(f,h)} [g]  \in F_0\ .
\end{equation}

Since $f,h \in \Gc$ are by assumption homogeneous
\[\deg([f])=\deg(f)=\deg(\LM(f))=\deg(\alpha(f))+\deg([e]) \ , \]  \[\deg([h])=\deg(h)=\deg(\LM(h))=\deg(\alpha(h))+\deg([e])\ .\]
 It follows that $s(f,h)$ is also homogeneous and its degree is equal to $\deg(\gamma(f,h))+\deg([e])$. Also, by definition $s(f,h) \in \syz(\Gc)$. More is true:

\begin{Theorem}[Schreyer \cite{Schreyer}] \label{thm:Schreyer}
The set
\[\mathcal{S}(\Gc)=\{s(f,h)  : \,  f , h \in \Gc \, , \, f \prec h \, , \, \LM(f)=\xb^{\alpha(f)} [e] \, , \,  \LM(h)=\xb^{\alpha(h)} [e] \text{ for some } e \in E \}\]
forms a homogeneous Gr\"obner bases for $(\syz(\Gc),<_1)$.
\end{Theorem}

Both the module ordering $<_1$ and the Gr\"obner bases $\mathcal{S}(\Gc)$ depend on $<_0$ and on $(\Gc, \prec)$.
\begin{Lemma} \label{lem:lm}
With respect to $<_1$ we have
 $\LM(s(f,h))=\xb^{\gamma(f,h)-\alpha(f)}[f]$.
\end{Lemma}
\begin{proof}
From the ``standard representation'' \eqref{eq:standard} we know
\[\LM(\sum_{g \in \Gc} a_g^{(f,h)} g)=\LM(\xb^{\gamma(f,h)-\alpha(f)}f- \xb^{\gamma(f,h)-\alpha(h)} h) \ <_0 \LM(\xb^{\gamma(f,h)-\alpha(f)}f)\]
and therefore from \eqref{orderpull} we obtain
\[\LM(\sum_{g \in \Gc} a_g^{(f,h)} [g])<_1 \xb^{\gamma(f,h)-\alpha(f)}[f] \ .\]
Moreover we have
\[\LM(\xb^{\gamma(f,h)-\alpha(h)} h) = \LM(\xb^{\gamma(f,h)-\alpha(f)}f)  \quad \text{and} \quad f \prec h  \ .\]
Again \eqref{orderpull} implies
\[ \xb^{\gamma(f,h)-\alpha(h)} [h] <_1 \xb^{\gamma(f,h)-\alpha(f)}[f]\ .\qed\]

\end{proof}

To read the Betti numbers for $M$ one needs to find a {\em minimal generating set} for the syzygy modules. In general the set $\mathcal{S}(\Gc)$ is far from being even a {\em minimal Gr\"obner bases}. One criterion to detect some of the redundant bases elements is given in the following lemma.

\begin{Lemma}
\label{lem:Chain criterion}
Let $\mathcal{S}(\Gc)$ be as in Theorem~\ref{thm:Schreyer}. Let $f_1 \prec f_2$ and $f_1 \prec f_3$ and
$
\LM(f_i)=\xb^{\alpha(f_i)} [e]
$
 for some $e \in E$ and $1 \leq i \leq 3$. If
$
\alpha(f_2) \leq \gamma(f_1, f_3)
$
then ${\mathcal{S}(\Gc)} \backslash \{s(f_1,f_3)\}$ is already a Gr\"obner bases for $(\syz(\Gc),<_1)$.
\end{Lemma}
\begin{proof}
The inequality implies $\gamma(f_1 ,f_2)-\alpha(f_1) \leq \gamma(f_1, f_3)-\alpha(f_1)$ which means \[\LM(s(f_1,f_2)) \mid \LM(s(f_1,f_3)) \ .\qed\]

\end{proof}

\begin{Remark}\label{rem:S(M)}
By repeatedly applying Lemma~\ref{lem:Chain criterion} we can find a subset $\Sc_{\min}(\Gc)$ of $\Sc(\Gc)$ which has the following properties:
\begin{itemize}
\item[(1)] $\Sc_{\min}(\Gc)$ forms a Gr\"obner bases for $(\syz(\Gc),<_1)$,
\item[(2)] there are no pair of elements $s(f,h),s(f,g) \in \Sc_{\min}(\Gc)$ such that \[\LM(s(f,h)) \mid \LM(s(f,g))\ .\]
\end{itemize}
In other words (see Lemma~\ref{lem:lm}) $\Sc_{\min}(\Gc)$ is a {\em minimal Gr\"obner bases} for $(\syz(\Gc),<_1)$.
\end{Remark}

\subsection{Free resolutions from Gr\"obner theory} \label{sec:freeres}
One can use Theorem~\ref{thm:Schreyer} to construct a graded free resolution of $M$ by induction on the homological degree. We summarize this procedure in Algorithm~\ref{alg:schr} which is due to Schreyer \cite{Schreyer} (also Spear \cite{Spear}, see, e.g., \cite{Eisenbud}).

%%%%%%%%%%%%%%%%%%%%%%%%%%%%%%%%%%%%%
\begin{algorithm}
\caption{Algorithm for computing a free resolution of $M$ (Schreyer's algorithm)}

\KwIn{\\ Graded polynomial ring $R=K[\xb]$ , \\
Monomial ordering $<$ on $R$ ,\\
Free $R$-module $F_{-1}$ generated by formal symbols $\{[e]\}_{e \in E}$ , \\
Graded $R$-submodule $M$ of $F_{-1}$ , \\
Module ordering $<_0$  on $F_{-1}$ extending the monomial ordering $<$ ,\\
Finite set $\Gc$ forming a homogeneous Gr\"obner bases for $(M,<_0)$ .\\
}
\BlankLine
\KwOut{\\ {\em A} graded free resolution: $ \cdots \rightarrow F_{i} \xrightarrow{\varphi_{i}} F_{i-1} \rightarrow \cdots \rightarrow F_0 \xrightarrow{\varphi_{0}} M \rightarrow 0$ .
}

\BlankLine

{\bf Initialization:}
\\
$\Gc_0:=\Gc$ ; \\

$F_0:=$ free $R$-module generated by formal symbols $\{[g]\}_{g \in \Gc_0}$ ; {\bf Output} $F_0$ ;\\

$\varphi_0: F_{0} \rightarrow M \subseteq F_{-1}$ defined by $[g] \mapsto g$  for each $g \in \Gc_{0}$ ;   {\bf Output} $\varphi_{0}$ ;\\

$i=0$ ;\\
\BlankLine

\While{$ F_i\ne 0$}{
\BlankLine
$\prec_{i}$ : arbitrary total ordering on $\Gc_{i}$ ;\\
\BlankLine
$<_{i+1} $ : module ordering on $F_{i}$ obtained from $<_{i}$ on $F_{i-1}$ (as in \eqref{orderpull}) ;\\
\BlankLine
$\Gc_{i+1}:=\Sc_{\min}(\Gc_{i}) \subset F_{i}$, a minimal Gr\"obner bases of $(\syz_{i+1}(\Gc), <_{i+1})$ (as in Theorem~\ref{thm:Schreyer} and Remark~\ref{rem:S(M)}) ;\\
\BlankLine
$F_{i+1}$:= free $R$-module generated by formal symbols $\{[u]\}_{u \in \Gc_{i+1}}$ ; {\bf Output} $F_{i+1}$ ;
\BlankLine
$\varphi_{i+1}: F_{i+1} \rightarrow F_{i}$ defined by $[u] \mapsto u$  for each $u \in \Gc_{i+1}$ ; {\bf Output} $\varphi_{i+1}$ ;
\BlankLine
$i \leftarrow i+1$ ;
}

\label{alg:schr}
\end{algorithm}
%%%%%%%%%%%%%%%%%%%%%%%%%%%%%%%%%%%%%

It follows immediately from definitions that the output sequence is exact and that we obtain a free resolution in this way. We note that the constructed free resolution is in general not minimal.

\medskip

\begin{Remark}
The map $\varphi_{i+1}: \, F_{i+1} \rightarrow F_{i}$ can be described more explicitly. Since $u \in \Gc_{i+1}=\Sc_{\min}(\Gc_{i})$ there are two elements $f \prec h$ in $\Gc_{i}$ such that
\[
u=s(f,h)=\xb^{\gamma(f,h)-\alpha(f)}[f]- \xb^{\gamma(f,h)-\alpha(h)}[h]-\sum_{g \in \Gc_{i}} a_g^{(f,h)} [g]  \ .
\]
In other words
\begin{equation}\label{phi_map}
\varphi_{i+1}([u])=\xb^{\gamma(f,h)-\alpha(f)}[f]- \xb^{\gamma(f,h)-\alpha(h)}[h]-\sum_{g \in \Gc_{i}} a_g^{(f,h)} [g]   \ .
\end{equation}

Since $\{[u] : \, u \in \Gc_{i+1}\}$ is the set of bases elements for $F_{i+1}$ and $\{[g] : \, g \in \Gc_{i}\}$ is the set of bases elements for $F_{i}$, the set of equalities in \eqref{phi_map}, as $u$ runs through the set $\Gc_{i}$, determines the map $\varphi_{i+1}$ completely. If we fix a labeling for the elements of $(\Gc_{i+1},\prec_{i+1})$ and $(\Gc_{i},\prec_{i})$ we can write down the corresponding matrix for $\varphi_{i+1}$ from this data.
\end{Remark}

\medskip

\begin{Remark} \label{rmk:schreyerminimal}
Although any total ordering $\prec_{i}$ on the sets $\Gc_{i}$ would work in Algorithm~\ref{alg:schr}, it follows from Lemma~\ref{lem:lm} and Lemma~\ref{lem:Chain criterion} that the ``quality of output'' very much depends on the choice of these total orderings; how far the free resolution produced by the algorithm is from being minimal depends on the choice of the total ordering in a crucial way.
\end{Remark}
%%%%%%%%%%%%%%%%%%%%%%%%%%%%%%%%%%%%%%%%%%%%%%%%%%%%%%%%%%%%%%%%%%%%%%%%%%%%%%%%%%%%%%%%%%%%%%%%%%%%%%%%%%%%%%%
%%%%%%%%%%%%%%%%%%%%%%%%%%%%%%%%%%%%%%%%%%%%%%%%%%%%%%%%%%%%%%%%%%%%%%%%%%%%%%%%%%%%%%%%%%%%%%%%%%%%%%%%%%%%%%%

\subsection{Minimal free resolutions and Betti numbers}

Let $R$ be a graded ring and $M$ be a graded $R$-module. Assume that
\[
\mathcal{F} \, : 0 \rightarrow \cdots \rightarrow F_{i} \xrightarrow{\varphi_{i}} F_{i-1} \rightarrow \cdots \rightarrow F_0 \xrightarrow{\varphi_{0}} M \rightarrow 0
\]
is a {\em minimal} graded free resolution (i.e., a graded free resolution such that $\varphi_{i+1}(F_{i+1}) \subseteq \mathfrak{m} F_i$ for all $i \geq 0$). The $i$-th {\em Betti number} $\beta_{i}(M)$ of $M$ is by definition the rank of $F_i$. The $i$-th {\em graded Betti number} in degree $\js \in \As$, denoted by $\beta_{i,\js}(M)$, is the rank of the degree $\js$ part of $F_i$. It is a consequence of Nakayama's lemma for graded rings that any finitely generated graded $R$-module has minimal free resolution, and that the numbers $\beta_{i,\js}(M)$ and $\beta_{i}(M)$ are independent of the choice of the minimal resolution.

\medskip

If $\Gc$ is a minimal set of homogeneous generators of $M$ then $\syz(M):= \syz(\Gc)$ is well defined up to graded isomorphism.
Similarly, setting $\syz_0(M):=M$, the $i$-th syzygy modules $\syz_i(M):=\syz(\syz_{i-1}(M))$ are well defined for all $i \geq 0$. In this case the $i$-th Betti number $\beta_{i}(M)$ is also the minimal number of generators of $\syz_i(M)$ and the graded Betti number $\beta_{i,\js}(M)$ is the minimal number of generators of the
$i$-th syzygy module $\syz_i(M)$ in degree $\js$.

\medskip

\begin{Theorem}\label{peeva}
Suppose that $\mathcal{F}$ is a minimal graded free resolution of $M$. Fix an $i \geq 0$. Let $E_i$ denote a bases for the free module $F_i$. Then $\{\varphi_{i}(f) : \, f\in E_i\}$ is a minimal system of homogeneous generators of $\syz_i(M)$.
\end{Theorem}
For a proof see, for example, \cite[Theorem~10.2]{peeva}.

\begin{Remark} \label{minfreegen}
It follows from Theorem~\ref{peeva} that if for all $i \geq 1$ and all $u\in \Gc_{i}$ the coefficients appearing in the expression \eqref{phi_map} of $\varphi_i([u])$ in terms of $\{[g] : \, g \in \Gc_{i-1}\}$ are all nonunit elements of $R$ (i.e. if they belong to the ideal $\mathfrak{m}$), then the resolution is a  {\it minimal} free resolution of $M$. In this case it follows from Theorem~\ref{peeva} that the sets $\Gc_{i}$ are {\em minimal generating sets} of $\syz_i(M)$ for all $i \geq 0$.
\end{Remark}
%%%%%%%%%%%%%%%%%%%%%%%%%%%%%%%%%%%%%%

\subsection{Betti numbers of $M$ and $\ini(M)$} \label{sec:BettiIn}

For a module ordering $<_0$ on $F_{-1}$ let $\ini(M)$ denote the ``leading module'' (i.e. the module generated by leading monomials) of $M$ with respect to $<_0$. The following theorem is well known and is a consequence of the fact that passing to $\ini(M)$ is a flat deformation (see, e.g., \cite[Theorem 8.29]{MillerSturmfels}).

\begin{Theorem}[Upper-semicontinuity]\label{thm:flat}
$\beta_{i, \js}(M)\leq \beta_{i, \js}(\ini(M))$ for all $i \geq 0$ and $\js \in \As$.
\end{Theorem}

\medskip

In the setting of Algorithm~\ref{alg:schr}, there is a general sufficient condition for equality to hold. The following result gives a general criterion guaranteeing that, if it is satisfied, then the answer of \cite[Question~1.1]{conca} is positive.

%%%%%%%%%%%%%%%%%%%%%%%%%%%%%%%%%%%%%%%%%%%%%%%%%%%%%%%%%%%%%%

\begin{Theorem}\label{thm:GBini}
If the output of Algorithm~\ref{alg:schr} is a {\em minimal} graded free resolution then $\beta_{i,\js}(M)=\beta_{i,\js}(\ini(M))$ for all $\ i \geq 0$ and $\js\in\As$.
\end{Theorem}

\begin{proof}
Let $\Gc'=\{\LM(g) : \, g \in \Gc\} \subset \ini(M)$. Since $\Gc$ forms a minimal Gr\"obner bases of $(M,<_0)$ the map
\[
\pi_0:\Gc\rightarrow \Gc' \subset \ini(M)\quad \text{ with }\quad \pi_{0}(g):=\LM(g)\quad {\rm for} \quad g\in \Gc
\]
is a bijection between $\Gc$ and $\Gc'$.
The proof for $i\geq 0$ is by induction on $i$. We show that for each $i \geq 0$ there is a free module $F'_i$ with bases elements
\[[s(\pi_{i-1}(f),\pi_{i-1}(h))]\] corresponding to bases elements $[s(f,h)]$ of $\Gc_i$. This extends to a natural bijective map from $F_i$ to $F'_i$. Moreover, this bijection induces maps
\[
\pi_i:\Gc_i\rightarrow \syz_{i}(\Gc') \quad \text{ with }\quad \pi_{i}(s(f,h)):=s({\pi_{i-1}(f),\pi_{i-1}(h)})\quad {\rm for} \quad f\prec h\ \text{ in }\ \Gc_{i-1}
\]  such that
\begin{itemize}
\item[(1)] $<_i$ is a term order on $F'_i$.
\item[(2)] $\LM(\pi_i(s(f,h)))=\xb^{\theta}[\pi_{i-1}(f)]$, where $\LM(s(f,h))=\xb^{\theta}[f]$.
\item[(3)] $\pi_i(\Gc_i)$ forms a minimal Gr\"obner bases of $(\syz_{i}(\Gc'),<_{i})$.
\end{itemize}

\medskip

We have already shown the case $i=0$.
Now assume that $i>0$ and the result holds for $i-1$.
Note that by the induction hypothesis $\pi_{i-1}$ is injective. This together with (2) and (3) for $i-1$ implies that
the elements $[s(\pi_{i-1}(f),\pi_{i-1}(h))]$ are pairwise distinct, since their leading terms are pairwise distinct.
Assume that $<_{i-1}$ is the term order on $F_{i-1}$. By the induction hypothesis we have the total order $\prec'_{i-1}$ on the elements $\pi_{i-1}(\Gc_i)$
such that
\begin{equation}\label{total order}
\pi_{i-1}(f)\prec'_{i-1} \pi_{i-1}(h)\quad\text{if and only if}\quad f\prec_{i-1} h\ .
\end{equation}
Since (2) and (3) hold for $i-1$, choosing the total order of (\ref{total order}) on the elements $\pi_{i-1}(\Gc_i)$
we will get the term order $<_i$ on $F'_i$ by (\ref{orderpull}).
To prove (2) for $i$, let $s(f,h)$ be an element of $\Gc_i$ with
\[\LM(f)=\xb^{\alpha(f)}[u]\ , \quad \LM(h)=\xb^{\alpha(h)}[u]\quad \text{ and }\quad \gamma(f,h)=\max(\alpha(f),\alpha(h))\]
for some bases element $[u]\in \Gc_{i-2}$.
Then by the induction hypothesis we have
\[
\LM(\pi_{i-1}(f))=\xb^{\alpha(f)}[\pi_{i-2}(u)]\quad\text{ and }\quad \LM(\pi_{i-1}(h))=\xb^{\alpha(h)}[\pi_{i-2}(u)]\ .
\]
Therefore Lemma~\ref{lem:lm} together with induction hypothesis implies
\[
\LM(\pi_{i}(s(f,h)))
=\LM(s({\pi_{i-1}(f),\pi_{i-1}(h)}))=\xb^{\gamma(f,h)-\alpha(f)}[\pi_{i-1}(f)]\
\]
which is (2). Now it follows that
$\pi_{i}$ is {\rm injective};
assume that $\pi_{i}(s(f,h))=\pi_{i}(s(p,q))$. Therefore $\LM(s(f,h))=\LM(s(p,q))=\xb^{\gamma(f,h)-\alpha(f)}[\pi_{i-1}(f)]$.
Now (2) implies that
\[
\LM(s(f,h))=\LM(s(p,q))=\xb^{\gamma(f,h)-\alpha(f)}[f]\ ,
\]
which is a contradiction by our assumption on $\Gc_i=\Sc_{\min}(\Gc_{i-1})$.
The fact that $\pi_i$ is injective implies that its extension from $F_i$ to $F'_i$ is a bijective map, as desired.

\medskip

Now we show that $\pi_i(\Gc_i)$ forms a Gr\"obner bases for $(\syz_{i}(\Gc'),<_{i})$.
For the sake of contradiction, assume that $\pi_i(\Gc_i)$ does not form a Gr\"obner bases for $(\syz_i({\Gc'}),<_{i})$. Then our induction hypothesis that
$\pi_{i-1}(\Gc_{i-1})$ forms a Gr\"obner bases for $(\syz_{i-1}({\Gc'}),<_{i-1})$,
implies that there exist elements $f$ and $h$ of $\Gc_{i-1}$ such that $f\prec_{i-1} h$ and $\LM(s(\pi_{i-1}(f),\pi_{i-1}(h)))$ is not divisible by the leading monomial of any element of $\pi_i(\Gc_i)$.
On the other hand, our assumption on $\Gc_{i}$ and the fact that $\LM(s(\pi_{i-1}(f),\pi_{i-1}(h)))=\LM(s(f,h))$ imply that there exists
an element $g$ in $\Gc_{i-1}$ such that $f\prec_{i-1} g$ and $\LM(s(f,g))$ divides $\LM(s(f,h))$. By the induction hypothesis $\pi_{i-1}(f)$ and $\pi_{i-1}(g)$ belong to $\syz_{i-1}(\Gc')$. Moreover by (2)
$\LM(s(\pi_{i-1}(f),\pi_{i-1}(g)))=\LM(s(f,g))$ divides
$\LM(s({\pi_{i-1}(f),\pi_{i-1}(h)}))$ which is a contradiction.
Thus $\pi_i(\Gc_{i})$ forms a Gr\"obner bases for $(\syz_i(\Gc'),<_i)$.

\medskip

Note that $\pi_i$ is a graded map of degree zero which preserves the degree of the elements $s(f,h)$ of $\Gc_i$. This together with the fact that $\pi_i(\Gc_i)$ forms a minimal Gr\"obner bases of $(\syz_i(\Gc'),<_i)$ implies that
$\beta_{i,\js}(\ini(M))\leq \beta_{i,\js}(M)$. Now Theorem~\ref{thm:flat} completes the proof.
\end{proof}

%%%%%%%%%%%%%%%%%%%%%%%%%%%%%%%%%%%%%%%%%%%%%%%%%%%%%%%%%%%%%%

%%%%%%%%%%%%%%%%%%%%%%%%%%%%%%%%%%%%%%%%%%%%%%%%%%%%%%%%%%%%%%%%%%%%%%%%%%%%%%%%%%%%%%%%%%%%%%%%%%%%%%%%%%%%%%%

\section{Connected flags on graphs}
\label{sec:Flag}

%%%%%%%%%%%%%%%%%%%%%%%%%%%%%%%%%%%%%%%%%%%%%%%%%%%%%%%%%%%%%%%%%%%%%%%%%%%%%%%%%%%%%%%%%%%%%%%%%%%%%%%%%%%%%%%

\subsection{Connected flags, partial orientations, and divisors}

From now on we fix a pointed graph $(G,q)$ and we let $n=|V(G)|$. Consider the poset
\[
\mathfrak{C}(G,q):=\{U \subseteq V(G): \, q \in U \}
\]
ordered by inclusion. The following special chains of this poset arise naturally in our setting.
\begin{Definition}
Fix an integer $1 \leq k \leq n$. A {\em connected $k$-flag} of $(G,q)$ is a (strictly increasing) sequence $\Uc$ of subsets of $V(G)$
\[
 U_1 \subsetneq U_2 \subsetneq \cdots \subsetneq U_k =V(G)
\]
such that $q \in U_1$ and, for all $1 \leq i \leq k-1$, both $G[U_i]$ and $G[U_{i+1} \backslash U_{i}]$ are connected.
\end{Definition}

The set of all connected $k$-flags of $(G,q)$ will be denoted by $\Ff_k(G, q)$.
\begin{Remark}
For a complete graph, $\Ff_k(G, q)$ is simply the order complex of $\mathfrak{C}(G,q)$, but in general $\Ff_k(G, q)$ is not a simplicial complex.
\end{Remark}

\begin{Notation}
For convenience, whenever we use index $0$ on a vertex set (e.g. $U_0, V_0, W_0$, etc.) we mean the empty set.
\end{Notation}

\begin{Definition}\label{def:Du}
Given $\Uc \in \Ff_k(G, q)$ we define:
\begin{itemize}
\item[(a)] a ``partial orientation'' of $G$ by orienting edges {\em from} $U_{i}$ {\em to} $U_{i+1} \backslash U_{i}$ (for all $1 \leq i \leq k-1$) and leaving all other edges unoriented. We denote the resulting partially oriented graph by $G(\Uc)$.

\item[(b)] an effective divisor $D(\Uc) \in \Div(G)$ given by (see \eqref{DAB})
$D(\Uc):= \sum_{i=1}^{k-1}{D(U_{i+1} \backslash U_{i}, U_{i})}$.

\end{itemize}
\end{Definition}
Note that the partial orientation in (a) is always acyclic.

\begin{Example}\label{exam:G(U)}
Let $G$ be the following graph on the vertices $v_1,v_2,\ldots,v_5$.
We let $v_1$ be the distinguished vertex. Then the partial orientation associated to
\[
\Uc:\ \ \{v_1\}\subset \{v_1,\boldsymbol{v_2}\}\subset \{v_1,v_2,\boldsymbol{v_3},\boldsymbol{v_4}\} \subset \{v_1,v_2,v_3,v_4,\boldsymbol{v_5}\}
\]
is depicted  in the following figure. Note that
$D(\Uc)=(v_2)+(v_3)+(v_4)+2(v_5)$.

 
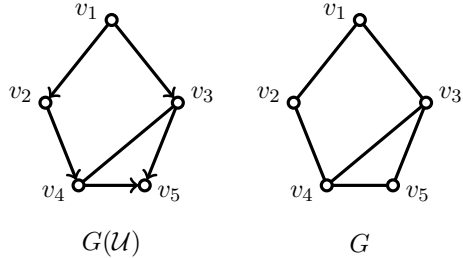
\begin{figure}[h!]
\begin{center} \begin{tikzpicture}
[scale = .22, very thick = 15mm]

  \node (n1) at (5,11) [Cgray] {};
  \node (n2) at (1,6)  [Cgray] {};
  \node (n3) at (9,6)  [Cgray] {};
  \node (n4) at (3,1)  [Cgray] {};
  \node (n5) at (7,1)  [Cgray] {};
  \foreach \from/\to in {n1/n2,n2/n4, n1/n3, n3/n5, n4/n5}
    \draw[->] (\from) -- (\to);
\foreach \from/\to in {n3/n4}
    \draw[] (\from) -- (\to);

    \node(p1) at (3.5, 11.5) [C0] {$v_1$};
    \node(p2) at (-0.5, 6.5) [C0] {$v_2$};
        \node(p3) at (10.5, 6.5) [C0] {$v_3$};
    \node(p4) at (1.5, 0.5) [C0] {${v_4}$};
    \node(p5) at (8.5, 0.5) [C0] {$v_5$};

 \node(p7) at (5, -2.5) [C0] {$G(\Uc)$};

  \node (n1) at (20,11) [Cgray] {};
  \node (n2) at (16,6)  [Cgray] {};
  \node (n3) at (24,6)  [Cgray] {};
  \node (n4) at (18,1)  [Cgray] {};
  \node (n5) at (22,1)  [Cgray] {};
  \foreach \from/\to in {n1/n2,n2/n4, n1/n3,n3/n5,n3/n4}
    \draw[] (\from) -- (\to);
\foreach \from/\to in {n4/n5}
    \draw[] (\from) -- (\to);

    \node(p1) at (3.5, 11.5) [C0] {$v_1$};
    \node(p2) at (-0.5, 6.5) [C0] {$v_2$};
        \node(p3) at (10.5, 6.5) [C0] {$v_3$};
    \node(p4) at (1.5, 0.5) [C0] {${v_4}$};
    \node(p5) at (8.5, 0.5) [C0] {$v_5$};

 \node(p7) at (5, -2.5) [C0] {$G(\Uc)$};

	\node(p1) at (18.5, 11.5) [C0] {$v_1$};
    \node(p2) at (14.5, 6.5) [C0] {$v_2$};
        \node(p3) at (25.5, 6.5) [C0] {${v_3}$};
    \node(p4) at (16.5, 0.5) [C0] {${v_4}$};
    \node(p5) at (23.5, 0.5) [C0] {$v_5$};

  \node(p6) at (20, -2.5) [C0] {$G$};

\end{tikzpicture}
\end{center}\end{figure}
\end{Example}

\begin{Remark} \label{rmk:indeg}
It is easy to check that $D(\Uc) = \sum_{v \in V(G)}{(\indeg_{G(\Uc)}(v)) (v)} $, where $\indeg_{G(\Uc)}(v)$ denotes the number of oriented edges directed {\em to} $v$ in $G(\Uc)$.
\end{Remark}

\medskip

%%%%%%%%%%%%%%%%%%%%%%%%%%%%%%%%%%%%%%%%%%%%%%%%%%%%%%%%%%%%%%%%%%%%%%%%%%%%%%%%%%%%%%%%%%%%%%%%%%%%%%%%%%%%%%%

\subsection{Total ordering on $\Ff_k(G, q)$}

We endow each $\Ff_k(G, q)$ with a {\it total ordering} $\prec_{k}$ for all $1 \leq k \leq n$. 

Let $\preceq$ denote the ordering on $\mathfrak{C}^{\text{op}}(G,q)$ (the opposite poset of $\mathfrak{C}(G,q)$) given by reverse inclusion:
\[
U \preceq V \iff U \supseteq V \ .
\]

\begin{Definition}\label{def:prec1}
We fix, once and for all, a {\em total ordering} extending $\preceq$. By a slight abuse of notation, $\preceq$ will be used to denote this total ordering extension. In particular, $\prec$ will denote the associated strict total order.
\end{Definition}

We consider one of the natural ``lexicographic extensions'' (more precisely, the reverse lexicographic extension) of $\prec$ to the set of connected $k$-flags.

\begin{Definition}\label{def:prec2}

For $\Uc \ne \Vc$ in $\Ff_k(G,q)$ written as
\[
\Uc:\ \, U_1 \subsetneq U_2 \subsetneq \cdots \subsetneq U_k =V(G)
\]
\[
\Vc:\ \, V_1 \subsetneq V_2 \subsetneq \cdots \subsetneq V_k =V(G)
\]
we say $\Uc \prec_k \Vc$ if for the maximum $1 \leq \ell \leq k-1$ with $U_{\ell} \ne V_{\ell}$ we have
$\Uc_{\ell} \prec \Vc_{\ell}$.

As usual, we write $\Uc \preceq_k \Vc$ if and only if $\Uc \prec_k \Vc$ or $\Uc =\Vc$.
\end{Definition}

\begin{Lemma}
$(\Ff_k(G, q), \preceq_k)$ is a totally ordered set.
\end{Lemma}

\begin{proof}
Let
\[
\Uc:\ \, U_1 \subsetneq U_2 \subsetneq \cdots \subsetneq U_k =V(G) \ ,
\]
\[
\Vc:\ \, V_1 \subsetneq V_2 \subsetneq \cdots \subsetneq V_k =V(G) \ ,
\]
\[
\Wc:\ \, W_1 \subsetneq W_2 \subsetneq \cdots \subsetneq W_k =V(G) \ .
\]

\medskip

If $\Uc \ne \Vc$ then there is an index $1 \leq \ell \leq k-1$ with $U_{\ell} \ne V_{\ell}$. Since $\prec$ is a strict total ordering we have either $U_{\ell} \prec V_{\ell}$ or $V_{\ell} \prec U_{\ell}$. It follows from the definition that if $\Uc \ne \Vc$ then either $\Uc \prec_k \Vc$ or $\Vc \prec_k \Uc$ (i.e., $\prec_k$ is trichotomous). It remains to show that $\prec_k$ is transitive. Assume that $\Uc \prec_k \Vc$ and $\Vc \prec_k \Wc$. Let $1 \leq \ell_1 \leq k-1$ and $1 \leq \ell_2 \leq k-1$ be such that
\[
U_{\ell_1} \prec V_{\ell_1} \text{ and } U_{i} = V_{i} \, \text{ for } \ell_1 <i \leq k \ ,
\]\[
V_{\ell_2} \prec W_{\ell_2} \text{ and } V_{i} = W_{i} \, \text{ for } \ell_2 <i \leq k \ .
\]
There are three cases:

$\bullet$  If $\ell_1=\ell_2$ we have
$
U_{\ell_1} \prec V_{\ell_1} \prec W_{\ell_1} \text{ and } U_{i} = W_{i} \text{ for } \ell_1 <i \leq k \ ,
$

$\bullet$  If $\ell_1<\ell_2$ we have
$
U_{\ell_2}=V_{\ell_2} \prec W_{\ell_2} \text{ and } U_{i} = W_{i} \text{ for } \ell_2 <i \leq k \ ,
$

$\bullet$  If $\ell_1>\ell_2$ we have
$
U_{\ell_1} \prec V_{\ell_1}=W_{\ell_1} \text{ and } U_{i} = W_{i} \text{ for } \ell_1 <i \leq k \ .
$

Therefore in any case $\Uc \prec_k \Wc$.
\end{proof}

%%%%%%%%%%%%%%%%%%%%%%%%%%%%%%%%%%%%%%%%%%%%%%%%%%%%%%%%%%%%%%%%%%%%%%%%%%%%%%%%%%%%%%%%%%%%%%%%%%%%%%%%%%%%%%%
%%%%%%%%%%%%%%%%%%%%%%%%%%%%%%%%%%%%%%%%%%%%%%%%%%%%%%%%%%%%%%%%%%%%%%%%%%%%%%%%%%%%%%%%%%%%%%%%%%%%%%%%%%%%%%%

\subsection{Equivalence relation on $\Ff_k(G,q)$}

It is easy to find two different connected $k$-flags having identical associated partially oriented graphs.

\begin{Example} \label{example:equiv}
If the connected $k$-flag
\[
\Uc:\ \, U_1 \subsetneq \cdots \subsetneq U_{i-1}  \subsetneq (U_{i-1} \cup A_{i}) \subsetneq (U_{i-1} \cup A_{i} \cup A_{i+1})\subsetneq\cdots \subsetneq U_k
\]
is such that $A_{i+1}$ is disjoint from $A_{i}$, and $d(A_{i+1},A_{i})=0$ (i.e., $A_{i+1}$ is not connected to $A_{i}$) then
\[
\Vc:\ \, U_1 \subsetneq \cdots \subsetneq U_{i-1}  \subsetneq (U_{i-1} \cup A_{i+1}) \subsetneq (U_{i-1} \cup A_{i+1} \cup A_{i}) \subsetneq\cdots \subsetneq U_k
\]
is a different connected $k$-flag, but $G(\Uc)$ and $G(\Vc)$ coincide.
\end{Example}

\begin{Example}
Let $G$ be the following graph on the vertices $v_1,v_2,\ldots,v_5$. We fix $v_1$ as the distinguished vertex. Consider the following connected flags: \begin{itemize}
\item $\Uc:\ \ \{v_1\}\subset \{v_1,\boldsymbol{v_2}\}\subset \{v_1,v_2,\boldsymbol{v_3}\} \subset \{v_1,v_2,v_3,\boldsymbol{v_4},\boldsymbol{v_5}\}$
\item $\Vc:\ \{v_1\}\subset \{v_1,\boldsymbol{v_3}\}\subset \{v_1,\boldsymbol{v_2},v_3\} \subset \{v_1,v_2,v_3,\boldsymbol{v_4 ,v_5}\}$\ .
\end{itemize}
Then, as we see, their associated graphs coincide.

\medskip

 
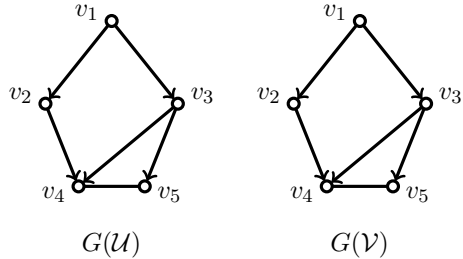
\begin{center}
 \begin{tikzpicture}
[scale = .22, very thick = 15mm]

  \node (n1) at (5,11) [Cgray] {};
  \node (n2) at (1,6)  [Cgray] {};
  \node (n3) at (9,6)  [Cgray] {};
  \node (n4) at (3,1)  [Cgray] {};
  \node (n5) at (7,1)  [Cgray] {};
  \foreach \from/\to in {n1/n2,n2/n4, n1/n3, n3/n5, n3/n4}
    \draw[->] (\from) -- (\to);
\foreach \from/\to in {n4/n5}
    \draw[] (\from) -- (\to);

  \node (n1) at (20,11) [Cgray] {};
  \node (n2) at (16,6)  [Cgray] {};
  \node (n3) at (24,6)  [Cgray] {};
  \node (n4) at (18,1)  [Cgray] {};
  \node (n5) at (22,1)  [Cgray] {};
  \foreach \from/\to in {n1/n2,n2/n4, n1/n3,n3/n5,n3/n4}
    \draw[->] (\from) -- (\to);
\foreach \from/\to in {n4/n5}
    \draw[] (\from) -- (\to);

    \node(p1) at (3.5, 11.5) [C0] {$v_1$};
    \node(p2) at (-0.5, 6.5) [C0] {$v_2$};
        \node(p3) at (10.5, 6.5) [C0] {$v_3$};
    \node(p4) at (1.5, 0.5) [C0] {${v_4}$};
    \node(p5) at (8.5, 0.5) [C0] {$v_5$};

 \node(p7) at (5, -2.5) [C0] {$G(\Uc)$};

	\node(p1) at (18.5, 11.5) [C0] {$v_1$};
    \node(p2) at (14.5, 6.5) [C0] {$v_2$};
        \node(p3) at (25.5, 6.5) [C0] {${v_3}$};
    \node(p4) at (16.5, 0.5) [C0] {${v_4}$};
    \node(p5) at (23.5, 0.5) [C0] {$v_5$};

  \node(p6) at (20, -2.5) [C0] {$G(\Vc)$};

\end{tikzpicture} \end{center}
\end{Example}

This example motivates the following definition.

\begin{Definition}
Two $k$-flags $\Uc, \Vc \in \Ff_k(G, q)$ are called {\it equivalent} if the associated partially oriented graphs $G(\Uc)$ and $G(\Vc)$ coincide.
\end{Definition}

\medskip

\begin{Remark}
\begin{itemize}
\item[]
\item[(i)] This equivalence relation is easily seen to be the transitive closure of the equivalences described in Example~\ref{example:equiv}.
\item[(ii)] When $k=2$ each equivalence class contains a unique element.
\end{itemize}
\end{Remark}

\begin{Notation}
The set of all equivalence classes in $\Ff_k(G, q)$ will be denoted by $\Ef_{k}(G, q)$. The set $\Sf_k(G, q)$ will denote the set of minimal representatives of the classes in $\Ef_{k}(G, q)$ with respect to $\prec_k$.
\end{Notation}

\begin{Example}
Here we list all acyclic (partial) orientations associated to equivalence classes of connected flags on the $4$-cycle graph on the vertices $1,2,3,4$. The vertex $1$ is chosen as the distinguished vertex (see Figures 1, 2, 3).

\begin{figure}[h!]

\label{figure thm1} \begin{center}

\begin{tikzpicture} [scale = .18, very thick = 10mm]

  \node (n4) at (4,1)  [Cgray] {1};
  \node (n1) at (4,11) [Cgray] {4};
  \node (n2) at (1,6)  [Cgray] {2};
  \node (n3) at (7,6)  [Cgray] {3};
  \foreach \from/\to in {n4/n2,n1/n3}
    \draw[] (\from) -- (\to);
\foreach \from/\to in {n2/n1,n4/n3}
    \draw[blue][->] (\from) -- (\to);

    \node (n4) at (14,1)  [Cgray] {1};
  \node (n1) at (14,11) [Cgray] {4};
  \node (n2) at (11,6)  [Cgray] {2};
  \node (n3) at (17,6)  [Cgray] {3};
  \foreach \from/\to in {n4/n2,n3/n1}
   \draw[blue][->] (\from) -- (\to); {n3/n4,n2/n4}

     \foreach \from/\to in {n1/n2,n4/n3}
    \draw[] (\from) -- (\to);

    \node (n4) at (24,1)  [Cgray] {1};
  \node (n1) at (24,11) [Cgray] {4};
  \node (n2) at (21,6)  [Cgray] {2};
  \node (n3) at (27,6)  [Cgray] {3};
  \foreach \from/\to in {n1/n2,n1/n3}
    \draw[] (\from) -- (\to);

    \foreach \from/\to in {n4/n2,n4/n3}
    \draw[blue][->] (\from) -- (\to);

\end{tikzpicture}

\bigskip

\begin{tikzpicture} [scale = .18, very thick = 10mm]

  \node (n4) at (4,1)  [Cgray] {1};
  \node (n1) at (4,11) [Cgray] {4};
  \node (n2) at (1,6)  [Cgray] {2};
  \node (n3) at (7,6)  [Cgray] {3};
   \foreach \from/\to in {n4/n2,n4/n3}
    \draw[] (\from) -- (\to);

    \foreach \from/\to in {n2/n1,n3/n1}
    \draw[blue][->] (\from) -- (\to);

    \node (n4) at (14,1)  [Cgray] {1};
  \node (n1) at (14,11) [Cgray] {4};
  \node (n2) at (11,6)  [Cgray] {2};
  \node (n3) at (17,6)  [Cgray] {3};
 \foreach \from/\to in {n1/n3,n4/n3}
    \draw[] (\from) -- (\to);

    \foreach \from/\to in {n1/n2,n4/n2}
    \draw[blue][->] (\from) -- (\to);

    \node (n4) at (24,1)  [Cgray] {1};
  \node (n1) at (24,11) [Cgray] {4};
  \node (n2) at (21,6)  [Cgray] {2};
  \node (n3) at (27,6)  [Cgray] {3};
 \foreach \from/\to in {n1/n2,n4/n2}
    \draw[] (\from) -- (\to);

    \foreach \from/\to in {n1/n3,n4/n3}
    \draw[blue][->] (\from) -- (\to);

\end{tikzpicture}

\caption{Acyclic orientations corresponding to $2$-partitions of $C_4$}
\end{center}
\end{figure}
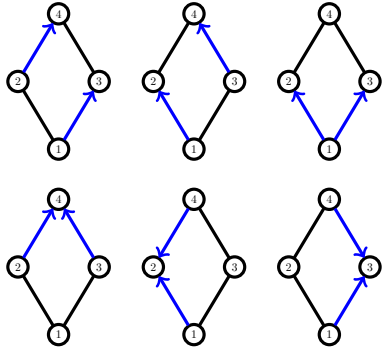

%%%%%%%%%%%%%%%%%%%%%%%%%%%%%%%%%%%

\begin{figure}[h!] \begin{center}

\begin{tikzpicture} [scale = .18, very thick = 10mm]

 \node (n4) at (-6,13)  [Cgray] {1};
  \node (n1) at (-6,23) [Cgray] {4};
  \node (n2) at (-9,18)  [Cgray] {2};
  \node (n3) at (-3,18)  [Cgray] {3};

\foreach \from/\to in {n2/n4}
   \draw[] (\from) -- (\to);
 \foreach \from/\to in {n4/n3,n2/n1,n3/n1}
    \draw[blue][->] (\from) -- (\to);

 \node (n4) at (4,13)  [Cgray] {1};

  \node (n1) at (4,23) [Cgray] {4};
  \node (n2) at (1,18)  [Cgray] {2};
  \node (n3) at (7,18)  [Cgray] {3};
 \foreach \from/\to in {n3/n4}
   \draw[] (\from) -- (\to);
 \foreach \from/\to in {n4/n2,n2/n1,n3/n1}
    \draw[blue][->] (\from) -- (\to);

 \node (n4) at (14,13)  [Cgray] {1};
  \node (n1) at (14,23) [Cgray] {4};

  \node (n2) at (11,18)  [Cgray] {2};
  \node (n3) at (17,18)  [Cgray] {3};

  \foreach \from/\to in {n1/n3}
    \draw[] (\from) -- (\to);

 \foreach \from/\to in {n4/n2,n4/n3, n2/n1}
    \draw[blue][->] (\from) -- (\to);

      \node (n4) at (24,13)  [Cgray] {1};
  \node (n1) at (24,23) [Cgray] {4};

  \node (n2) at (21,18)  [Cgray] {2};
  \node (n3) at (27,18)  [Cgray] {3};
   \foreach \from/\to in {n1/n2}
    \draw[] (\from) -- (\to);
    \foreach \from/\to in {n1/n3, n3/n4,n2/n4}
    \draw[blue][<-] (\from) -- (\to);
\end{tikzpicture}

\end{center}

\end{figure}

\vspace{-.5cm}

%%%%%%%%%%%%%%%%%%%%%%%%%%%%%%%%%%%

%%%%%%%%%%%%%%%%%%%%%
\begin{figure}[h!] \label{fig:o_j} \begin{center}

\begin{tikzpicture}  [scale = .18, very thick = 10mm]

 \node (n4) at (-6,13)  [Cgray] {1};
  \node (n1) at (-6,23) [Cgray] {4};
  \node (n2) at (-9,18)  [Cgray] {2};
  \node (n3) at (-3,18)  [Cgray] {3};
\foreach \from/\to in {n2/n1,n2/n4}
    \draw[blue][<-] (\from) -- (\to);

 \foreach \from/\to in {n4/n3}
    \draw[blue][->] (\from) -- (\to);

\foreach \from/\to in {n1/n3}
    \draw[] (\from) -- (\to);
 \node (n4) at (4,13)  [Cgray] {1};

  \node (n1) at (4,23) [Cgray] {4};
  \node (n2) at (1,18)  [Cgray] {2};
  \node (n3) at (7,18)  [Cgray] {3};
  \foreach \from/\to in {n2/n4}
    \draw[blue][<-] (\from) -- (\to);
\foreach \from/\to in {n2/n1}
    \draw[] (\from) -- (\to);
 \foreach \from/\to in {n4/n3,n1/n3}
    \draw[blue][->] (\from) -- (\to);

 \node (n4) at (14,13)  [Cgray] {1};
  \node (n1) at (14,23) [Cgray] {4};

  \node (n2) at (11,18)  [Cgray] {2};
  \node (n3) at (17,18)  [Cgray] {3};

  \foreach \from/\to in {n1/n2,n3/n1}
    \draw[blue][<-] (\from) -- (\to);

 \foreach \from/\to in {n4/n3}
    \draw[blue][->] (\from) -- (\to);
 \foreach \from/\to in {n2/n4}
    \draw[] (\from) -- (\to);

      \node (n4) at (24,13)  [Cgray] {1};
  \node (n1) at (24,23) [Cgray] {4};

  \node (n2) at (21,18)  [Cgray] {2};
  \node (n3) at (27,18)  [Cgray] {3};
  \foreach \from/\to in {n2/n1,n1/n3}
    \draw[blue][<-] (\from) -- (\to);

 \foreach \from/\to in {n4/n2}
    \draw[blue][->] (\from) -- (\to);
 \foreach \from/\to in {n3/n4}
    \draw[] (\from) -- (\to);

\end{tikzpicture}

\caption{Acyclic orientations corresponding to $3$-partitions of $C_4$}
\end{center}

\end{figure}
\medskip
%%%%%%%%%%%%%%%%%%%%%%%%%%%%%%%%%%%

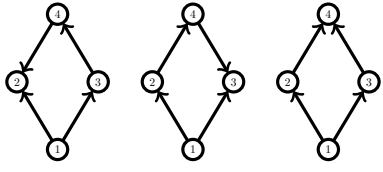
\begin{figure}[ht] \label{fig:o_j_1} \begin{center}

\begin{tikzpicture}  [scale = .18, very thick = 10mm]

 \node (n4) at (4,13)  [Cgray] {1};

  \node (n1) at (4,23) [Cgray] {4};
  \node (n2) at (1,18)  [Cgray] {2};
  \node (n3) at (7,18)  [Cgray] {3};
  \foreach \from/\to in {n2/n1,n2/n4}
    \draw[<-] (\from) -- (\to);

 \foreach \from/\to in {n4/n3,n3/n1}
    \draw[->] (\from) -- (\to);

 \node (n4) at (14,13)  [Cgray] {1};
  \node (n1) at (14,23) [Cgray] {4};

  \node (n2) at (11,18)  [Cgray] {2};
  \node (n3) at (17,18)  [Cgray] {3};

  \foreach \from/\to in {n1/n2,n3/n1}
    \draw[<-] (\from) -- (\to);

 \foreach \from/\to in {n4/n3,n4/n2}
    \draw[->] (\from) -- (\to);

      \node (n4) at (24,13)  [Cgray] {1};
  \node (n1) at (24,23) [Cgray] {4};

  \node (n2) at (21,18)  [Cgray] {2};
  \node (n3) at (27,18)  [Cgray] {3};
  \foreach \from/\to in {n1/n2,n1/n3, n3/n4,n2/n4}
    \draw[<-] (\from) -- (\to);
\end{tikzpicture}
\caption{Acyclic orientations corresponding to $4$-partitions of $C_4$}
\end{center}

\end{figure}

\end{Example}
%%%%%%%%%%%%%%%%%%%%%%%%%%%%%%%%%%%%%%%%%%%%%%%%%%%%%%%%%%%%%%%%%%%%%%%%%%%%%%%%%%%%%%%%%%%%%%%%%%%%%%%%%%%%%%%

\subsection{Main properties of $\Sf_k(G, q)$} \label{sec:technical}

This section, while elementary, is the most technical part of the paper. The reader is invited to draw graphs and Venn diagrams to follow the proofs. The main ingredients needed for our main theorems are Definition~\ref{def:U12}, Proposition~\ref{pro:well-def(a)}, Proposition~\ref{prop:minimal1-1}, Corollary~\ref{cor:injectivity}, Lemma~\ref{lem:KU1U2}, and Proposition~\ref{prop:max}. Other results are used to establish these main ingredients.

\begin{Lemma}\label{lem:S_k}
Let
\[
\Uc:\ \, U_1 \subsetneq U_2 \subsetneq \cdots \subsetneq V(G)
\]
be an element of $\Sf_k(G, q)$. Let
\[
\Wc:\ \, W_1 \subsetneq W_2 \subsetneq  \cdots \subsetneq V(G)
\]
be an element of $\Ff_{k'}(G, q)$.

\begin{itemize}

\item[(a)] If $k'=k-1$ and $W_i=U_{i+1}$ for $1 \leq i \leq k-1$ then $\Wc \in \Sf_{k-1}(G, q)$.

\medskip

\item[(b)] If $k'=k-1$ and $W_1= U_1 \cup (U_3 \backslash U_2)$ and $W_i=U_{i+1}$ for $2 \leq i \leq k-1$ then $\Wc \in \Sf_{k-1}(G, q)$.

\medskip

\item[(c)] If $k'=k$ and $W_1\subseteq U_1$ and $W_i=U_i$ for $2 \leq i \leq k$ then $\Wc \in \Sf_k(G, q)$.

\end{itemize}

\end{Lemma}

\begin{proof}
The strategy of the proof is similar for all three parts. Namely in each case, for the sake of contradiction, we assume that
\[
\Wc':\ \, W'_1 \subsetneq W'_2 \subsetneq \cdots \subsetneq V(G)
\]
is another connected flag equivalent to $\Wc$ which precedes $\Wc$ in the total ordering. We will then find a connected flag $\Uc'$ equivalent to $\Uc$ such that $\Uc' \prec_k \Uc$.

\medskip

We first note that in all cases, since $G(\Wc)$ and $G(\Wc')$ coincide and $q \in W_1,  W'_1$, we must have $W_1=W'_1$. Let $\ell \geq 2$ be such that $W'_{\ell} \prec W_{\ell}$ and $W'_{i} = W_{i}$ for $\ell < i $.

\medskip

(a)  The ordered collection $(W'_{j}\backslash W'_{j-1})_{j=2}^{k-1}$ is a permutation of the ordered collection $(W_{j}\backslash W_{j-1})_{j=2}^{k-1}=(U_{j}\backslash U_{j-1})_{j=3}^{k}$. Also $W'_1\backslash U_1=W_1\backslash U_1=U_2\backslash U_1$. It follows that
\[
\Uc':\ \, U_1 \subsetneq W'_1\subseteq W'_2\subseteq\cdots \subsetneq W'_{k-2} \subsetneq W'_{k-1}=V(G)
\]
is a connected $k$-flag equivalent to $\Uc$. But then $W'_{\ell} \prec W_{\ell}$ and $W'_{i} = W_{i}$ for $\ell < i \leq k-1$ implies $\Uc'\prec_k \Uc$, a contradiction.

\medskip

(b) The ordered collection $(W'_{j}\backslash W'_{j-1})_{j=2}^{k-1}$ is a permutation of the ordered collection $(W_{j}\backslash W_{j-1})_{j=2}^{k-1}=(U_{j}\backslash U_{j-1})_{ j=2,4,5,\ldots,k}$; here we used $W_2\backslash W_1= U_3\backslash(U_1 \cup (U_3 \backslash U_2))=U_2\backslash U_1$. Also $W'_1\backslash U_1=W_1 \backslash U_1=(U_1 \cup (U_3 \backslash U_2))\backslash U_1=U_3\backslash U_2$. It follows that
\[
\Uc':\ \, U_1 \subsetneq W'_1\subseteq W'_2\subseteq\cdots \subsetneq W'_{k-2} \subsetneq W'_{k-1}=V(G) \ .
\]
is a connected $k$-flag equivalent to $\Uc$. But then $W'_{\ell} \prec W_{\ell}$ and $W'_{i} = W_{i}$ for $\ell < i \leq k-1$ implies $\Uc'\prec_k \Uc$, a contradiction.

\medskip

(c) The ordered collection $(W'_{j}\backslash W'_{j-1})_{j=2}^k$ is a permutation of the ordered collection $(W_{j}\backslash W_{j-1})_{j=2}^k$.
Therefore $W_2\backslash W_1=W'_t\backslash W'_{t-1}$ for some $t \geq 2$. Note that if $t \geq 3$ then $W'_{t-1}\neq W_{t-1}$; for $t \geq 3$ we have $W_2 \subseteq W_{t-1}$ but $W_2\not\subset W'_{t-1}$. In particular we deduce that $\max(2,t-1) \leq \ell$.

\medskip

Now let $A=U_1\backslash W_1$.

$\bullet$  For $0 \leq i \leq t-1$ the sets $A$ and $W'_i$ are disjoint. This is because $A \subseteq W_2 \backslash W_1=W'_t \backslash W'_{t-1}$, so $A \cap W'_{t-1}=\emptyset$.

$\bullet$  For $2 \leq i \leq t-1$ there is no edge between $A$ and $W'_i\backslash W'_{i-1}$. This is because if there is an edge connecting $u \in W'_i\backslash W'_{i-1}$ and $v \in A$ then there is an oriented edge from $u \in W'_i\backslash W'_{i-1}$ to $v \in W'_t\backslash W'_{t-1}$ in $G(\Wc')$. But $G(\Wc')=G(\Wc)$ and the oriented edge from $u$ to $v$ cannot appear in $G(\Wc)$ because $v \in W_2\backslash W_1$ and $u \not \in W_1$.

 We now define
\[
\Uc':\ \, U'_1 \subsetneq U'_2 \subsetneq \cdots \subsetneq V(G)
\]
by letting
\[
U'_i=
\begin{cases}
W'_i\cup A, &\text{if $1 \leq i \leq t-1$;}\\
W'_i, &\text{if $t \leq i \leq k$.}
\end{cases}
\]

Note that

$\bullet$  $U'_1=W'_1\cup A=U_1$. This is because $A=U_1\backslash W_1$ and $W_1=W'_1$.

$\bullet$ All subgraphs $G[U'_i]$ are connected: for $1 \leq i \leq t-1$ since $W'_1\cup A=U_1$ is connected, we know that $A$ is connected to $W'_i$ at least via $W'_1 \subset W'_i$.

Moreover

$\bullet$  For $2 \leq i \leq t-1$ we have $U'_i \backslash U'_{i-1}=(W'_i\cup A) \backslash (W'_{i-1}\cup A)=W'_{i}\backslash W'_{i-1}$. This follows from the fact that $A$ is disjoint from $W'_{i-1}$ and $W'_{i}$.

$\bullet$  $U'_t \backslash U'_{t-1}=W'_t\backslash (W'_{t-1}\cup A)=(W_2\backslash W_1)\backslash A=(U_2\backslash W_1)\backslash (U_1\backslash W_1)=U_2\backslash U_1$.

$\bullet$  For $t+1 \leq i \leq k$ we have $U'_i \backslash U'_{i-1}=W'_i\backslash W'_{i-1}$.

Recall that the ordered collection $(W'_{j}\backslash W'_{j-1})_{j}$ is a permutation of the ordered collection $(W_{j}\backslash W_{j-1})_{j}$ and that $W_{j}\backslash W_{j-1}=U_{j} \backslash U_{j-1}$ for $3 \leq j \leq k$.

\medskip

It follows that $\Uc' \in \Ff_k(G,q)$ and that $(U'_{j}\backslash U'_{j-1})_{1}^k$ is a permutation of $(U_{j}\backslash U_{j-1})_{1}^k$.

\medskip

Moreover $\Uc'$ is equivalent to $\Uc$. To see this first observe that the only difference between $G(\Uc)$ and $G(\Wc)$ is that we orient the edges from $U_{1}$ to $U_{2} \backslash U_{1}$ in $G(\Uc)$ and we orient the edges from $W_{1}$ to $U_{2} \backslash W_{1}$ in $G(\Wc)$ (other oriented edges are identical). Similarly, since there is no edge between $A$ and $W'_i\backslash W'_{i-1}$ for $2 \leq i \leq t-1$, the only difference between $G(\Uc')$ and $G(\Wc')$ is that we orient the edges from $U'_{1}=U_1$ to $U'_{t} \backslash U'_{t-1}=U_2\backslash U_1$ in $G(\Uc')$ and we orient the edges from $W'_{1}=W_1$ to $W'_{t} \backslash W'_{t-1}=U_2 \backslash W_1$ in $G(\Wc)$ (other oriented edges are identical). Since $G(\Wc)$ and $G(\Wc')$ coincide it follows that $G(\Uc)$ and $G(\Uc')$ coincide.

\medskip

Finally we show that $\Uc'\prec_k\Uc$. Recall that $\ell \geq 2$ and $\ell \geq t-1$.

$\bullet$  If $\ell \geq t$, then $U'_{\ell}=W'_{\ell} \prec W_{\ell}=U_{\ell}$ and $U'_{i}=W'_{i} = W_{i}=U_{i}$ for $\ell <i \leq k$.

$\bullet$  If $\ell=t-1$, then $U'_{\ell}=W'_{\ell}\cup A\prec W'_{\ell}\prec W_{\ell}=U_{\ell}$ and $U'_{i}=W'_{i} = W_{i}=U_{i}$ for $\ell <i \leq k$.

Therefore in any case $\Uc'\prec_k\Uc$, a contradiction.
\end{proof}

\medskip

\begin{Proposition}\label{prop:well-def(b)}
Let
\[
 U_1 \subsetneq U_2 \subsetneq \cdots \subsetneq V(G)
\]
\[
 V_1 \subsetneq V_2 \subsetneq \cdots \subsetneq V(G)
\]
be two elements of $\Sf_{k}(G, q)$ such that $U_i=V_i$ for $2 \leq i \leq k$. Let
$A=U_2 \backslash  U_1$, $ B=V_2 \backslash V_1$, $C=U_1\cap V_1$, and assume at least one of the following holds:
\begin{itemize}
\item[(i)] $G[C]$ is not connected.
\item[(ii)] $G[A \cup B]$ is connected.
\end{itemize}

Then
\[
 W_1 \subsetneq W_2 \subsetneq \cdots \subsetneq V(G)
\]
is also an element of $\Sf_{k}(G, q)$, where $W_i=U_i$ for $2 \leq i \leq k$ and $W_1$ is (the vertex set of) the connected component of $G[C]$ that contains $q$.

\end{Proposition}

\begin{proof}
It is enough to check that $G[W_2\backslash W_1]$ is connected. Then the assertion follows by Lemma~\ref{lem:S_k} (c),
since $W_1\subset U_1$.

\medskip

Write
\begin{equation}\label{eq:W21}
W_2\backslash W_1=A\cup B\cup (C\backslash W_1)\ ,
\end{equation}
and write $U_1$ and $V_1$ as the disjoint unions:
\begin{equation}\label{eq:disjoint}
U_1=(U_1\backslash V_1) \cup (C\backslash W_1) \cup W_1 \quad \text{and} \quad  V_1=(V_1\backslash U_1)\cup (C\backslash W_1) \cup W_1 \ .
\end{equation}

(i) Assume that $G[C]$ is not connected, that is, $C\backslash W_1 \ne \emptyset$. Since $G[U_1]$ is connected and there are no edges between $C\backslash W_1$ and $W_1$, it follows from \eqref{eq:disjoint} that there must be 
an edge between each connected component $C\backslash W_1$ and $U_1\backslash V_1$.
As $U_1\backslash V_1 \subseteq B$ we conclude that every connected component of $C \backslash W_1$ has an edge to $B$.
Similarly, every connected component of $C \backslash W_1$ has an edge to $A$.
Since $G[A]$ and $G[B]$ are both connected it follows from \eqref{eq:W21} that $G[W_2\backslash W_1]$ is connected.

(ii) Assume that $G[A \cup B]$ is connected. We may assume $C\backslash W_1 = \emptyset$, since otherwise the result follows from (i) above. But then \eqref{eq:W21} becomes $W_2\backslash W_1=A\cup B$ and therefore $G[W_2\backslash W_1]$ is connected.
\end{proof}

\medskip

Given an element in $\Sf_k(G, q)$ there is a canonical way to obtain two related elements in $\Sf_{k-1}(G, q)$. To state this result we first need a definition.
\begin{Definition} \label{def:U12}
Given $\Uc \in \Ff_k(G, q)$, the elements $\Uc^{(1)}, \Uc^{(2)}\in \Ff_{k-1}(G, q)$ are obtained from $\Uc$ by removing the first and second elements in the following appropriate sense. Let

\[
\Uc:\ \, U_1 \subsetneq U_2 \subsetneq \cdots \subsetneq V(G) \ .
\]

\begin{itemize}
\item[(a)] $\Uc^{(1)}$ will denote
\[
U_2 \subsetneq U_3 \subsetneq U_4 \subsetneq \cdots \subsetneq V(G) \ .
\]
\item[(b)] $\Uc^{(2)}$ will denote
\[
\begin{cases}
 U_1 \subsetneq U_3 \subsetneq U_4 \subsetneq \cdots \subsetneq V(G), &\text{if $G[U_3 \backslash U_1]$ is connected;}\\
\text{or} \\
(U_1 \cup (U_3 \backslash U_2)) \subsetneq U_3 \subsetneq U_4 \subsetneq \cdots \subsetneq V(G), & \text{if $G[U_3 \backslash U_1]$  is not connected.}
\end{cases}
\]
\end{itemize}
\end{Definition}

\medskip

{We remark that $\Uc^{(1)}$ and $\Uc^{(2)}$ are essential in the expression of our minimal free resolutions (see \eqref{eq:map}).

\begin{Remark}\label{rmk:1vs2cases}
For part (b) in Definition~\ref{def:U12} we note that $U_3 \backslash U_1=(U_3 \backslash U_2) \cup (U_2 \backslash U_1)$ and by assumption $G[U_3 \backslash U_2]$ and $G[U_2 \backslash U_1]$ are both connected and nonempty. So $G[U_3 \backslash U_1]$ is connected if and only if there are some edges connecting $U_3 \backslash U_2$ to $U_2 \backslash U_1$. Moreover, if there are no edges connecting $U_3 \backslash U_2$ to $U_2 \backslash U_1$ then there must be some edges connecting $U_3 \backslash U_2$ to $U_1$. Since $G[U_3]$ is connected and $U_3 =(U_3 \backslash U_2) \cup (U_2 \backslash U_1) \cup U_1$ we get  $G[U_1 \cup (U_3 \backslash U_2)]$ is connected.
\end{Remark}

\medskip

\begin{Proposition}\label{pro:well-def(a)}
If $\Uc \in \Sf_k(G, q)$ then
\begin{itemize}
\item[(a)] $\Uc^{(1)} \in \Sf_{k-1}(G, q)$.
\item[(b)] $\Uc^{(2)} \in \Sf_{k-1}(G, q)$.
\item[(c)] $\Uc^{(1)} \prec_{k-1} \Uc^{(2)}$.
\end{itemize}
\end{Proposition}

\begin{proof}
Let
\[
\Uc:\ \, U_1 \subsetneq U_2 \subsetneq \cdots \subsetneq V(G) \ .
\]

Part (a) is exactly Lemma~\ref{lem:S_k} (a). For part (b), if $G[U_3 \backslash U_1]$ is not connected then the result follows from Lemma~\ref{lem:S_k} (b). If $G[U_3 \backslash U_1]$ is connected the result follows from part (a) and Lemma~\ref{lem:S_k} (c).

For part (c) we first note that if  $G[U_3 \backslash U_1]$ is connected $\Uc^{(1)} \prec_{k-1} \Uc^{(2)}$ follows directly from definitions. If  $G[U_3 \backslash U_1]$ is not connected assume that $\Uc^{(2)}\prec_{k-1} \Uc^{(1)}$. Then one can easily see that
\[
\Uc':\ \,  U_1 \subsetneq U_1 \cup (U_3\backslash U_2) \subsetneq U_3 \subsetneq U_4 \subsetneq \cdots \subsetneq V(G)
\]
is a connected $k$-flag equivalent to $\Uc$ with $\Uc'\prec_{k}\Uc$ which is a contradiction.
\end{proof}

\medskip

There is a nice converse to Proposition~\ref{pro:well-def(a)} which is our next result.

%%%%%%%%%%%%%%%%%%%%%%%%%%%%%%%%%%%%%%%%%%%%%%%%%%%%%%%%%%%%%%%%%%%%%%%%%%%%%%%%%%%%%%%%%%%%%%%%%%%%%%%%%%%%%%%

\begin{Proposition}
\label{prop:well-def(a)converse}
Assume that $\Uc \in \Ff_k(G, q)$ and the following three conditions hold:
\begin{itemize}
\item[(i)] $\Uc^{(1)} \in \Sf_{k-1}(G, q)$.
\item[(ii)] $\Uc^{(2)} \in \Sf_{k-1}(G, q)$.
\item[(iii)] $\Uc^{(1)} \prec_{k-1} \Uc^{(2)}$.
\end{itemize}
 Then $\Uc \in \Sf_{k}(G, q)$.
\end{Proposition}

\begin{proof}
Let
\[
\Uc:\ \, U_1 \subsetneq U_2 \subsetneq \cdots \subsetneq V(G) \ .
\]
To simplify the notation we let $A=U_2\backslash U_1$ and $B=U_3\backslash U_2$. For the sake of contradiction, we assume that
\[
\Uc':\ \, U'_1 \subsetneq U'_2 \subsetneq \cdots \subsetneq V(G)
\]
is another connected flag equivalent to $\Uc$ which precedes $\Uc$ in the total ordering $\prec_k$. We will then find a connected flag equivalent to $\Uc^{(j)}$ (for $j=1$ or $2$) preceding it in the total ordering $\prec_{k-1}$.

\medskip

We note that since $G(\Uc)$ and $G(\Uc')$ coincide and $q \in U_1$ and $q \in U'_1$, we must have $U_1=U'_1$.
Let $\ell \geq 2$ be such that $U'_{\ell} \prec U_{\ell}$ and $U'_{i} = U_{i}$ for $i>\ell$.

\medskip

We first show that $\ell>2$. If $\ell=2$, then $U'_1=U_1$, $U'_2=U_1\cup (U_3\backslash U_2)$ and $U'_i=U_i$ for all $i>2$.
If there is at least one edge between $A$ and $B$, then it is oriented from $B$ to $A$ in $G(\Uc')$. But that edge must be oriented from $A$ to $B$ in $G(\Uc)$ which contradicts $G(\Uc)=G(\Uc')$.
If there exists no edge between $A$ and $B$, then
\[
\Uc^{(2)}:\ \, (U_1 \cup (U_3 \backslash U_2)) \subsetneq U_3 \subsetneq U_4 \subsetneq \cdots \subsetneq V(G) \ .
\]
Now our assumption that $\Uc^{(1)}\prec_{k-1} \Uc^{(2)}$ implies that $U_2\prec U_1\cup (U_3\backslash U_2)$. Thus $\Uc\prec_k \Uc'$ which is a contradiction. Thus we must have $\ell>2$.

\medskip

The ordered collection $(U'_{j}\backslash U'_{j-1})_{j=2}^k$ is a permutation of the ordered collection $(U_{j}\backslash U_{j-1})_{j=2}^k$.
Then $A=U'_t\backslash U'_{t-1}$ and $B=U'_{s}\backslash U'_{s-1}$ for some $t,s \geq 2$. Note that if $t \geq 3$ then $U'_{t-1}\neq U_{t-1}$; for $t \geq 3$ we have $U_2 \subseteq U_{t-1}$ but $U_2\not\subset U'_{t-1}$ (since $A\not\subset U'_{t-1}$). On the other hand, if $s \geq 4$ then $U'_{s-1}\neq U_{s-1}$; for $s \geq 4$ we have
$U_3 \subseteq U_{s-1}$ but $U_3\not\subset U'_{s-1}$ (since $B\not\subset U'_{s-1}$). In particular we deduce that $\ell \geq \max(s-1,t-1)$.

\medskip

We first show that $\Uc^{(1)} \in \Sf_{k-1}(G, q)$ implies $t>s$.
Since all edges connecting $A$ and $V(G)\backslash U_2$ are oriented from $A$ to $V(G)\backslash U_2$ in $G(\Uc)$, and $G(\Uc)=G(\Uc')$ we conclude that there is no edge between $A$ and $U'_{t-1}\backslash U'_1$ if $t>2$. Now we define
\[
\Wc:\ \, W_1 \subsetneq W_2 \subsetneq \cdots \subsetneq W_{k-1}=V(G)
\]
by letting
\[
W_i=\begin{cases}
U'_i\cup A, &\text{if $1 \leq i \leq t-1$;}\\
U'_{i+1}, &\text{if $t \leq i \leq k-1$.}
\end{cases}
\]

Note that

$\bullet$  $W_1=U'_1\cup A=U_2$. This is because $A=U_2\backslash U_1$ and $U_1=U'_1$.

$\bullet$  The subgraphs $G[W_i]$ are all connected; for $1 \leq i \leq t-1$ since $U'_1\cup A=U_2$ is connected, we know that $A$ is connected to $U'_i$ at least via $U'_1 \subset U'_i$.

Moreover

$\bullet$  For $2 \leq i \leq t-1$ we have $W_i \backslash W_{i-1}=(U'_i\cup A) \backslash (U'_{i-1}\cup A)=U'_{i}\backslash U'_{i-1}$. This follows from the fact that $A$ is disjoint from $U'_{i-1}$ and $U'_{i}$.

$\bullet$  $W_t \backslash W_{t-1}=U'_{t+1}\backslash (U'_{t-1}\cup A)=U'_{t+1}\backslash U'_{t}$.

$\bullet$  For $t+1 \leq i \leq k-1$ we have $W_i\backslash W_{i-1}=U'_{i+1} \backslash U'_{i}$.

This implies that the ordered collection $(W_i \backslash W_{i-1})_{1}^{k-1}$ is a permutation of the ordered collection $U_2\cup (U'_{j}\backslash U'_{j-1})_{j\in \{2,\cdots, k\} \backslash \{t\}}$ which is a permutation of ordered collection $U_2\cup (U_{j}\backslash U_{j-1})_{j=3}^{k}$ of $\Uc^{(1)}$. We show that $\Wc$ is equivalent to $\Uc^{(1)}$:
first note that the only difference between $G(\Uc)$ and $G(\Uc^{(1)})$ is that we orient the edges from $U_{1}$ to $U_{2} \backslash U_{1}$ in $G(\Uc)$ but we keep these edges unoriented in $G(\Uc^{(1)})$ (other oriented edges are identical).
If $t=2$ then $U'_2=U_2$ and all oriented edges from $A$ to $V(G)\backslash U_2$ in $G(\Uc^{(1)})$ are
also in $G(\Wc)$. Other edges are identical, since $G(\Uc)$ and $G(\Uc')$ coincide.
If $t>2$ then there is no edge between $A$ and $U'_c\backslash U'_1$ for $2\leq c\leq t-1$. Therefore the only difference between $G(\Uc')$ and $G(\Wc)$ is that we orient the edges from $U'_{1}=U_1$ to $U'_{t} \backslash U'_{t-1}=U_2\backslash U_1$ in $G(\Uc')$ and we keep these edges unoriented in $G(\Wc)$. Other edges are identical, since $G(\Uc)$ and $G(\Uc')$ coincide. Thus it follows that $G(\Uc^{(1)})$ and $G(\Wc)$ coincide.

\medskip

Note that the $i$th element in $\Uc^{(1)}$ is $U_{i+1}$ for all $i$. If $\ell>t-1$, then $W_{\ell-1}=U'_\ell\prec U_\ell$ and $W_i=U_{i+1}$ for $i>\ell$. Thus $\Wc\prec_{k-1} \Uc^{(1)}$ which is a contradiction by our assumption that $\Uc^{(1)}$ belongs to $\Sf_{k-1}(G,q)$. Therefore we have $\ell=t-1$. This also implies that $t>s$.

\medskip

Now we consider two cases:

$\bullet$  If $G[U_3 \backslash U_1]$ is connected we have
\[
\Uc^{(2)}:\ \, U_1 \subsetneq U_3 \subsetneq U_4 \subsetneq \cdots \subsetneq V(G) \ .
\]
In this case $\Uc^{(2)} \in \Sf_{k-1}(G,q)$ implies that $U_3\backslash U_1=(U_3\backslash U_2)\cup (U_2\backslash U_1)=A\cup B$ is connected. 
So there is at least one edge between $A$ and $B$. This edge must have opposite orientations in  $G(\Uc)$ and $G(\Uc')$, which is a contradiction because $G(\Uc)=G(\Uc')$.

\medskip

$\bullet$  If $G[U_3 \backslash U_1]$ is not connected we have
\[
\Uc^{(2)}:\ \, (U_1 \cup (U_3 \backslash U_2)) \subsetneq U_3 \subsetneq U_4 \subsetneq \cdots \subsetneq V(G) \ .
\]
In this case we need to do more work.
We define
\[
\Wc':\ \, W'_1 \subsetneq W'_2 \subsetneq \cdots \subsetneq W'_{k-1}=V(G)
\]
by letting
\[
W'_i=\begin{cases}
U'_i\cup B, &\text{if $1 \leq i \leq s-1$;}\\
U'_{i+1}, &\text{if $s \leq i \leq k-1$.}
\end{cases}
\]

Note that
\begin{itemize}
\item $W'_1=U'_1\cup B=U_1\cup (U_3\backslash U_2)$. This is because $B=U_3\backslash U_2$ and $U_1=U'_1$.

\item The subgraphs $G[W'_i]$ are all connected; for $1 \leq i \leq s-1$ since $U'_1\cup B$ is connected (by our assumption that $\Uc^{(2)}$ belongs to $\Sf_{k-1}(G, q)$), we know that $B$ is connected to $U'_i$ at least via $U'_1 \subset U'_i$.
\end{itemize}

Moreover

\begin{itemize}
\item For $2 \leq i \leq s-1$ we have $W'_i \backslash W'_{i-1}=(U'_i\cup B) \backslash (U'_{i-1}\cup B)=U'_{i}\backslash U'_{i-1}$. This follows from the fact that $B$ is disjoint from $U'_{i-1}$ and $U'_{i}$.

\item $W'_s \backslash W'_{s-1}=U'_{s+1}\backslash (U'_{s-1}\cup B)=U'_{s+1}\backslash U'_{s}$.

\item For $s+1 \leq i \leq k-1$ we have $W'_i\backslash W'_{i-1}=U'_{i+1} \backslash U'_{i}$.

\end{itemize}
This implies that the ordered collection $(W'_i \backslash W'_{i-1})_{1}^{k-1}$ is a permutation of the ordered collection
$(U_1\cup B)\cup (U'_{j}\backslash U'_{j-1})_{j\in \{2,\cdots, k\} \backslash \{s\}} $ which is a permutation of the ordered collection for $\Uc^{(2)}$.
Now we check that $\Wc'$ is equivalent to $\Uc^{(2)}$:
the only difference between $G(\Uc)$ and $G(\Uc^{(2)})$ is that we orient the edges from $U_{1}$ to $U_{3}\backslash U_{2}$ in $G(\Uc)$ but we keep these edges unoriented in $G(\Uc^{(2)})$ (other oriented edges are identical).
Similarly, since there is no edge between $B$ and $U'_c\backslash U'_2$ for $2\leq c\leq s-1$, the only difference between $G(\Uc')$ and $G(\Wc')$ is that we orient edges from $U'_{1}=U_1$ to $U'_{s} \backslash U'_{s-1}=U_3\backslash U_2$ in $G(\Uc')$ and we keep these edges unoriented in $G(\Wc')$ (other oriented edges are identical). Since $G(\Uc)$ and $G(\Uc')$ coincide it follows that $G(\Uc^{(2)})$ and $G(\Wc')$ coincide.

\medskip
Now our assumption that $U'_\ell\prec U_\ell$ and $\ell>s-1$ implies that $U'_\ell=W'_{\ell-1}\prec U_\ell$ and $W'_i=U_{i+1}$ for $i>\ell-1$. Note that the $i$th element in $\Uc^{(2)}$ is $U_{i+1}$ for $i>1$. Thus $\Wc'\prec_{k-1}\Uc^{(2)}$ which is a contradiction.
\end{proof}

\begin{Proposition}\label{prop:minimal1-1}
Let $X_1, X_2, Y_1, Y_2$ be four nonempty subsets of $V(G)$ such that:
\begin{itemize}
\item[(1)] $G[X_1]$, $G[X_2]$, $G[Y_1]$, and $G[Y_2]$ are connected.
\item[(2)] $X_1 \cap X_2 =\emptyset$ and $Y_1 \cap Y_2=\emptyset$,
\item[(3)] $X_1 \cup X_2 =Y_1 \cup Y_2$,
\item[(4)] $X_2 \cap Y_2 \ne \emptyset$.
\end{itemize}
Then $D(X_1,X_2)\leq D(Y_1,Y_2)$ implies $X_1=Y_1$ and $X_2=Y_2$.
\end{Proposition}

\begin{proof}
First we show that we must have $Y_2 \subseteq X_2$. Assume that $Y_2 \not \subseteq X_2$. Then for any $v \in X_1\backslash Y_1$ we have $D(X_1,X_2)(v) \leq D(Y_1,Y_2)(v)=0$ because $D(Y_1,Y_2)$ is supported on $Y_1$. Therefore there is no edge between $X_1 \cap Y_2$ (i.e. the subset  $X_1\backslash Y_1 \subseteq X_1$) and $X_2 \cap Y_2$ (a subset of $X_2$). Note that $Y_2=(X_1 \cap Y_2) \cup (X_2 \cap Y_2)$. Since $X_2 \cap Y_2 \ne \emptyset$, the assumption $X_1\backslash Y_1\ne \emptyset$ results in $G[Y_2]$ being disconnected which is a contradiction.

\medskip

So we may assume $Y_2 \subseteq X_2$. Then for any $v \in X_1 \subseteq Y_1$ the set of all edges connecting $v$ to a vertex in $X_2$ contains the set of all edges connecting $v$ to a vertex in $Y_2$. Therefore we have $D(Y_1,Y_2)(v) \leq D(X_1,X_2)(v)$. Comparing this with the inequality in the assumption we get $D(Y_1,Y_2)(v) = D(X_1,X_2)(v)$ for all $v \in X_1$.
    This means that there cannot be any edge connecting $X_1$ to $X_2 \backslash Y_2$. Since $Y_1=X_1 \cup (X_2 \backslash Y_2)$, if $X_2 \backslash Y_2 \ne \emptyset$ then $G[Y_1]$ would be disconnected. So we must have $Y_2 = X_2$ and so $Y_1=X_1$.
\end{proof}

%%%%%%%%%%%%%%%%%%%%%%%%%%%%%%%%%%%%%%%%%%%%%%%%%%%%%%%%%%%%%%%%%%%%%%%%%%%%%%%%%%%%%%%%%%%%%%%%%%%%%%%%%%%%%%%
\begin{Corollary}\label{cor:injectivity}
Let
\[
\Uc:\ \, U_1 \subsetneq U_2 \subsetneq \cdots \subsetneq V(G)
\]
\[
\Vc:\ \, V_1 \subsetneq V_2 \subsetneq \cdots \subsetneq V(G)
\]
be two elements of $\Sf_{k}(G, q)$.  If for some $2\leq i \leq k$ \[ U_i=V_i \quad \text{ and } \quad D(U_i\backslash U_{i-1},U_{i-1}) \leq D(V_i\backslash V_{i-1},V_{i-1})\] then $U_{i-1}=V_{i-1}$.
\end{Corollary}

\begin{proof}

Let $X_1=U_i\backslash U_{i-1}$, $X_2=U_{i-1}$, $Y_1=V_i\backslash V_{i-1}$, and $Y_2=V_{i-1}$ in  Proposition~\ref{prop:minimal1-1}. Note that $X_2 \cap Y_2 \ne \emptyset$ because $q\in U_1 \subseteq U_{i-1}$ and similarly $q\in V_1 \subseteq V_{i-1}$.
\end{proof}

%%%%%%%%%%%%%%%%%%%%%%%%%%%%%%%%%%%%%%%%%%%%%%%%%%%%%%%%%%%%%%%%%%%%%%%%%%%%%%%
\begin{Definition}\label{def:KVW}
Write $\Wc, \Vc \in \Sf_{k}(G, q)$ as
\[
\Wc:\ \, W_1 \subsetneq W_2 \subsetneq \cdots \subsetneq W_k=V(G) \ ,
\]
\[
\Vc:\ \, V_1 \subsetneq V_2 \subsetneq \cdots \subsetneq V_k=V(G) \ .
\]

Assume that $W_i=V_i$ for $i\geq 2$. We define an effective divisor

\[
\K(\Wc,\Vc):=\max(D(W_2\backslash W_1,W_1),D(V_2\backslash V_1,V_1)) \ ,
\]
where $\max$ denotes the entry-wise maximum.
\end{Definition}

We remark that the notion $\K(\Wc,\Vc)$ is essential in the study of our ideals and modules using Gr\"obner theory (see proofs of Theorem~\ref{thm:Cori} and Theorem~\ref{thm:GB}).

\medskip

The following lemma gives an alternate formula for computing $\K(\Wc,\Vc)$ which is sometimes more convenient.
\begin{Lemma}\label{lem:KVW}
For $\Wc, \Vc \in \Sf_{k}(G, q)$ as in Definition~\ref{def:KVW}, we have the following alternate formula:
\[
\begin{aligned}
\K(\Wc,\Vc)&=\max(D(W_2\backslash(W_1\cup V_1),W_1),D(W_2\backslash(W_1\cup V_1),V_1))\\
& \quad +D(V_1\backslash W_1,W_1)+D(W_1\backslash V_1,V_1) \ .
\end{aligned}
\]
\end{Lemma}

\begin{proof}
Let
\[
\K=\max(D(W_2\backslash W_1,W_1),D(V_2\backslash V_1,V_1)) \ ,
\]
\[
\K'=\max(D(W_2\backslash(W_1\cup V_1),W_1),D(W_2\backslash(W_1\cup V_1),V_1))+D(V_1\backslash W_1,W_1)+D(W_1\backslash V_1,V_1)\ .
\]
Note that $W_2$ is the disjoint union of sets $W_1 \cap V_1$, $W_1 \backslash V_1$, $V_1 \backslash W_1$, and $W_2 \backslash (W_1 \cup V_1)$.

$\bullet$  If $v\in W_1 \cap V_1$ then $\K(v)=\K'(v)=0$.

$\bullet$  If $v\in W_1 \backslash V_1$ then $\K(v)=\K'(v)=D(W_1\backslash V_1,V_1)(v)$.

$\bullet$  If $v\in V_1 \backslash W_1$ then $\K(v)=\K'(v)=D(V_1\backslash W_1,W_1)(v)$.

$\bullet$  If $v\in W_2 \backslash (W_1 \cup V_1)$ then \[\K(v)=\K'(v)=\max(D(W_2\backslash(W_1\cup V_1),W_1),D(W_2\backslash(W_1\cup V_1),V_1))(v) \ . \]
\end{proof}
%%%%%%%%%%%%%%%%%%%%%%%%%%%%%%%%%%%%%%%%%%%%%%%%%%%%%%%%%%%%%%%%%%%%%%%%%%%%%%%%%%%%%%%%%%%%%%%%%%%%%%%%%%%%%%%

\begin{Lemma}\label{lem:KU1U2}
For $\Uc \in \Sf_{k}(G, q)$ of the form
\[
\Uc:\ \, U_1 \subsetneq U_2 \subsetneq \cdots \subsetneq U_k=V(G)
\]
we have
\[
\K(\Uc ^{(1)} ,\Uc ^{(2)})=D(U_2 \backslash U_1, U_1)+D(U_3 \backslash U_2, U_2) \ .
\]
\end{Lemma}
\begin{proof}
From Definition~\ref{def:U12}, Remark~\ref{rmk:1vs2cases}, and Definition~\ref{def:KVW} we need to compute $\max(\alpha, \beta)$ where

\[
\alpha=D(U_3\backslash U_2, U_2) \quad \text{ and }\quad
\beta=\begin{cases}
D(U_3\backslash U_1, U_1), &\text{if $G[U_3 \backslash U_1]$ is connected;}\\
\text{or} \\
D(U_2\backslash U_1, U_1), & \text{if $G[U_3 \backslash U_1]$  is not connected.}
\end{cases}
\]
Since $D(U_3\backslash U_1, U_1)=D(U_3\backslash U_2, U_1)+D(U_2\backslash U_1, U_1)$ and 
$D(U_2\backslash U_1, U_1)\geq D(U_3 \backslash U_2, U_2) $, it follows that in either case
\[\max(\alpha, \beta)=D(U_2\backslash U_1, U_1) +D(U_3 \backslash U_2, U_2) \ .\]
\end{proof}

%%%%%%%%%%%%%%%%%%%%%%%%%%%%%%%%%%%%%%%%%%%%%%%%%%%%%%%%%%%%%%%%%%%%%%%%%%%%%%%%%%%%%%%%%%%%%%%%%%%%%%%%%%%%%%%

We end this section by the following result which uses (and generalizes) many results of this section. This result plays a crucial role in the proof of Theorem~\ref{thm:GB}.

\begin{Proposition}\label{prop:max}
Fix $\Wc \in \Sf_{k}(G, q)$ and define
\[
\mathfrak{N}_\Wc=\{\Vc\in\Sf_{k}(G, q): \, \Vc^{(1)} = \Wc^{(1)} \text{ and }\ \Wc\prec_{k}\Vc\}\ .
\]

For any $\Vc\in \mathfrak{N}_\Wc$ there exists a $\Wc' \in \mathfrak{N}_\Wc$ such that
\begin{itemize}
\item[(i)] $\K(\Wc,\Wc') \leq \K(\Wc,\Vc)$,
\item[(ii)] $\Uc^{(1)}=\Wc$ and $\Uc^{(2)}=\Wc'$ for some $\Uc \in \Sf_{k+1}(G, q)$.
\end{itemize}
\end{Proposition}

\begin{proof}
Fix $\Vc\in \mathfrak{N}_\Wc$. Consider the following subset of $\Div(G)$ containing $\K(\Wc,\Vc)$:
\[
Q= \{ \K(\Wc,\Vc'): \, \Vc' \in \mathfrak{N}_\Wc \quad \text{and} \quad \K(\Wc,\Vc') \leq \K(\Wc,\Vc)\}\ .
\]

 This is a nonempty finite set of effective divisors, so it has some minimal elements with respect to the partial ordering $\leq$ on $\Div(G)$. Choose the largest (with respect to the total ordering $\prec_k$) element $\Wc' \in \mathfrak{{N_\Wc}}$ such that $\K(\Wc,\Wc')$ is a minimal element of $Q$. Write
\[
\Wc:\ \,  W_1 \subsetneq W_2 \subsetneq \cdots \subsetneq W_k=V(G)
\]
\[
\Wc':\ \,  W'_1 \subsetneq W_2 \subsetneq \cdots \subsetneq W_k=V(G) \ ,
\]
and to simplify the notation let $A=W_2\backslash W_{1}$, $B=W_{2}\backslash W'_{1}$, and $C=W_1\cap W'_1$. Let $X$ be (the vertex set of) the connected component of $G[C]$ containing $q$, and define the (auxiliary) $k$-flag (i.e. increasing sequence of subsets with no connectivity assumption)
\[
\Xc:\ \, X \subsetneq W_{2}\subsetneq W_3\subsetneq \cdots \subsetneq W_{k}=V(G)\ .
\]

\medskip

{\bf Claim.} If $\Xc \in \Sf_{k}(G, q)$ then $W_1' \subsetneq W_1$.

Assume that $\Xc \in \Sf_{k}(G, q)$. Since $\Wc \prec_k \Wc'$ we have $W_1 \ne W_1'$ and therefore $X \subsetneq W_1$. It follows  that $\Wc \prec_k \Xc$ and $\Xc \in \mathfrak{{N_\Wc}}$. We also have $X \subseteq W_1'$ and $\Wc' \preceq_k \Xc$.

We will show that
\[
\K(\Wc,\Xc) \leq \K(\Wc,\Wc')\ .
\]

Once this is shown, the claim is proved; if $\K(\Wc,\Xc) < \K(\Wc,\Wc')$ then $\K(\Wc,\Wc')$ is not minimal which is a contradiction. If $\K(\Wc,\Xc) = \K(\Wc,\Wc')$ then $\Xc$ also realizes a minimal divisor and $\Wc' \preceq_k \Xc$. This contradicts the definition of $\Wc'$ unless $\Wc' = \Xc$, which means $X=W_1'$ and hence $W_1' \subsetneq W_1$.

By the formula in Lemma~\ref{lem:KVW} we obtain
\[
\begin{aligned}
\K(\Wc,\Xc)&=
\max(D(W_2\backslash(W_1\cup X),W_1),D(W_2\backslash(W_1\cup X),X))\\
& \quad +D(X\backslash W_1,W_1)+D(W_1\backslash X,X) \\
&= \max(D(W_2\backslash W_1,W_1),D(W_2\backslash W_1,X))+D(W_1\backslash X,X)\\
&= D(W_2\backslash W_1,W_1)+D(W_1\backslash X,X)\\
&= D(W_2\backslash (W_1\cup W'_1),W_1)+D(W'_1\backslash W_1,W_1)+D(W_1\backslash X,X)\ .
\end{aligned}
\]
Compare this with
\[
\begin{aligned}
\K(\Wc,\Wc')&=\max(D(W_2\backslash(W_1\cup W'_1),W_1),D(W_2\backslash(W_1\cup W'_1),W'_1))\\
& \quad +D(W'_1\backslash W_1,W_1)+D(W_1\backslash W'_1,W'_1) \ .
\end{aligned}
\]
Since there is no edge between $C \backslash X$ and $X$ we have
\[D(W_1\backslash X,X)= D(W_1\backslash C,X)=D(W_1\backslash W_1',X) \leq D(W_1\backslash W'_1,W'_1) \ .\]
 We get $\K(\Wc,\Xc)\leq \K(\Wc,\Wc')$, and the claim is proved.

\medskip

We now define
\[
\Uc:\ \, C \subsetneq W_1 \subsetneq W_2 \subsetneq \cdots \subsetneq W_k=V(G) \ .
\]
By definition $\Uc^{(1)}=\Wc$. We will show that $\Uc \in \Sf_{k+1}(G, q)$ and $\Uc^{(2)}=\Wc'$. Since $\Wc' \in \mathfrak{{N_\Wc}}$ implies $\Wc\prec_{k}\Wc'$, by Proposition~\ref{prop:well-def(a)converse} it suffices to prove $\Uc \in \Ff_{k+1}(G,q)$ and $\Uc^{(2)}=\Wc'$. Recall from Proposition~\ref{prop:well-def(b)} that if $G[C]$ is not connected
 then we have $\Xc \in \Sf_{k}(G, q)$. By {\bf Claim} above, we then must have $C=W_1 \cap W_1'=W_1'$ which is connected. Therefore, we assume that $G[C]$ is connected.
Therefore to show $\Uc \in \Ff_{k+1}(G,q)$ we only need to check that $G[W_1 \backslash C]$ is connected;
all other connectivities are guaranteed by the assumption that $\Wc \in \Sf_{k}(G,q)$.

We need to consider two cases:

\medskip

$\bullet$ If $G[A \cup B]$ is not connected: First we note that since $G[A]$ and $G[B]$ are both connected, if $A \cap B \ne \emptyset$ then $A$ would be connected to $B$ via $A \cap B$, contradicting the fact that $G[A \cup B]$ is not connected. So we must have $A \cap B = \emptyset$ or equivalently $W_2 = W_1 \cup W_1'$. Now
\[W_1 \backslash C=W_1 \backslash W_1'=W_2 \backslash W_1' \]
which is connected since $\Wc' \in \Sf_{k}(G, q)$.

Now $\Wc'=\Uc^{(2)}$ by noticing that
\[W'_1=(W_1\cap W'_1)\cup (W_2\backslash W_1)=U_1\cup (U_3\backslash U_2) \ . \]

\medskip

$\bullet$ If $G[A \cup B]$ is connected:
Proposition~\ref{prop:well-def(b)} implies that $\Xc \in \Sf_{k}(G, q)$. Now by the claim above, we then must have $C=W_1 \cap W_1'=W_1'$ and so
\[
\Uc:\ \, W_1' \subsetneq W_1 \subsetneq W_2 \subsetneq \cdots \subsetneq W_k=V(G)\ .
\]
Thus $\Uc^{(2)}=\Wc'$ by definition.

\medskip

If $G[W_1\backslash W'_1]$ is not connected, let $Y \ne \emptyset$ be the vertex set of one of the connected components of $G[W_1\backslash W'_1]$, and consider the flag
\[
\Yc:\ \, W'_1\cup Y\subsetneq W_{2}\subsetneq W_3\subsetneq \cdots \subsetneq W_{k}=V(G) \ .
\]

Then $\Yc \in \Ff_{k}(G,q)$ because
\begin{itemize}
\item $G[W'_1\cup Y]$ is connected: $W_1$ is connected and is the disjoint union $W'_1 \cup Y \cup W_1\backslash (W'_1\cup Y)$. Since there are no edges between $Y$ and $W_1\backslash (W'_1\cup Y)$, there must be some edges connecting $Y$ to $W'_1$.
\item $G[W_2\backslash (W'_1\cup Y)]$ is connected: $W_2\backslash W'_1$  is connected and is the disjoint union $(W_2\backslash W_1)\cup Y\cup(W_1\backslash (W'_1\cup Y))$. Since there are no edges between $Y$ and $W_1\backslash (W'_1\cup Y)$, there must be some edges connecting $W_2\backslash W_1$ to $W_1\backslash (W'_1\cup Y)$. So $(W_2\backslash W_1) \cup (W_1\backslash (W'_1\cup Y))=W_2\backslash (W'_1\cup Y)$ is connected.
\end{itemize}

Since $W'_1\cup Y \subseteq W_1$ by comparing $\Yc$ and $\Wc$ and using Lemma~\ref{lem:S_k}(c) we get $\Yc \in \Sf_{k}(G, q)$. We also note that $\Wc\prec_{k}\Yc$. Consequently $\Yc \in \mathfrak{{N_\Wc}}$.

However $\K(\Wc,\Yc) < \K(\Wc,\Wc')$ which contradicts the definition of $\Wc'$. To see this compare
\[
\begin{aligned}
\K(\Wc,\Xc)&
= D(W_2\backslash W_1,W_1)+D(W_1\backslash (W'_1\cup Y),W'_1\cup Y)\\
&= D(W_2\backslash (W_1\cup W'_1),W_1)\\
&\quad +D(W'_1\backslash W_1,W_1)+D(W_1\backslash (W'_1\cup Y),W'_1 \cup Y)
\end{aligned}
\]
with
\[
\begin{aligned}
\K(\Wc,\Wc')&=\max(D(W_2\backslash(W_1\cup W'_1),W_1),D(W_2\backslash(W_1\cup W'_1),W'_1))\\
& \quad +D(W'_1\backslash W_1,W_1)+D(W_1\backslash W'_1,W'_1) \ .
\end{aligned}
\]
Here we have
\[
\begin{aligned}
D(W_1\backslash (W'_1\cup Y),W'_1 \cup Y) &= D(W_1\backslash (W'_1\cup Y), W'_1)\\
&< D(W_1\backslash (W'_1\cup Y), W'_1) + D(Y, W'_1)\\
& = D(W_1\backslash W'_1,W'_1) \ .
\end{aligned}
\]

The first equality is because there are no edges between $W_1\backslash (W'_1\cup Y)$ and $Y$. The strict inequality is because we have shown above that there are edges connecting $Y$ to $W'_1$.

Therefore $G[W_1\backslash W'_1]$ is connected, and $\Uc \in \Ff_{k+1}(G,q)$ which is what we wanted.
\end{proof}

%%%%%%%%%%%%%%%%%%%%%%%%%%%%%%%%%%%%%%%%%%%%%%%%%%%%%%%%%%%%%%%%%%%%%%%%%%%%%%%%%%%%%%%%%%%%%%%%%%%%%%%%%%%%%%%

%%%%%%%%%%%%%%%%%%%%%%%%%%%%%%%%%%%%%%%%%%%%%%%%%%%%%%%%%%%%%%%%%%%%%%%%%%%%%%%%%%%%%%%%%%%%%%%%%%%%%%%%%%%%%%%
%%%%%%%%%%%%%%%%%%%%%%%%%%%%%%%%%%%%%%%%%%%%%%%%%%%%%%%%%%%%%%%%%%%%%%%%%%%%%%%

%%%%%%%%%%%%%%%%%%%%%%%%%%%%%%%%%%%%%%%%%%%%%%%%%%%%%%%%%%%%%%%%%%%%%%%%%%%%%%%%%%%%%%%%%%%%%%%%%%%%%%%%%%%%%%%
%%%%%%%%%%%%%%%%%%%%%%%%%%%%%%%%%%%%%%%%%%%%%%%%%%%%%%%%%%%%%%%%%%%%%%%%%%%%%%%

\section{Syzygies and free resolutions for $I_G$ and $\ini(I_G)$}
\label{sec:Syzygy}

%%%%%%%%%%%%%%%%%%%%%%%%%%%%%%%%%%%%%%%%%%%%%%%%%%%%%%%%%%%%%%%%%%%%%%%%%%%%%%%%%%%%%%%%%%%%%%%%%%%%%%%%%%%%%%%

Let $K$ be a field and let $R=K[\xb]$ be the polynomial ring in $n$ variables $\{x_v: v \in V(G)\}$. Recall from \S\ref{sec:grade} that $K[\xb]$ has a natural $\As$-grading, where $\As$ can be replaced by $\ZZ$, $\Div(G)$, or $\Pic(G)$. Recall that for $\As=\ZZ$ and $\As=\Pic(G)$ the ideal $I_G$ is graded.

Let the monomial ordering $<$ on $R$ be as in Definition~\ref{Def:MonomialOrder}. Recall that this ordering depends on the choice of the fixed vertex $q$. The following theorem is essentially in \cite[Theorem~14]{CoriRossinSalvy02}. Here we state and prove the theorem in a language that suggests a generalization.

\begin{Theorem}\label{thm:Cori}
Fix a pointed graph $(G,q)$ and let $\As=\ZZ$ or $\As=\Pic(G)$. A minimal $\As$-homogeneous Gr\"obner bases of $(I_G, <)$ is
\[
\Gb(G,q)=\{\xb^{D(U_2\backslash U_1, U_1)}-\xb^{D(U_1, U_2\backslash U_1)} : \, U_1 \subsetneq U_2=V(G) \text{ is in } \Sf_2(G, q) \} \ .
\]

Moreover
$\LM(\xb^{D(U_2\backslash U_1, U_1)}-\xb^{D(U_1, U_2\backslash U_1)})=\xb^{D(U_2\backslash U_1, U_1)} $.

\end{Theorem}

\begin{proof}
To simplify the notation for a subset $A \subseteq V(G)$ we use $\bar{A}=V(G) \backslash A$. Since $q \in U_1$ it follows from the definition of $<$ that $\xb^{D(U_1, \bar{U_1})} < \xb^{D(\bar{U_1}, U_1)}$.

We first prove
\[
\Gb'(G,q)=\{\xb^{D(U_2\backslash U_1, U_1)}-\xb^{D(U_1, U_2\backslash U_1)} : \, U_1 \subsetneq U_2=V(G) 
, q \in U_1 \}
\]
forms a Gr\"obner bases of $I_G$. We will call a sequence of subsets $U_1 \subsetneq U_2=V(G)$ with $q \in U_1$ a $2$-flag of $(G,q)$. Note that for a $2$-flag there is no connectivity assumption on $G[U_1]$ or on $G[U_2 \backslash U_1]$.

As usual, we use Buchberger's criterion. Let
$f=\xb^{D(\bar{U}, U)}-\xb^{D(U, \bar{U})}$ and $g=\xb^{D(\bar{V}, V)}-\xb^{D(V, \bar{V})}$ be two elements of $\Gb'(G,q)$. Define the effective divisor $D' \in \Div(G)$ by
\[D' = \max(D(\bar{U}, U), D(\bar{V}, V)) = \K(\Uc, \Vc) \ .\]

In the language of chip-firing games, $D'$ is the minimal divisor that allows one to ``fire'' either the set $\bar{U}$ or the set $\bar{V}$ and still have an effective divisor as outcome, that is,
\[
D'-\Delta(\chi_{\bar{U}}) \geq 0 \text{ and } D'-\Delta(\chi_{\bar{V}}) \geq 0 \ .
\]

Buchberger's $s$-polynomial is
\[
\spoly(f,g)=\xb^{D'-D(\bar{U}, U)}f-\xb^{D'-D(\bar{V}, V)}g=\xb^{D_1}-\xb^{D_2} \ ,
\]
where $D_1=D'-D(\bar{V}, V)+D(V, \bar{V})=D'-\Delta(\chi_{\bar{V}})$ is the effective divisor obtained from $D'$ by firing the set $\bar{V}$. Similarly $D_2=D'-D(\bar{U}, U)+D(U, \bar{U})=D'-\Delta(\chi_{\bar{U}})$ is the effective divisor obtained from $D'$ by firing the set $\bar{U}$. It follows from this interpretation that
\begin{equation} \label{eq:chip}
D_1-\Delta(\chi_{\bar{U} \backslash \bar{V}})=D_2-\Delta(\chi_{\bar{V} \backslash \bar{U}}) =D' - \Delta(\chi_{\bar{U} \cup \bar{V}})\geq 0 \ .
\end{equation}
The reason is the net effect of firing first the set $\bar{V}$ and then the set $\bar{U} \backslash \bar{V}$ is the same as firing the set $\bar{U} \cup \bar{V}$; chips going along edges connecting $\bar{V}$ and $\bar{U} \backslash \bar{V}$ cancel each other.

Without loss of generality we assume that $\LM(\spoly(f,g))=\xb^{D_1}$. It follows from \eqref{eq:chip} that we can reduce it by $h_1=\xb^{D(\bar{U}\backslash \bar{V}, U \cup \bar{V})}-\xb^{D(U \cup \bar{V}, \bar{U}\backslash\bar{V})} \in \Gb'(G,q)$ associated to the $2$-flag $(U\cup \bar{V}) \subsetneq V(G)$, and get:

\[
\begin{aligned}
\spoly(f,g) - \xb^{D_1-D(\bar{U}\backslash \bar{V}, U \cup \bar{V})}h_1&=\xb^{D_1-\Delta(\chi_{\bar{U} \backslash \bar{V}})}-\xb^{D_2}\\
&=\xb^{D' - \Delta(\chi_{\bar{U} \cup \bar{V}})}-\xb^{D'-\Delta(\chi_{\bar{U}})} \ .
\end{aligned}
\]

The leading monomial is now $\xb^{D'-\Delta(\chi_{\bar{U}})}$. Again, it follows from \eqref{eq:chip} that we can reduce this by $h_2=\xb^{D(\bar{V}\backslash \bar{U}, \bar{U} \cup V)}-\xb^{D(\bar{U} \cup V, \bar{V}\backslash\bar{U})} \in \Gb'(G,q)$ associated to the $2$-flag $(\bar{U}\cup V) \subsetneq V(G)$, and get:
\[
\begin{aligned}
\spoly(f,g) - \xb^{D_1-D(\bar{U}\backslash \bar{V}, U \cup \bar{V})}h_1 -\xb^{D_2-D(\bar{V}\backslash \bar{U}, \bar{U} \cup V)}h_2
&= \xb^{D' - \Delta(\chi_{\bar{U} \cup \bar{V}})}-\xb^{D' - \Delta(\chi_{\bar{U} \cup \bar{V}})} \\
&=0
\end{aligned}
\]
completing the proof that $\Gb'(G,q)$ is a Gr\"obner basis.

\medskip

Finally we show that if we only consider the flags in $\Sf_2(G,q)$ then we get a minimal Gr\"obner bases. We show this by successively removing the binomials which are not coming from connected $2$-flags. There are two steps:

$\bullet$  If $U_1 \subsetneq V(G)$ is a $2$-flag which is not in $\Sf_2(G,q)$ then there exists another  $2$-flag $V_1 \subsetneq V(G)$ such that $\xb^{D(\bar{V_1}, V_1)}  \mid \xb^{D(\bar{U_1},U_1)}$.

\medskip

\begin{itemize}
\item If $U_1$ is not connected let $V_1$ be the connected component of $U_1$ containing $q$; then $\bar{V_1}$ is the union of $\bar{U_1}$ and $U_1 \backslash V_1$ (i.e. other connected components of $U_1$). There is no edge between $V_1$ and other connected components of $U_1$ so for $v \in U_1 \backslash V_1$ we have $0=D(\bar{V_1}, V_1)(v) \leq D(\bar{U_1}, U_1)(v)$. Since $\bar{U_1} \subseteq \bar{V_1}$,  for $v \in \bar{U_1}$ we have $D(\bar{V_1}, V_1)(v) \leq D(\bar{U_1}, U_1)(v)$.
\item If $\bar{U_1}$ is not connected let $V_1$ be the complement of any connected component of $\bar{U_1}$. In this case $D(\bar{V_1}, V_1)(v) = D(\bar{U_1}, U_1)(v)$ for all $v \in \bar{V_1}$.
\end{itemize}

\medskip

$\bullet$  If $U_1 \subsetneq V(G)$ is in $\Sf_2(G,q)$ then its binomial cannot be removed. Otherwise,
there exists a different $2$-flag $V_1 \subsetneq V(G)$ in $\Sf_2(G,q)$ such that $\xb^{D(\bar{V_1}, V_1)}  \mid \xb^{D(\bar{U_1},U_1)}$ which is a contradiction by Corollary~\ref{cor:injectivity}.

\medskip

Homogeneity with respect to the $\ZZ$ and $\Pic(G)$ gradings is obvious.
\end{proof}

\medskip

\begin{Remark}
It is easy to check with examples (e.g. a path) that $\Gb(G,q)$ is generally not the {\em reduced} Gr\"obner bases for $(I_G, <)$.
\end{Remark}
%%%%%%%%%%%%%%%%%%%%%%%%%%%%%%%%%%%%%%%%%%%%%%%%%%%%%%%%%%%%%%%%%%%%%%%%%%%%%%%%%%%%%%%%%%%%%%%%%%%%%%%%%%%%%%%

Theorem~\ref{thm:Cori} can be rephrased as having a bijection between $\Sf_2(G,q)$ and $\Gb(G,q)$. The following theorem gives a generalization of this fact.

\medskip

%%%%%%%%%%%%%%%%%%%%%%%%%%%%%%%%%%%%%%%%%%%%%%%%%%%%%%%%%%%%%%%%%%%%%%%%%%%%%%%%%%%%%%%%%%%%%%%%%%%%%%%%%%%%%%%

\begin{Theorem}\label{thm:GB}
Fix a pointed graph $(G,q)$ and let $\As=\ZZ$ or $\As=\Pic(G)$. For each $k \geq 0$ there exists a natural injection
\[\psi_k:\Sf_{k+2}(G, q)\hookrightarrow \syz_{k}(\Gb(G,q))\]
such that
\begin{itemize}
\item[(i)] For some module ordering $<_k$, the set $\Gb_k(G,q):=\Image(\psi_k)$ forms a minimal $\As$-homogeneous Gr\"obner bases of $(\syz_{k}(\Gb(G,q)), <_k)$,
\item[(ii)] For $\Uc \in \Sf_{k+2}(G,q)$ of the form $U_1 \subsetneq U_2 \subsetneq \cdots \subsetneq V(G)$
we have
\begin{equation}\label{eq:LMpsi}
\LM(\psi_k(\Uc))=\xb^{D(U_2\backslash U_1, U_1)}[\psi_{k-1}(\Uc^{(1)})] \ .
\end{equation}
\end{itemize}
\end{Theorem}

\medskip

\begin{proof}
For consistency in the notation we define $\syz_{-1}(\Gb(G,q))=\{0\}$ and the map
\[\psi_{-1}:\Sf_{1}(G, q)\hookrightarrow \{0\}\]
sends the canonical connected $1$-flag $V(G)$ to $0$.

\medskip

The proof is by induction on $k \geq 0$.

\medskip

{\bf Base case.} For $k=0$ the result is proved in Theorem~\ref{thm:Cori}. Here $\Gb_0(G,q)=\Gb(G,q)$ and $<_0$ is $<$, and
\[
\psi_{0}:\Sf_{2}(G, q)\hookrightarrow \syz_{0}(\Gb(G,q))=I_G
\]
\[
(U_1 \subsetneq U_2) \mapsto (\xb^{D(U_2\backslash U_1, U_1)}-\xb^{D(U_1, U_2\backslash U_1)})[0] \ ,
\]
and
 $\LM(\psi_k(\Uc))=\xb^{D(U_2\backslash U_1, U_1)}[0]$.

\medskip

{\bf Induction hypothesis.} Now let $k>0$ and assume that there exists a bijection
\[\psi_{k-1}:\Sf_{k+1}(G, q)\rightarrow \Gb_{k-1}(G,q)\subseteq \syz_{k-1}(\Gb(G,q))\]
such that $\Gb_{k-1}(G,q)$ forms a minimal homogeneous Gr\"obner bases of $\syz_{k-1}(\Gb(G,q))$ with respect to $<_{k-1}$, and \eqref{eq:LMpsi} holds for the leading monomials.

\medskip

Via the bijection $\psi_{k-1}$, the set $\Gb_{k-1}(G,q)$ inherits a total ordering $\prec'_{k-1}$ from the total ordering $\prec_{k+1}$ on $\Sf_{k+1}(G, q)$, that is
\[f\prec'_{k-1} h \quad \text{in} \quad \Gb_{k-1}(G,q) \quad \Leftrightarrow \quad \psi_{k-1}^{-1}(f)\prec_{k+1} \psi_{k-1}^{-1}(h) \quad \text{in} \quad \Sf_{k+1}(G, q) .\]

\medskip

{\bf Inductive step.} Given $\Uc \in \Sf_{k+2}(G, q)$ let $\Uc^{(1)}$ and $\Uc^{(2)}$ be as defined in Definition~\ref{def:U12}. We {\bf \em define}
\begin{equation} \label{eq:map}
\psi_k:\Sf_{k+2}(G, q)\rightarrow \syz_{k}(\Gb(G,q))
\end{equation}
\[
\Uc \mapsto s(\psi_{k-1}(\Uc^{(1)}),\psi_{k-1}(\Uc^{(2)})) \ .
\]

\medskip

In the following $\Uc, \Vc \in \Sf_{k+2}(G, q)$ are of the form
\[
U_1 \subsetneq U_2 \subsetneq \cdots \subsetneq V(G)
\]
\[
V_1 \subsetneq V_2 \subsetneq \cdots \subsetneq V(G) \ .
\]

The result follows from a series of claims. 

\medskip

\noindent{\bf Claim 1.} $\psi_k$ is a well-defined.

\medskip

By Proposition~\ref{pro:well-def(a)}
\[
\Uc^{(1)} \in \Sf_{k+1}(G, q) , \quad \Uc^{(2)} \in \Sf_{k+1}(G, q), \quad \Uc^{(1)}  \prec_{k+1} \Uc^{(2)} \ .
\]
So by the induction hypothesis
\[\psi_{k-1}(\Uc^{(1)}),\psi_{k-1}(\Uc^{(2)}) \in \Gb_{k-1}(G,q)\]
 and by the definition of the total ordering on $\Gb_{k-1}(G,q)$ we have
\[
\psi_{k-1}(\Uc^{(1)}) \prec'_{k-1} \psi_{k-1}(\Uc^{(2)}) \ .
\]

Let $\Uc^{(1,1)}:=(\Uc^{(1)})^{(1)}$ and $\Uc^{(2,1)}:=(\Uc^{(2)})^{(1)}$. It is apparent from Definition~\ref{def:U12} that
\[
\Uc^{(1,1)} = \Uc^{(2,1)}.
\]
By the induction hypothesis and \eqref{eq:LMpsi}, $\LM(\psi_{k-1}(\Uc^{(1)}))$  and $\LM(\psi_{k-1}(\Uc^{(2)}))$ are both multiples of the same free bases element $[\psi_{k-2}(\Uc^{(1,1)})] = [\psi_{k-2}((\Uc^{(2,1)})]$. It follows that
\[
s(\psi_{k-1}(\Uc^{(1)}),\psi_{k-1}(\Uc^{(2)})) \in \Sc(\Gb_{k-1}(G,q)) \subset \syz_{k}(\Gb(G,q)) \
\]
is well-defined (see Theorem~\ref{thm:Schreyer}).

\medskip

\noindent{\bf Claim 2.} $\Gb_k(G,q):=\Image(\psi_k)$ consists of homogeneous elements.

\medskip

Since $\psi_{k-1}(\Uc^{(1)})$ and $\psi_{k-1}(\Uc^{(2)})$ are homogeneous by the induction hypothesis, it follows that $s(\psi_{k-1}(\Uc^{(1)}),\psi_{k-1}(\Uc^{(2)}))$ is also homogeneous.

\medskip

\noindent{\bf Claim 3.} $\LM(\psi_k(\Uc))=\xb^{D(U_2\backslash U_1, U_1)}[\psi_{k-1}(\Uc^{(1)})]$ .

\medskip

From Lemma~\ref{lem:lm} it suffices to show that $D(U_2\backslash U_1, U_1)=\max(\alpha, \beta)-\alpha$ where
\[
\LM(\psi_{k-1}(\Uc^{(1)}))=\xb^{\alpha}[\psi_{k-2}(\Uc^{(1,1)})] \quad , \quad\LM(\psi_{k-1}(\Uc^{(2)}))=\xb^{\beta}[\psi_{k-2}(\Uc^{(2,1)})] \ .
\]
But this is precisely Lemma~\ref{lem:KU1U2}.

\medskip

\noindent{\bf Claim 4.} $\psi_k$ is injective.

\medskip

If $\Uc, \Vc \in \Sf_{k+2}(G, q)$ are such that $\psi_k(\Uc)=\psi_k(\Vc)$ then their leading monomials should be equal:
\[
\xb^{D(U_2\backslash U_1, U_1)}[\psi_{k-1}(\Uc^{(1)})]=\xb^{D(V_2\backslash V_1, V_1)}[\psi_{k-1}(\Vc^{(1)})] \ .
\]
Therefore $\psi_{k-1}(\Uc^{(1)})=\psi_{k-1}(\Vc^{(1)})$ and $D(U_2\backslash U_1, U_1)=D(V_2\backslash V_1, V_1)$. By the induction hypothesis $\psi_{k-1}$ is injective which implies $\Uc^{(1)}=\Vc^{(1)}$ and $D(U_2\backslash U_1, U_1)=D(V_2\backslash V_1, V_1)$. It follows from Corollary~\ref{cor:injectivity} that $U_1=V_1$ and $\Uc = \Vc$.

\medskip

Our last claim below will finish the inductive step.
\medskip

\noindent{\bf Claim 5.} $\Image(\psi_k)$ forms a minimal homogeneous Gr\"obner bases of $\syz_{k}(\Gb(G,q))$ with respect to $<_k$ obtained from $<_{k-1}$ according to \eqref{orderpull}.

\medskip

We have already shown in the proof of Claim 1 that $\Image(\psi_k) \subseteq \Sc(\Gc_{k-1}(G,q))$. By Theorem~\ref{thm:Schreyer} and Remark~\ref{rem:S(M)} it remains to show that
\begin{itemize}
\item[{\bf (I)}] $0 \not \in \Image(\psi_k)$.
\item[{\bf (II)}] For any element $s(f,h) \in \Sc(\Gc_{k-1}(G,q))$ there exists an element $g \in \Image(\psi_k)$ such that $\LM(g) \mid \LM(s(f,h))$.
\item[{\bf (III)}] For any two elements $g, g' \in \Image(\psi_k)$, if $\LM(g)\mid\LM(g')$ then $g=g'$.
\end{itemize}

\medskip

{\bf (I)} follows immediately from Claim 3 above.

{\bf Proof of (II).} By the induction hypothesis $f=\psi_{k-1}(\Wc)$ and $h=\psi_{k-1}(\Vc)$ for two $\Wc \prec_{k+1}\Vc$ in $\Sf_{k+1}(G,q)$ such that $\Wc^{(1)}=\Vc^{(1)}$. We need to find $\Uc \in \Sf_{k+2}(G,q)$ such that \[\LM(s(\Uc^{(1)},\Uc^{(2)}))\mid \LM(s(\psi_{k-1}(\Wc),\psi_{k-1}(\Vc))) \ .\]

We use Proposition~\ref{prop:max}. From the previous paragraph it follows
\[\Vc \in \mathfrak{{N_\Wc}} = \{\Xc\in\Sf_{k+1}(G, q): \, \Wc^{(1)} = \Xc ^{(1)} \text{ and }\ \Wc\prec_{k+1}\Xc\}\ .\]

Hence there exists a $\Wc' \in \mathfrak{{N_\Wc}}$ such that $\K(\Wc,\Wc') \leq \K(\Wc,\Vc)$, and $\Uc^{(1)}=\Wc$ and $\Uc^{(2)}=\Wc'$ for some $\Uc \in \Sf_{k+1}(G, q)$.

By \eqref{eq:LMpsi} and Lemma~\ref{lem:KU1U2} (or Claim 3 above) we have
\[
\LM(\psi_k(\Uc))=\xb^{\K(\Wc,\Wc')-\alpha}[\psi_{k-1}(\Wc)] \ ,
\]
\[
\LM(s(\psi_{k-1}(\Wc),\psi_{k-1}(\Vc)))=\xb^{\K(\Wc,\Vc)-\alpha}[\psi_{k-1}(\Wc)] \ ,
\]
where $\alpha=D(U_3\backslash U_2, U_2)$. Therefore
\[
\LM(\psi_k(\Uc)) \mid \LM(s(\psi_{k-1}(\Wc),\psi_{k-1}(\Vc))) \ .
\]

\medskip

{\bf Proof of (III).} We need to show that for any $\Uc, \Vc \in \Sf_{k+2}(G, q)$ with $\Uc ^{(1)}=\Vc^{(1)}$, if $\LM(\psi_k(\Uc) \mid\LM(\psi_k(\Vc))$ then $\Uc=\Vc$.

From \eqref{eq:LMpsi}  $\LM(\psi_k(\Uc) \mid\LM(\psi_k(\Vc))$ is equivalent to $D(U_2\backslash U_1, U_1) \leq D(V_2\backslash V_1, V_1)$. This together with $\Uc ^{(1)}=\Vc^{(1)}$ implies $\Uc=\Vc$ by Corollary~\ref{cor:injectivity}.
\end{proof}

\medskip

\begin{Remark}\label{rmk:inithm}
In Theorem~\ref{thm:GB} if we replace $\Gb(G,q)$ with
\[\{\xb^{D(U_2\backslash U_1, U_1)}: \, U_1 \subsetneq U_2=V(G) \text{ is in } \Sf_2(G, q) \} \ ,
\]
(i.e. the initial terms of the Gr\"obner bases constructed in Theorem~\ref{thm:Cori}) and replace $\psi_{0}$ with
\[
\Sf_{2}(G, q)\hookrightarrow \ini(I_G)
\]
\[
(U_1 \subsetneq U_2) \mapsto \xb^{D(U_2\backslash U_1, U_1)}[0] \ ,
\]
then the exact same statement and proof are correct for the case of $\ini(I_G)$. As a corollary the exact same recipe gives a free resolution for $\ini(I_G)$ as well.
\end{Remark}

%%%%%%%%%%%%%%%%%%%%%%%%%%%%%%%%%%%%%%%%%%%%%%%%%%%%%%%%%%%%%%%%%%%%%%%%%%%%%%%%%%%%%%%%%%%%%%%%%%%%%%%%%%%%%%%
%%%%%%%%%%%%%%%%%%%%%%%%%%%%%%%%%%%%%%%%%%%%%%%%%%%%%%%%%%%%%%%%%%%%%%%%%%%%%%%

\section{Minimality of the resolution of $I_G$ and $\ini(I_G)$} \label{sec:minimality}

In Theorem~\ref{thm:GB} and Remark~\ref{rmk:inithm} we constructed free resolutions for the ideals $I_G$ and $\ini(I_G)$. In this section we take a close look at \eqref{eq:map} to show that the constructed resolutions are indeed minimal. Note that the basis elements of the free module $F_{k}$ correspond to elements of $\Sf_{k+2}(G,q)$. We show that for any $\Uc \in \Sf_{k+2}(G,q)$ the corresponding basis element $[\psi_{k}(\Uc)]$ maps to a combination of basis elements $[\psi_{k-1}(\Vc)]$, where each $\Vc$ is obtained from $\Uc$ by {\em merging} (\S\ref{sec:Mergeable}) appropriate connected parts. Moreover, the coefficients appearing in this combination are all non-units and, therefore, the constructed resolution is minimal (Theorem~\ref{thm:betti_binom}).

\subsection{Contraction map}

To understand merging, we first need to study contraction maps.

\begin{Definition}\label{def:U^k}
Assume that $\Uc \in \Sf_{k}(G, q)$. Let $G_{/\Uc}$ be the graph obtained from $G$ by contracting the unoriented edges of $G(\Uc)$ and let $\phi: G\rightarrow G_{/\Uc}$ be the contraction map.
More precisely, $G_{/\Uc}$ is the graph on the vertices $u_1,\ldots,u_{k}$ corresponding to the collection $(U_i\backslash U_{i-1})_{i=1}^{k}$, i.e. $u_i=\phi(U_i\backslash U_{i-1})$. For any edge between $U_i\backslash U_{i-1}$ and $U_j\backslash U_{j-1}$ there is an edge between $u_i$ and $u_j$.
\end{Definition}

\begin{Example}\label{exam:contract}
Let $G$ be the graph in Example~\ref{exam:G(U)}. For
\[
\Uc:\ \ \{v_1\}\subset \{v_1,\boldsymbol{v_2}\}\subset \{v_1,v_2,\boldsymbol{v_3},\boldsymbol{v_4}\} \subset \{v_1,v_2,v_3,v_4,\boldsymbol{v_5}\}
\]
the graph $G_{/\Uc}$ depicted in the following figure in which $u_1=v_1$, $u_2=v_2$, the vertex $u_3$ corresponds to $U_3\backslash U_2=\{v_3,v_4\}$, and $u_4$ corresponds to $U_4\backslash U_3=\{v_5\}$.

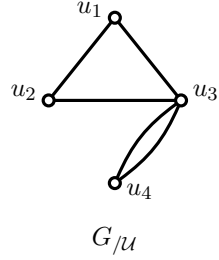
 \begin{figure}[h!]
\begin{center} \begin{tikzpicture}
[scale = .22, very thick = 15mm]

  \node (n1) at (5,11) [Cgray] {};
  \node (n2) at (1,6)  [Cgray] {};
  \node (n3) at (9,6)  [Cgray] {};
  \node (n5) at (5,1)  [Cgray] {};
  \foreach \from/\to in {n1/n2,n2/n3, n1/n3}
    \draw[-] (\from) -- (\to);

 \path[-] (n3) edge [bend right = 15] node {} (n5);
  \path[-] (n3) edge [bend left = 15] node {} (n5);
    \node(p1) at (3.5, 11.5) [C0] {$u_1$};
    \node(p2) at (-0.5, 6.5) [C0] {$u_2$};
        \node(p3) at (10.5, 6.5) [C0] {$u_3$};
    \node(p5) at (6.5, 0.5) [C0] {$u_4$};

 \node(p7) at (5, -2.5) [C0] {$G_{/\Uc}$};

%%%%%%%%%%%%%

\end{tikzpicture}
\end{center}\end{figure}
\end{Example}

\begin{Remark}\label{rem:unique divisor}
The contraction map $\phi: G\rightarrow G_{/\Uc}$ induces the map
\[
\phi_{\ast}:\Div(G)\rightarrow \Div(G_{/\Uc})\ \ \text{ with }\ \
\phi_{\ast}(\sum_{v\in V(G)}a_v (v))=
\sum_{v\in V(G)}a_v (\phi(v)) \ .
\]
If the indices $i$ and $j$ are given, we obtain two divisors
\[D'(u_i,u_j) \in \Div(G_{/\Uc}) \quad \text{and} \quad D(U_i\backslash U_{i-1},U_j\backslash U_{j-1}) \in \Div(G)\]
 which are related by the map $\phi_{\ast}$ (see \eqref{DAB}). Here we use the notation $D'( \cdot, \cdot)$ for divisors on $G_{/\Uc}$ and $D( \cdot, \cdot)$ for divisors on $G$.

In particular, an ordering on the vertices of $G_{/\Uc}$ gives an  ordering on the collection of subsets $(U_{i}\backslash U_{i-1})_{i=1}^k$ of $V(G)$. By Definition~\ref{def:Du} we get a divisor $D'$ on $G_{/ \Uc}$ and a divisor $D$ on $G$, and $\phi_{\ast}(D) = D'$.
\end{Remark}

\begin{Remark}\label{rem:contract}
We also have the map $\phi^{\ast}:\Ff_{s}(G_{/\Uc},u_1)\rightarrow \Ff_{s}(G,q)$ induced by sending each vertex of $G_{/\Uc}$ to its preimage under $\phi$. The map $\phi$ and the total ordering $\preceq$ on $\mathfrak{C}^{\text{op}}(G,q)$  (as in Definition~\ref{def:prec1}) give a total ordering $\preceq'$ on $\mathfrak{C}^{\text{op}}(G_{/ \Uc},u_1)$. The ordering $\preceq'$ induces a strict total ordering $\prec'_\ell$ on $\Ff_{\ell}(G_{/\Uc}, u_1)$ compatible with the total ordering on connected flags on $(G,q)$; that is, $\Xc \prec'_\ell \Yc$ if and only if $\phi^{\ast}(\Xc)\prec_\ell \phi^{\ast}(\Yc)$.  Therefore, we get a map

\begin{equation} \label{phistar}
\phi^{\ast}:\Sf_{s}(G_{/\Uc},u_1)\rightarrow \Sf_{s}(G,q) \ .
\end{equation}

This gives a one-to-one correspondence between the elements $\Vc' \in \Sf_{s}(G_{/\Uc},u_1)$ and the elements $\Vc \in \Sf_{s}(G,q)$. Under this correspondence, for any $u_i \in V'_j \backslash V'_{j-1}$ we have $U_i\backslash U_{i-1}\subseteq V_j \backslash V_{j-1}$ and thus $V_j=\bigcup_{u_i\in V_j'} (U_i\backslash U_{i-1})$. For any element $\Vc$ in the image of $\phi^{\ast}$ the preimage $\Vc'$ is obtained by $V'_j=\{u_i: \, u_i = \phi(U_i\backslash U_{i-1})  \text{ and }  U_i\backslash U_{i-1} \subseteq V_j  \}$. In particular, $\Uc$ itself is in the image of $\phi^{\ast}$.

\end{Remark}

The following example explains the notation introduced in the above remarks.

\begin{Example}
In Example~\ref{exam:contract} the ordering $u_1,u_2,u_3,u_4$ on $V(G_{/\Uc})$ induces the ordering 
$U_1\backslash U_0, U_2\backslash U_1, U_3\backslash U_2, U_4\backslash U_3$  on the collection $(U_{i}\backslash U_{i-1})_{i=1}^4$ of $V(G)$ which corresponds to $\Uc$. Also corresponding to the ordering $u_1,u_2,u_3,u_4$ on $V(G_{/\Uc})$ we get the divisor $D'=(u_2)+2(u_3)+2(u_4)$ on 
$G_{/ \Uc}$, where $\phi_{\ast}(D')=D(\Uc)$.

We consider 
\[
\Vc':\ \{u_1\}\subset \{u_1,\boldsymbol{u_3}\}\subset \{u_1,\boldsymbol{u_2},u_3,\boldsymbol{u_4}\}
\]
 in 
$\Sf_{3}(G_{/\Uc},u_1)$. Then $\phi^{\ast}(\Vc')=\Vc$, where 
\[
\Vc:\ \{v_1\}\subset \{v_1,\boldsymbol{v_3,v_4}\}\subset \{v_1,\boldsymbol{v_2},v_3,v_4,\boldsymbol{v_5}\}\ .
\]
More precisely, $V_1=U_1\backslash U_0, V_2=(U_1\backslash U_0)\cup (U_3\backslash U_2)$,  and  $V_3=(U_2\backslash U_1)\cup (U_4\backslash U_3)$.
\end{Example}

\medskip

\subsection{Mergeable parts}
\label{sec:Mergeable}
Given $\Uc \in \Sf_{k}(G,q)$ of the form
\[
 U_1 \subsetneq U_2 \subsetneq \cdots \subsetneq U_k =V(G)
\]
it is sometimes more convenient to work with the connected partition given by $A_\ell:= U_{\ell}\backslash U_{\ell-1}$ (for $1 \leq \ell \leq k$).

\medskip
 
Recall for any $\Uc \in \Sf_{k}(G,q)$ we get a partial orientation of $G$ which we denoted by $G(\Uc)$ in Definition~\ref{def:Du}. This partial orientation is acyclic with unique source on the underlying partition graph $G_{/\Uc}$ (Definition~\ref{def:U^k}). This means that the underlying partition graph does not contain any directed cycle and it has a unique source on the vertex corresponding to $A_1$. More generally, we say a partial orientation is {\em acyclic} if the associated oriented partition graph (obtained by contracting all unoriented edges) is acyclic. Equivalently, a partial orientation is acyclic if replacing every undirected edge with two antiparallel edges yields an acyclic directed graph. Recall that associated to each partial orientation we get a divisor as in Remark~\ref{rmk:indeg}. \

\medskip

\begin{Definition}\label{def:o_j}
Let $\Uc \in \Sf_{k}(G,q)$ and $A_\ell:= U_{\ell}\backslash U_{\ell-1}$ (for $1 \leq \ell \leq k$) as before. We set $o_0(\Uc):=G(\Uc)$. For $j>0$ the partial orientation $o_j(\Uc)$ is defined inductively as follows: we obtain $o_j(\Uc)$ from $o_{j-1}(\Uc)$ by reversing the orientation of the edges between $A_j$ and $V(G)\backslash A_j$ in $o_{j-1}(\Uc)$.
\end{Definition}
\medskip

Note that, when all oriented edges are directed away from $A_j$, reversing the orientation of the edges between $A_j$ and $V(G)\backslash A_j$ in $o_{j-1}(\Uc)$ is equivalent to performing a chip-firing move, in which all vertices in $A_j$ borrow chips from their neighbors in $V(G)\backslash A_j$.
Note that $o_j(\Uc)$ is well-defined since all edges are directed away from $A_j$ in $o_{j-1}(\Uc)$.

\begin{Definition}
Let $\mathfrak{c}(\Uc)$ denote the set consisting of all partial orientations $o_j(\Uc)$ of $G$.
\end{Definition}

\medskip

\begin{Example}\label{exam:C4} {Let $G$ be the $4$-cycle on the vertices $1,2,3,4$ such that $1$ is the distinguished vertex. Let $\Uc$ be the connected flag
$\Uc:\ \{1\}\subset\{1,2\}\subset\{1,2,3\}\subset\{1,2,3,4\}$. Then in the following we depict the graph corresponding to
$o_j(\Uc)$ for all $j$ (see Definition~\ref{def:o_j}). Note that $o_4(\Uc)=G(\Uc)$. }
\medskip

\begin{figure}[h!] \label{fig:o_j_2} \begin{center}

\begin{tikzpicture}  [scale = .22, very thick = 13mm]

 \node (n4) at (-6,13)  [Cgray] {1};
  \node (n1) at (-6,23) [Cgray] {4};
  \node (n2) at (-9,18)  [Cgray] {2};
  \node (n3) at (-3,18)  [Cgray] {3};
\foreach \from/\to in {n2/n1,n2/n4}
    \draw[<-] (\from) -- (\to);

 \foreach \from/\to in {n4/n3,n1/n3}
    \draw[brown][->] (\from) -- (\to);
     \node(m1) at (-6,10.3) [C0] {$o_3(\Uc)$};

 \node (n4) at (4,13)  [Cgray] {1};
 \node(m1) at (4,10.3) [C0] {$o_2(\Uc)$};
  \node (n1) at (4,23) [Cgray] {4};
  \node (n2) at (1,18)  [Cgray] {2};
  \node (n3) at (7,18)  [Cgray] {3};
  \foreach \from/\to in {n2/n1,n2/n4}
    \draw[red][<-] (\from) -- (\to);

 \foreach \from/\to in {n3/n4,n3/n1}
    \draw[->] (\from) -- (\to);

 \node (n4) at (14,13)  [Cgray] {1};
  \node (n1) at (14,23) [Cgray] {4};
  \node(m1) at (14,10.3) [C0] {$o_1(\Uc)$};
  \node (n2) at (11,18)  [Cgray] {2};
  \node (n3) at (17,18)  [Cgray] {3};
  \foreach \from/\to in {n1/n2,n1/n3}
    \draw[<-] (\from) -- (\to);

 \foreach \from/\to in {n3/n4,n2/n4}
    \draw[blue][->] (\from) -- (\to);

      \node (n4) at (24,13)  [Cgray] {1};
  \node (n1) at (24,23) [Cgray] {4};
  \node(m1) at (24,10.3) [C0] {$G(\Uc)$};
  \node (n2) at (21,18)  [Cgray] {2};
  \node (n3) at (27,18)  [Cgray] {3};
  \foreach \from/\to in {n1/n2,n1/n3, n3/n4,n2/n4}
    \draw[<-] (\from) -- (\to);
\end{tikzpicture}
\end{center}

\end{figure}

\end{Example}

\medskip

\bigskip

For disjoint subsets $A, B \subset V(G)$ let $E(A,B)$ denote the set of edges between $A$ and $B$.

\begin{Definition}
Let $\Uc \in \Sf_{k}(G,q)$ and $A_\ell:= U_{\ell}\backslash U_{\ell-1}$ (for $1 \leq \ell \leq k$) as before, and assume that there are some edges connecting $A_i$ and $A_j$, that is, $E(A_i,A_j) \ne \emptyset$.
\begin{itemize}
\item[(i)] We say $A_i$ is {\em mergeable} with $A_j$ {\em in $G(\Uc)$}  if  all edges in $E(A_i,A_j)$ are oriented from $A_i$ to $A_j$ and the partial orientation obtained from $G(\Uc)$ by removing the orientations on $E(A_i,A_j)$ is acyclic.

\medskip

Note that in this case $i<j$. We let $\Merge(\Uc;A_i,A_j) \in \Sf_{k-1}(G,q)$ denote the corresponding unique connected $(k-1)$-flag whose connected parts are $A_\ell$ (for $\ell \ne i,j$) and $A_i \cup A_j$.

\medskip

\item[(ii)] We say $A_i$ is {\em mergeable} with $A_j$ {\em in $o_j(\Uc)$} where $i>j>0$, if the partial orientation obtained by removing the orientations on $E(A_i,A_j)$ in $o_j(\Uc)$ results in an acyclic partial orientation. Note that $E(A_i,A_j)$ are oriented from $A_i$ to $A_j$ in $o_j(\Uc)$.

\medskip

Let $\Merge(o_j(\Uc);A_i,A_j) \in \Sf_{k-1}(G,q)$ denote the connected partition of $G$ whose connected parts are $A_\ell$ (for $\ell \ne i,j$) and $A_i \cup A_j$, together with the acyclic (partial) orientation obtained from $o_j(\Uc)$ by removing the orientations on $E(A_i,A_j)$.
As usual one obtains an associated divisor by reading the indegrees in this new partial orientation. This gives a {\em maximal reduced divisor} (see \S\ref{sec:Connection}) on the associated graph of partitions via the map $\phi$ (see Remark~\ref{rem:unique divisor}). This maximal reduced divisor gives a total ordering on the vertices of the graph of partitions (e.g., by performing {\em Dhar's algorithm} -- see \S\ref{sec:Connection} and \cite{FarbodMatt12}).
Consider the induced partial orientation of $G$ obtained in this way, and let $\Merge(\mathfrak{c}(\Uc);A_i,A_j) \in \Sf_{k-1}(G,q)$ denote the associated connected flag.

\end{itemize}

\end{Definition}

%%%%%%%%%%%%%%%%%%%%%%%%%%%%%%%%%%%%%%%
%%%%%%%%%%%%%%%%%%%%%%%%%%%%%%%%%%%%%%%%%%%%%%

\begin{Definition}
For $\Uc \in \Sf_{k}(G,q)$ we associate two subsets of $\Sf_{k-1}(G,q)$ as follows:
\begin{itemize}
\item[(i)] $\Ifr(\Uc):=\{\Wc : \, \Wc=\Merge(\Uc; A_i, A_j), \text{ for } A_i, A_j \text{ mergeable in } G(\Uc) \}$.
\item[(ii)] $\Bf(\Uc):=\{\Wc : \, \Wc=\Merge(\mathfrak{c}(\Uc); A_i, A_j), \text{ for } A_i, A_j \text{  mergeable in $o_j(\Uc)$ or $G(\Uc)$}\}$.
\end{itemize}
\end{Definition}

It immediately follows from the definitions that $\Ifr(\Uc) \subseteq \Bf(\Uc)$.
As we will see soon, $\Bf(\Uc)$ is related to the differential maps in our resolution of the binomial ideal $I_G$ and $\Ifr(\Uc)$ is related to the differential maps  in our resolution of the monomial ideal $\ini(I_G)$.

\medskip

\begin{Example} \label{exam:I(U)}
Let $G$ be of the $4$-cycle on the vertices $1,2,3,4$ in which we fix $1$ be the distinguished vertex.  Let $\Uc$ be the connected flag
$\Uc:\ \{1\}\subset\{1,2\}\subset\{1,2,3\}\subset\{1,2,3,4\}$. Here we list the elements of $\Ifr(\Uc)$.

%%%%%%%%%%%%%%%%%%%%%
\begin{figure}[ht] \label{fig:o_j_3} \begin{center}

\begin{tikzpicture} [scale = .22, very thick = 13mm]

 \node (n4) at (-6,13)  [Cgray] {1};
  \node (n1) at (-6,23) [Cgray] {4};
  \node (n2) at (-9,18)  [Cgray] {2};
  \node (n3) at (-3,18)  [Cgray] {3};
\foreach \from/\to in {n1/n2,n2/n4}
    \draw[<-] (\from) -- (\to);
 \foreach \from/\to in {n4/n3}
    \draw[brown][->] (\from) -- (\to);
     \node(m1) at (-6,8.3) [C1] {$\Merge(\Uc; 3, 4)$};
\foreach \from/\to in {n1/n3}
    \draw[brown][] (\from) -- (\to);
 \node (n4) at (4,13)  [Cgray] {1};
 \node(m1) at (4,8.3) [C1] {$\Merge(\Uc; 2, 4)$};
  \node (n1) at (4,23) [Cgray] {4};
  \node (n2) at (1,18)  [Cgray] {2};
  \node (n3) at (7,18)  [Cgray] {3};
  \foreach \from/\to in {n2/n4}
    \draw[red][<-] (\from) -- (\to);
\foreach \from/\to in {n2/n1}
    \draw[red][] (\from) -- (\to);
 \foreach \from/\to in {n4/n3,n3/n1}
    \draw[->] (\from) -- (\to);

 \node (n4) at (14,13)  [Cgray] {1};
  \node (n1) at (14,23) [Cgray] {4};
  \node(m1) at (14,8.3) [C1] {$\Merge(\Uc;1,2)$};
  \node (n2) at (11,18)  [Cgray] {2};
  \node (n3) at (17,18)  [Cgray] {3};
  \foreach \from/\to in {n1/n2,n1/n3}
    \draw[<-] (\from) -- (\to);

 \foreach \from/\to in {n4/n3}
    \draw[blue][->] (\from) -- (\to);
 \foreach \from/\to in {n2/n4}
    \draw[blue][] (\from) -- (\to);

      \node (n4) at (24,13)  [Cgray] {1};
  \node (n1) at (24,23) [Cgray] {4};
  \node(m1) at (24,8.3) [C1] {$\Merge(\Uc;1,3)$};
  \node (n2) at (21,18)  [Cgray] {2};
  \node (n3) at (27,18)  [Cgray] {3};
  \foreach \from/\to in {n1/n2,n1/n3}
    \draw[<-] (\from) -- (\to);

 \foreach \from/\to in {n4/n2}
    \draw[blue][->] (\from) -- (\to);
 \foreach \from/\to in {n3/n4}
    \draw[blue][] (\from) -- (\to);

\end{tikzpicture}
\end{center}

\end{figure}
\end{Example}
%%%%%%%%%%%%%%%%%%%%%%%%%%%%%%%%%%%%%%%%%%%%%%

\begin{Example}\label{exam:running example}
Returning to Example~\ref{exam:C4}, i.e., for $\Uc:\ \{1\}\subset\{1,2\}\subset\{1,2,3\}\subset\{1,2,3,4\}$, we list
the acyclic orientations $\Merge(\mathfrak{c}(\Uc);A_i,A_j)$ corresponding to the elements of $\Bf(\Uc)\backslash \Ifr(\Uc)$ which have been obtained from connected partitions $\Merge(o_j(\Uc);A_i,A_j)$.

%%%%%%%%%%%%%%%%%%%%%
\begin{figure}[h!] \label{fig:o_j_4} \begin{center}

\begin{tikzpicture} [scale = .22, very thick = 13mm]

\node (n4) at (-6,13)  [Cgray] {1};
  \node (n1) at (-6,23) [Cgray] {4};
  \node (n2) at (-9,18)  [Cgray] {2};
  \node (n3) at (-3,18)  [Cgray] {3};
\foreach \from/\to in {n2/n1,n2/n4}
    \draw[<-] (\from) -- (\to);
 \foreach \from/\to in {n4/n3}
    \draw[brown][->] (\from) -- (\to);
     \node(m1) at (-6,7.5) [C1] {$\Merge(o_3(\Uc); 4, 3)$};
\foreach \from/\to in {n1/n3}
    \draw[brown][] (\from) -- (\to);
 \node (n4) at (4,13)  [Cgray] {1};
 \node(m1) at (4,7.5) [C1] {$\Merge(o_2(\Uc); 4,2)$};
  \node (n1) at (4,23) [Cgray] {4};
  \node (n2) at (1,18)  [Cgray] {2};
  \node (n3) at (7,18)  [Cgray] {3};
  \foreach \from/\to in {n2/n4}
    \draw[red][<-] (\from) -- (\to);
\foreach \from/\to in {n2/n1}
    \draw[red][] (\from) -- (\to);
 \foreach \from/\to in {n3/n4,n3/n1}
    \draw[->] (\from) -- (\to);

 \node (n4) at (14,13)  [Cgray] {1};
  \node (n1) at (14,23) [Cgray] {4};
  \node(m1) at (14,7.5) [C1] {$\Merge(o_1(\Uc);2,1)$};
  \node (n2) at (11,18)  [Cgray] {2};
  \node (n3) at (17,18)  [Cgray] {3};
  \foreach \from/\to in {n1/n2,n1/n3}
    \draw[<-] (\from) -- (\to);

 \foreach \from/\to in {n3/n4}
    \draw[blue][->] (\from) -- (\to);
 \foreach \from/\to in {n2/n4}
    \draw[blue][] (\from) -- (\to);

      \node (n4) at (24,13)  [Cgray] {1};
  \node (n1) at (24,23) [Cgray] {4};
  \node(m1) at (24,7.5) [C1] {$\Merge(o_1(\Uc);3,1)$};
  \node (n2) at (21,18)  [Cgray] {2};
  \node (n3) at (27,18)  [Cgray] {3};
  \foreach \from/\to in {n1/n2,n1/n3}
    \draw[<-] (\from) -- (\to);

 \foreach \from/\to in {n2/n4}
    \draw[blue][->] (\from) -- (\to);
 \foreach \from/\to in {n3/n4}
    \draw[blue][] (\from) -- (\to);

\end{tikzpicture}

\end{center}

\end{figure}

%%%%%%%%%%%%%%%%%%%%%%%%%%%%%%%%%%%%%%%%%%%%%%%%%%

%%%%%%%%%%%%%%%%%%%%%
\begin{figure}[ht] \label{fig:o_j_5} \begin{center}

\begin{tikzpicture} [scale = .22, very thick = 13mm]

 \node (n4) at (-6,13)  [Cgray] {1};
  \node (n1) at (-6,23) [Cgray] {4};
  \node (n2) at (-9,18)  [Cgray] {2};
  \node (n3) at (-3,18)  [Cgray] {3};
\foreach \from/\to in {n2/n1,n2/n4}
    \draw[<-] (\from) -- (\to);
 \foreach \from/\to in {n4/n3}
    \draw[brown][->] (\from) -- (\to);
     \node(m1) at (-6,7.7) [C1] {$\Merge(\mathfrak{c}(\Uc); 4,3)$};
\foreach \from/\to in {n1/n3}
    \draw[brown][] (\from) -- (\to);
 \node (n4) at (4,13)  [Cgray] {1};
 \node(m1) at (4,7.7) [C1] {$\Merge(\mathfrak{c}(\Uc); 4,2)$};
  \node (n1) at (4,23) [Cgray] {4};
  \node (n2) at (1,18)  [Cgray] {2};
  \node (n3) at (7,18)  [Cgray] {3};
  \foreach \from/\to in {n2/n4}
    \draw[red][<-] (\from) -- (\to);
\foreach \from/\to in {n2/n1}
    \draw[red][] (\from) -- (\to);
 \foreach \from/\to in {n4/n3,n1/n3}
    \draw[->] (\from) -- (\to);

 \node (n4) at (14,13)  [Cgray] {1};
  \node (n1) at (14,23) [Cgray] {4};
  \node(m1) at (14,7.7) [C1] {$\Merge(\mathfrak{c}(\Uc);2,1)$};
  \node (n2) at (11,18)  [Cgray] {2};
  \node (n3) at (17,18)  [Cgray] {3};
  \foreach \from/\to in {n1/n2,n3/n1}
    \draw[<-] (\from) -- (\to);

 \foreach \from/\to in {n4/n3}
    \draw[blue][->] (\from) -- (\to);
 \foreach \from/\to in {n2/n4}
    \draw[blue][] (\from) -- (\to);

      \node (n4) at (24,13)  [Cgray] {1};
  \node (n1) at (24,23) [Cgray] {4};
  \node(m1) at (24,7.7) [C1] {$\Merge(\mathfrak{c}(\Uc);3,1)$};
  \node (n2) at (21,18)  [Cgray] {2};
  \node (n3) at (27,18)  [Cgray] {3};
  \foreach \from/\to in {n2/n1,n1/n3}
    \draw[<-] (\from) -- (\to);

 \foreach \from/\to in {n4/n2}
    \draw[blue][->] (\from) -- (\to);
 \foreach \from/\to in {n3/n4}
    \draw[blue][] (\from) -- (\to);

\end{tikzpicture}

\end{center}

\end{figure}

\end{Example}
%%%%%%%%%%

\begin{Lemma}\label{lem:contractible}
Let $\Uc\in \Sf_{k}(G,q)$ and assume that $E(A_i,A_j) \ne \emptyset$.
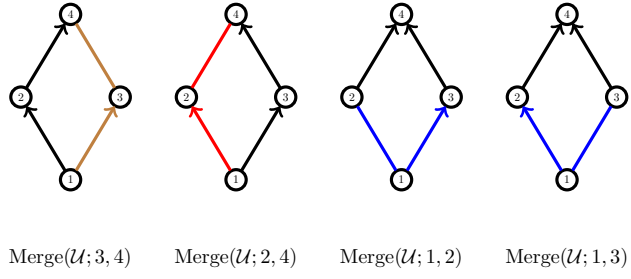
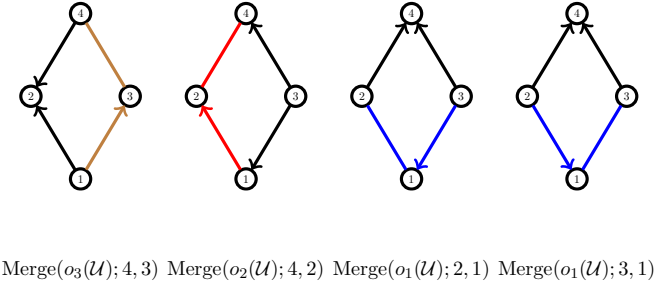
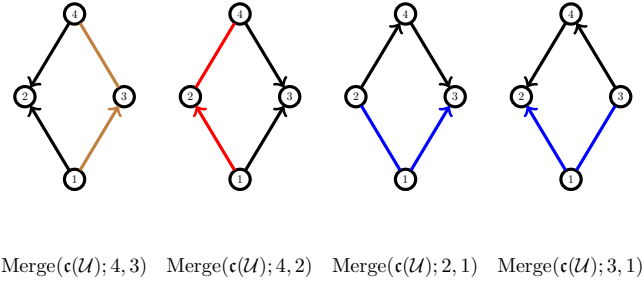
\begin{itemize}
\item[(a)] Assume that $\Merge(\Uc;A_i,A_j) \in \Ifr(\Uc)$. Then there exists $B_i \supseteq A_i$ such that
\[\Merge(\Uc^{(1)}; B_i,A_j) \in  \Ifr(\Uc^{(1)})\quad \textit{or}\quad \Merge(\Uc^{(2)}; B_i,A_j) \in \Ifr(\Uc^{(2)})\ .\]

\item[(b)] Assume that $\Merge(o_j(\Uc);A_i,A_j) \in \Bf(\Uc)\backslash \Ifr(\Uc)$ for $j>0$.
Then
\begin{itemize}
\item $\Merge(o_{j-1}(\Uc^{(1)});A_i,A_j) \in \Bf(\Uc^{(1)})$, or
\item there exist $\Wc \in \Ifr(\Uc)$, $B_i \supseteq A_i$ and $B_j \supseteq A_j$ such that
\[
\Merge(o_1(\Wc);B_i,B_j) \in \Bf(\Wc) \ .
\]
\end{itemize}
\end{itemize}
\end{Lemma}

\begin{proof}
(a) Since $A_1$ is a source in the partial orientation and cannot appear in any directed cycle we have the following:
\begin{itemize}
\item[(i)] if $A_i \ne A_1, A_2$ then $\Merge(\Uc^{(1)};A_i,A_j) \in \Ifr(\Uc^{(1)})$.
\item[(ii)] if $A_i=A_1$ then $\Merge(\Uc^{(1)};A_1 \cup A_2,A_j) \in \Ifr(\Uc^{(1)})$.
\item[(iii)] if $A_i=A_2$
\begin{itemize}
\item if $E(A_2,A_3)=\emptyset$ then $\Merge(\Uc^{(2)}; A_2,A_j) \in \Ifr(\Uc^{(2)})$.
\item if $E(A_2,A_3) \ne \emptyset$ then $\Merge(\Uc^{(2)}; A_2 \cup A_3 , A_j) \in \Ifr(\Uc^{(2)})$.
\end{itemize}
\end{itemize}

\medskip

In other words, in each case, we can find a $B_i \supseteq A_i$ such that $\Merge(\Uc^{(1)}; B_i,A_j) \in  \Ifr(\Uc^{(1)})$ or $\Merge(\Uc^{(2)}; B_i,A_j) \in \Ifr(\Uc^{(2)})$.
\medskip

(b) Assume that $\Merge(o_j(\Uc);A_i,A_j) \in \Bf(\Uc)\backslash \Ifr(\Uc)$. First assume that $G_{/\Uc}$ is a star graph, i.e., for each pair $\ell_1,\ell_2>1$ of indices $E(A_{\ell_1},A_{\ell_2})=\emptyset$. Then $A_j=A_1$ and $A_\ell$ is mergeable with $A_j$ in $G(\Uc)$ for all $\ell>1$. In particular, $A_2$ is mergeable with $A_1\cup A_3$ in $\Uc^{(2)}$ and $A_i$ is mergeable with $A_1\cup A_2$ in $\Uc^{(1)}$ for $i>2$ which is what we want.

\medskip

Now we assume that $E(A_{\ell_1},A_{\ell_2})\ne\emptyset$ for some $\ell_1,\ell_2$. Then we define
\[
r:=\min\{p:\ E(A_{p},A_\ell)\ne \emptyset\ \textit{have the same orientations in $G(\Uc)$ and $o_j(\Uc)$ for some $\ell$}\}
\]
and
\[
s:=\min\{\ell:\ E(A_{r},A_\ell)\ne \emptyset\ \textit{have the same orientations in $G(\Uc)$ and $o_j(\Uc)$}\}\ .
\]

Note that our assumption on $E(A_{\ell_1},A_{\ell_2})$ shows that these sets are nonempty. We also have $A_j \ne A_r$ since the outdegree of each vertex of $A_j$ in $o_j(\Uc)$ is zero, but there are some edges going out from $A_r$ to $A_s$. We set $\Wc=\Merge(\Uc;A_r,A_s)$.
In the following we consider all possible cases to show that $\Wc\in \Ifr(\Uc)$ with the desired properties:

\begin{itemize}
\item[(i)] $A_i\neq A_1$: then the edges between $A_1$ and $A_2$ are oriented from $A_1$ to $A_2$ in $o_j(\Uc)$ and $G(\Uc)$. We first reverse the orientation of the edges between $A_1\cup A_2$ and $V(G)\backslash (A_1\cup A_2)$. If $A_j=A_2$ then $A_i$ is mergeable with $A_1\cup A_2$ in $o_1(\Uc^{(1)})$ and so $\Merge(o_1(\Uc^{(1)});A_1\cup A_2,A_i) \in \Bf(\Uc^{(1)})$.
    Assume that $A_j\neq A_2$. Then we let the vertices of $A_{3}$ borrow from their neighbors in $V(G)\backslash A_3$ in order to get $o_2(\Uc^{(1)})$ which differs with $o_3(\Uc)$ just by merging the parts $A_1$ and $A_2$. We continue the same process of chip firing on the parts of $A_4,\ldots, A_{j}$ step-by-step in order to get $o_{j-1}(\Uc^{(1)})$ which can be obtained from $o_j(\Uc)$ just by merging the parts $A_1$ and $A_2$. This shows that $A_i$ is mergeable with $A_j$ in $o_{j-1}(\Uc^{(1)})$ as well, and so
$\Merge(o_{j-1}(\Uc^{(1)});A_i,A_j) \in \Bf(\Uc^{(1)})$.

\item[(ii)] $A_i=A_1$: then the same argument as case (i) shows the following cases can happen:

\begin{itemize}
\item if $A_j \ne A_r, A_s$ then
$\Merge(o_1(\Wc);A_1,A_j) \in \Bf(\Wc)$.

\item if $A_j=A_r$ or $A_j=A_s$  then
$\Merge(\mathfrak{c}(\Wc);A_1,A_r \cup A_s) \in \Bf(\Wc)$. \qed

\end{itemize}
\end{itemize}
\end{proof}

\medskip

There is a nice converse to Lemma~\ref{lem:contractible}(a). Our next result shows that the mergeable parts of $\Uc$ can be obtained from mergeable parts of the canonical flags $\Uc^{(1)}$  and $\Uc^{(2)}$.

\begin{Lemma}\label{lem:U1,U2}
There exists a one-to-one correspondence between elements $\Ifr(\Uc^{(1)})\cup \Ifr(\Uc^{(2)})$ and elements of $\Ifr(\Uc)$.
\end{Lemma}

\begin{proof}
Let $\Uc\in \Sf_{k}(G,q)$. Corresponding to each pair of mergeable parts in $G(\Uc^{(1)})$ or $G(\Uc^{(2)})$ we will find a unique pair of mergeable parts in $G(\Uc)$.
Assume that $\Wc\in \Ifr(\Uc^{(1)})\cup \Ifr(\Uc^{(2)})$. Then we consider the following cases:

$\bullet$ $\Wc\in \Ifr(\Uc^{(1)})$: Since $A_1$ is a source in the partial orientation and cannot appear in any directed cycle we have the following:
\begin{itemize}
\item[(i)] if $\Merge(\Uc^{(1)};A_i,A_j) \in \Ifr(\Uc^{(1)})$ then $A_i$ is mergeable with $A_j$ in $G(\Uc)$.
\item[(ii)] if $\Merge(\Uc^{(1)};A_1\cup A_2,A_i) \in \Ifr(\Uc^{(1)})$ then
\begin{itemize}
\item
$\Merge(\Uc;A_2,A_i) \in \Ifr(\Uc)$ if $E(A_2,A_i)\ne \emptyset$.
\item
$\Merge(\Uc;A_1,A_i) \in \Ifr(\Uc)$ if $E(A_2,A_i)=\emptyset$.
\end{itemize}
\end{itemize}

\medskip

$\bullet$ $\Wc\in \Ifr(\Uc^{(2)})$: First of all note that $\Merge(\Uc^{(2)};A_1,A_2) \in \Ifr(\Uc)$.
Then we have the following cases:
\begin{itemize}
\item[(i)] if $\Merge(\Uc^{(2)};A_i,A_j) \in \Ifr(\Uc^{(2)})$ then $\Merge(\Uc;A_i,A_j) \in \Ifr(\Uc)$.
\item[(ii)] if $\Merge(\Uc^{(2)};A_2\cup A_3,A_i) \in \Ifr(\Uc^{(2)})$ then
\begin{itemize}
\item $\Merge(\Uc;A_3,A_i) \in \Ifr(\Uc)$ if $E(A_3,A_i)\ne \emptyset$.
\item $\Merge(\Uc;A_2,A_i) \in \Ifr(\Uc)$ if $E(A_3,A_i)=\emptyset$.
\end{itemize}

\item[(iii)] if $\Merge(\Uc^{(2)};A_1\cup A_3,A_i) \in \Ifr(\Uc^{(2)})$ then
\begin{itemize}
\item $\Merge(\Uc;A_3,A_i) \in \Ifr(\Uc)$ if $E(A_3,A_i)\ne \emptyset$.
\item $\Merge(\Uc;A_1,A_i) \in \Ifr(\Uc)$ if $E(A_3,A_i)=\emptyset$.
\end{itemize}
\end{itemize}
In Lemma~\ref{lem:contractible}(a) we have already shown that each element $\Merge(\Uc;A_i,A_j)$ of $\Ifr(\Uc)$ corresponds to an element of $\Ifr(\Uc^{(1)})\cup \Ifr(\Uc^{(2)})$.
\end{proof}

\medskip

\subsection{Incidence function}
Signs of the summands in the image of the basis elements under differential maps can be read from incidence functions as follows.

Assume $\Uc \in \Sf_{k}(G,q)$ for $3 \leq k \leq n$. For $\Wc \in \Bf(\Uc)$ we want to define an {\em incidence value} $\epsilon(\Uc, \Wc) \in \{-1,+1\}$. For this we look at two set of natural permutations on parts of $\Uc$ and on parts of $\Wc$. Let  $\Wc=\Merge(\mathfrak{c}(\Uc); A_i, A_j)$.

\begin{itemize}
\item Let $\delta(\Uc)=(A_1,A_2, \cdots, A_k)$ and $\delta(\Wc)=(A_{\ell_1},A_{\ell_2}, \cdots, A_i \cup A_j, \cdots , A_{\ell_{k-1}})$ denote the permutations corresponding to the fixed ordering of parts (as fixed by the choice of minimal representatives of the classes in $\Ef_{k}(G, q)$ with respect to $\prec_k$).
\item Let $\alpha(\Uc)=(A_i,A_j, A_{s_1},\cdots, A_{s_{k-2}})$ be an arbitrary permutation which fixes $A_i$ and $A_j$ at the beginning and is arbitrary otherwise. Correspondingly, we define $\alpha(\Wc)=(A_i \cup A_j, A_{s_1},\cdots, A_{s_{k-2}})$ compatible with $\alpha(\Uc)$.
\end{itemize}

\begin{Definition} \label{def:sign}
We define
\begin{itemize}
\item[(i)] $\epsilon(\Uc, \Wc) = \sgn(\delta(\Uc),\alpha(\Uc))\sgn(\delta(\Wc),\alpha(\Wc))$,
where $\sgn(\cdot , \cdot)$ denote the standard sign function for permutations.
\item[(ii)]  $\theta(\Uc,\Wc)=D(A_j,A_i)$.
\end{itemize}
\end{Definition}

The definition of  $\epsilon(\Uc, \Wc)$ is easily seen to be independent of the choice of $\alpha(\Uc)$, because if $\alpha(\Uc)$ is replaced with $\alpha'(\Uc)=(A_i,A_j, A_{t_1},\cdots, A_{t_{k-2}})$ then
\[
\sgn(\alpha'(\Uc), \alpha(\Uc))=\sgn(\alpha'(\Wc), \alpha(\Wc))=\sgn((A_{t_1},\cdots, A_{t_{k-1}}), (A_{s_1},\cdots, A_{s_{k-2}}))
\]
 and  $\epsilon(\Uc, \Wc)$ is multiplied by $\sgn(\alpha'(\Uc), \alpha(\Uc))^2=1$. It is easy to see that $\theta(\Uc,\Wc)$ is also well-defined, and is independent of the choice acyclic orientation on $\mathfrak{c}(\Uc)$ where $A_i$ and $A_j$ are mergeable.

\medskip

\begin{Proposition}\label{prop:sign}
Fix $3 \leq k \leq n$ and let $\Uc \in \Sf_{k}(G,q)$. For any $\Wc \in \Bf(\Uc)$ and $\Xc \in \Bf(\Wc)$ there exists a unique $\Wc' \in \Bf(\Uc)$ such that $\Xc \in \Bf(\Wc')$ and
\[
\epsilon(\Uc,\Wc)\epsilon(\Wc,\Xc) =-\epsilon(\Uc,\Wc')\epsilon(\Wc',\Xc)
\]
\[
\theta(\Uc,\Wc)+\theta(\Wc,\Xc) =\theta(\Uc,\Wc')+\theta(\Wc',\Xc)  \ .
\]
Moreover, if $\Wc \in \Ifr(\Uc)$ and $\Xc \in \Ifr(\Wc)$ we have $\Wc' \in \Ifr(\Uc)$ and $\Xc \in \Ifr(\Wc')$.

In particular we have
\[
\sum_{\Wc\in \Bf(\Uc)\atop \Xc\in \Bf(\Wc)} \epsilon(\Uc,\Wc) \epsilon(\Wc,\Xc) \xb^{\theta(\Uc,\Wc)} \xb^{\theta(\Wc,\Xc)}[\psi(\Xc)]=0
\]
and
\[
\sum_{\Wc\in \Ifr(\Uc)\atop \Xc\in \Ifr(\Wc)} \epsilon(\Uc,\Wc) \epsilon(\Wc,\Xc) \xb^{\theta(\Uc,\Wc)} \xb^{\theta(\Wc,\Xc)}[\psi(\Xc)]=0\ .
\]
\end{Proposition}

\begin{proof}
Let the connected parts of $\Uc$ be $A_\ell$ for $1 \leq \ell \leq k$. Assume that $\Wc=\Merge(\mathfrak{c}(\Uc); A_i, A_j)$ and $\Xc=\Merge(\mathfrak{c}(\Wc); B_r, B_s)$. Since connected parts $B_r$ and $B_s$ of $\Wc$ are among $A_\ell$ (for $\ell \ne i,j$) and $A_i \cup A_j$, we need to consider three cases.

\medskip

$\bullet$ $B_r=A_r$, $B_s=A_s$: In this case $A_r$ and $A_s$ are mergeable in $\mathfrak{c}(\Uc)$ and we let $\Wc'=\Merge(\mathfrak{c}(\Uc); A_r, A_s)$.
{Then clearly} $\Xc=\Merge(\mathfrak{c}(\Wc'); A_i, A_j)$.
It follows that
\[
\theta(\Uc,\Wc)+\theta(\Wc,\Xc) =\theta(\Uc,\Wc')+\theta(\Wc',\Xc)=D(A_j,A_i)+D(A_s,A_r) \ .
\]

{There is a unique $\Wc' \ne \Wc$ with this property because there are only two ways to merge $A_i$ with $A_j$ and $A_r$ with $A_s$ .}
 Let
\[\alpha(\Uc)=(A_i,A_j,A_r,A_s, \cdots) \quad , \quad  \alpha'(\Uc)=(A_r,A_s,A_i,A_j, \cdots) ,\]
\[\alpha(\Wc)=(A_i \cup A_j,A_r,A_s, \cdots) \quad , \quad \alpha'(\Wc')=(A_r \cup A_s,A_i,A_j, \cdots) ,\]
\[\beta(\Wc)=(A_r,A_s, A_i \cup A_j , \cdots) \quad , \quad \beta'(\Wc')=(A_i, A_j, A_r \cup A_s , \cdots) ,\]
\[\beta(\Xc)=(A_r \cup A_s, A_i \cup A_j , \cdots) \quad , \quad \beta'(\Xc)=(A_i \cup A_j, A_r \cup A_s , \cdots) .\]
 From Definition~\ref{def:sign} we know
\[
\epsilon(\Uc,\Wc)\epsilon(\Wc,\Xc) =\sgn(\delta(\Uc),\alpha(\Uc))\sgn(\delta(\Wc), \alpha(\Wc))  \sgn(\delta(\Wc),\beta(\Wc))\sgn(\delta(\Xc), \beta(\Xc))
\]
\[
\epsilon(\Uc,\Wc')\epsilon(\Wc',\Xc) =\sgn(\delta(\Uc),\alpha'(\Uc))\sgn(\delta(\Wc'), \alpha'(\Wc'))  \sgn(\delta(\Wc'),\beta'(\Wc'))\sgn(\delta(\Xc), \beta'(\Xc)).
\]
The result follows from
\[\sgn(\delta(\Wc), \alpha(\Wc))  \sgn(\delta(\Wc),\beta(\Wc))=\sgn(\alpha(\Wc), \beta(\Wc))=1 , \]
\[\sgn(\delta(\Wc'), \alpha'(\Wc'))  \sgn(\delta(\Wc'),\beta'(\Wc'))= \sgn(\alpha'(\Wc'), \beta'(\Wc'))=1 ,\]
\[\sgn(\delta(\Uc),\alpha(\Uc)) \sgn(\delta(\Uc),\alpha'(\Uc))=\sgn(\alpha(\Uc),\alpha'(\Uc))=1 ,\]
\[\sgn(\delta(\Xc), \beta(\Xc))\sgn(\delta(\Xc), \beta'(\Xc))=\sgn(\beta(\Xc)), \beta'(\Xc))=-1 .\]

\medskip

$\bullet$ $B_r=A_r$, $B_s=A_i \cup A_j$: There are two cases:

\medskip

(1)  $E(A_r,A_i) \ne \emptyset$ in which case we let $\Wc'=\Merge(\mathfrak{c}(\Uc); A_r, A_i)$, and we have $\Xc=\Merge(\mathfrak{c}(\Wc'); A_r\cup A_i, A_j)$,

(2)  $E(A_r,A_i) = \emptyset$ in which case we must have $E(A_r,A_j) \ne \emptyset$ and we let $\Wc'=\Merge(\mathfrak{c}(\Uc); A_r, A_j)$. We then have $\Xc=\Merge(\mathfrak{c}(\Wc'); A_i, A_r \cup A_j)$.

\medskip

In each case it follows that
\[
\theta(\Uc,\Wc)+\theta(\Wc,\Xc) =\theta(\Uc,\Wc')+\theta(\Wc',\Xc)=D(A_j,A_i)+D(A_i,A_r)+D(A_j , A_r) \ .
\]
There is a unique $\Wc' \ne \Wc$ with this property because there are only two ways to merge $A_i$, $A_j$ and $A_r$.

We now verify the equality for the incidence function $\epsilon$ in case (1). Let

\[\alpha(\Uc)=(A_i,A_j,A_r,\cdots) \quad , \quad  \alpha'(\Uc)=(A_r,A_i,A_j, \cdots) ,\]
\[\alpha(\Wc)=(A_i \cup A_j,A_r, \cdots) \quad , \quad \alpha'(\Wc')=(A_r \cup A_i,A_j, \cdots) ,\]
\[\beta(\Wc)=(A_r, A_i \cup A_j , \cdots) \quad , \quad \beta'(\Wc')=\alpha'(\Wc') ,\]
\[\beta(\Xc)=(A_r \cup A_i \cup A_j , \cdots) \quad , \quad \beta'(\Xc)=\beta(\Xc) .\]
 From Definition~\ref{def:sign} we know
\[
\epsilon(\Uc,\Wc)\epsilon(\Wc,\Xc) =\sgn(\delta(\Uc),\alpha(\Uc))\sgn(\delta(\Wc), \alpha(\Wc))  \sgn(\delta(\Wc),\beta(\Wc))\sgn(\delta(\Xc), \beta(\Xc)),
\]
\[
\epsilon(\Uc,\Wc')\epsilon(\Wc',\Xc) =\sgn(\delta(\Uc),\alpha'(\Uc))\sgn(\delta(\Wc'), \alpha'(\Wc'))  \sgn(\delta(\Wc'),\beta'(\Wc'))\sgn(\delta(\Xc), \beta'(\Xc)).
\]
The result follows from
\[\sgn(\delta(\Wc), \alpha(\Wc))  \sgn(\delta(\Wc),\beta(\Wc))=\sgn(\alpha(\Wc), \beta(\Wc))=-1 , \]
\[\sgn(\delta(\Wc'), \alpha'(\Wc'))  \sgn(\delta(\Wc'),\beta'(\Wc'))= \sgn(\alpha'(\Wc'), \beta'(\Wc'))=1 ,\]
\[\sgn(\delta(\Uc),\alpha(\Uc)) \sgn(\delta(\Uc),\alpha'(\Uc))=\sgn(\alpha(\Uc),\alpha'(\Uc))=1 ,\]
\[\sgn(\delta(\Xc), \beta(\Xc))\sgn(\delta(\Xc), \beta'(\Xc))=\sgn(\beta(\Xc)), \beta'(\Xc))=1 .\]

\medskip

For case (2) this verification is completely analogous. Let

\[\alpha(\Uc)=(A_i,A_j,A_r,\cdots) \quad , \quad  \alpha'(\Uc)=(A_r,A_j,A_i, \cdots) ,\]
\[\alpha(\Wc)=(A_i \cup A_j,A_r, \cdots) \quad , \quad \alpha'(\Wc')=(A_r \cup A_j,A_i, \cdots) ,\]
\[\beta(\Wc)=(A_r, A_i \cup A_j , \cdots) \quad , \quad \beta'(\Wc')=(A_i, A_r \cup A_j, \cdots) ,\]
\[\beta(\Xc)=(A_r \cup A_i \cup A_j , \cdots) \quad , \quad \beta'(\Xc)=\beta(\Xc) .\]
 From Definition~\ref{def:sign} we know
\[
\epsilon(\Uc,\Wc)\epsilon(\Wc,\Xc) =\sgn(\delta(\Uc),\alpha(\Uc))\sgn(\delta(\Wc), \alpha(\Wc))  \sgn(\delta(\Wc),\beta(\Wc))\sgn(\delta(\Xc), \beta(\Xc)),
\]
\[
\epsilon(\Uc,\Wc')\epsilon(\Wc',\Xc) =\sgn(\delta(\Uc),\alpha'(\Uc))\sgn(\delta(\Wc'), \alpha'(\Wc'))  \sgn(\delta(\Wc'),\beta'(\Wc'))\sgn(\delta(\Xc), \beta'(\Xc)).
\]
The result follows from
\[\sgn(\delta(\Wc), \alpha(\Wc))  \sgn(\delta(\Wc),\beta(\Wc))=\sgn(\alpha(\Wc), \beta(\Wc))=-1 , \]
\[\sgn(\delta(\Wc'), \alpha'(\Wc'))  \sgn(\delta(\Wc'),\beta'(\Wc'))= \sgn(\alpha'(\Wc'), \beta'(\Wc'))=-1 ,\]
\[\sgn(\delta(\Uc),\alpha(\Uc)) \sgn(\delta(\Uc),\alpha'(\Uc))=\sgn(\alpha(\Uc),\alpha'(\Uc))=-1 ,\]
\[\sgn(\delta(\Xc), \beta(\Xc))\sgn(\delta(\Xc), \beta'(\Xc))=\sgn(\beta(\Xc)), \beta'(\Xc))=1 .\]

$\bullet$ $B_s=A_r$, $B_r=A_i \cup A_j$: This case is proved precisely as the previous case by permuting the indices.

\medskip

In all cases  if $\Wc \in \Ifr(\Uc)$ and $\Xc \in \Ifr(\Wc)$ the constructed $\Wc'$ is in $\Ifr(\Uc)$.
\end{proof}

\medskip

%%%%%%%%%%%%%%%%%%%%%%%%%%%%%%%%%%%%%%%%%%%%%%%%%%%%%%%%%%%%%%%%%%%%%%%%%

\subsection{Differential maps and minimality of the free resolutions}

We are now ready to use \eqref{eq:map} and induction to give a precise description of the differential maps constructed in Theorem~\ref{thm:GB} (respectively, Remark~\ref{rmk:inithm}) for $I_G$ (respectively, $\ini(I_G)$). The minimality of the constructed resolutions follows from this explicit description, as no units appear in the described differential maps. To simplify the notation we use $\psi$ instead of $\psi_k$ for all $k$ (as defined in Theorem~\ref{thm:GB} and Remark~\ref{rmk:inithm}).

\begin{Theorem}\label{thm:betti_binom}
Let $k \geq 0$ and $\Uc \in \Sf_{k+2}(G,q)$.

For $I_G$ the differential maps given by \eqref{eq:map} are of the form
\[
\varphi_{k}([\psi(\Uc)])=\sum_{\Wc\in \Bf(\Uc)} \epsilon(\Uc,\Wc) \xb^{\theta(\Uc,\Wc)}[\psi(\Wc)]\ .
\]

In particular, the set $\psi(\Sf_{k+2}(G, q))$ minimally generates $\syz_{k}(\Gb(G,q))$.
\end{Theorem}

\begin{Remark} \label{rmk:betti_ini}
For $\ini(I_G)$ the differential maps given by \eqref{eq:map} and the initial condition described in Remark~\ref{rmk:inithm} are of the form
\[
\varphi_{k}([\psi(\Uc)])=\sum_{\Wc\in \Ifr(\Uc)} \epsilon(\Uc,\Wc) \xb^{\theta(\Uc,\Wc)}[\psi(\Wc)]\ .
\]
The proof is completely analogous to the binomial case, and is skipped here. We will not use this description when we discuss the Betti numbers because we instead appeal to Theorem~\ref{thm:GBini}.
\end{Remark}

\begin{proof}
For $\Uc \in \Sf_{k+2}(G, q)$, from \eqref{eq:map}, we need to show that

\begin{equation} \label{eq:inmap}
s([\psi(\Uc^{(1)})],[\psi(\Uc^{(2)})])=\sum_{\Wc\in \Bf(\Uc)} \epsilon(\Uc,\Wc) \xb^{\theta(\Uc,\Wc)}[\psi(\Wc)]\ .
\end{equation}

Note that this would prove the minimality of the resolution because $\Wc$ is a connected flag and therefore $\theta(\Uc,\Wc) \ne 0$.

The proof is by induction on the number of vertices of the graph. The result is obvious for a graph with $2$ vertices. Suppose the result holds for graphs with less than $n$ vertices, and consider the graph $G$ with $n$ vertices. We need to show that all maps $\varphi_k$ (for $0 \leq k \leq n-2$) are of the form \eqref{eq:inmap}. Fix a $\Uc \in \Sf_{k+2}(G,q)$. We consider two cases:

\medskip

$\bullet$ $0 \leq k < n-2$.
We consider the graph $G_{/\Uc}$ on the vertex set $\{u_1,u_2,\ldots,u_{k+2}\}$. This graph has fewer than $n$ vertices because $k +2< n$.
Let $\Uc' \in \Sf_{k+2}(G_{/\Uc},u_1)$ be the inverse image of $\Uc$ under the map $\phi^{\ast}$ as described  in Remark~\ref{rem:contract}.
By the induction hypothesis we have
\begin{equation} \label{induc-claim}
\psi(\Uc')=\sum_{\Wc' \in \Bf(\Uc')} \epsilon(\Uc',\Wc')\xb^{\theta(\Uc',\Wc')}[\psi(\Wc')] \ .
\end{equation}

For each $\Wc'$, let $\Wc$ be the inverse image under the map $\phi^{\ast}$. Then $\theta(\Uc,\Wc)$ and $\theta(\Uc',\Wc')$ are related by the map $\phi_{\ast}$ of Remark~\ref{rem:unique divisor}, and there is a one-to-one correspondence between elements of $\mathcal{B}(\Uc')$ and elements of $\mathcal{B}(\Uc)$. We claim that
\begin{equation}\label{our claim}
\psi(\Uc)=\sum_{\Wc \in \Bf(\Uc)}\epsilon(\Uc,\Wc)\xb^{\theta(\Uc,\Wc)}[\psi(\Wc)] \ ,
\end{equation}
where $c(\Wc)=c(\Wc')$.
To see this, we need to show that
\begin{equation} \label{our claim2}
\sum_{\Wc \in \Bf(\Uc)} \epsilon(\Uc,\Wc)\xb^{\theta(\Uc,\Wc)}\psi(\Wc) = 0 \ .
\end{equation}

First note that \eqref{induc-claim} is equivalent to
\begin{equation} \label{our claim 22}
\sum_{\Wc' \in \Bf(\Uc')} \epsilon(\Uc',\Wc')\xb^{\theta(\Uc',\Wc')}\psi(\Wc') = 0 \ .
\end{equation}

We again use the induction hypothesis for $\psi(\Xc')$ to write
\begin{equation}\label{contraction maps}
\sum_{\Wc' \in \Bf(\Uc') \atop \Xc' \in \Bf(\Wc')} \epsilon(\Uc',\Wc')\epsilon(\Wc',\Xc')\xb^{\theta(\Uc',\Wc')} \xb^{\theta(\Wc',\Xc')}  [\psi(\Xc')] = 0 \ .
\end{equation}

Equation (\ref{contraction maps}) implies that
corresponding to each term $\xb^{\theta(\Uc',\Wc'_1)} \xb^{\theta({\Wc'_1},\Xc'_1)}[\psi(\Xc'_1)]$ there exists a term
$\xb^{\theta(\Uc',\Wc'_2)} \xb^{\theta(\Wc'_2,\Xc'_2)}[\psi(\Xc'_2)]$ with opposite sign, with which it cancels. Therefore
\begin{equation}\label{E}
E':=\theta(\Uc',\Wc'_1)+\theta(\Wc'_1,\Xc'_1)=\theta(\Uc',\Wc'_2)+\theta(\Wc'_2,\Xc'_2) \quad \text{and} \quad [\psi(\Xc'_1)]=[\psi(\Xc'_2)] \ .
\end{equation}
Since $\psi$ is injective (Theorem~\ref{thm:GB}) we get $\Xc'_1=\Xc'_2$.
By the induction hypothesis we know
\[
\theta(\Uc',\Wc'_1)=D'(u_i,u_j)
\]
\[
\theta({\Wc'_1},\Xc'_1) =
\begin{cases}
D'(u_r,u_{s}) &\text{or}\\
D'(\{u_i,u_j\},u_r)&\text{or}\\
D'(u_r,\{u_i,u_j\})\\
\end{cases}
\]
for some distinct $i,j,r,s$ (depending on what parts of $\Wc'$ are merged to get $\Xc'_1=\Xc'_2$).

Since $E'$ can be written as a sum of $D'(u_a, u_b)$'s, once we recognize all $a$'s and $b$'s appearing in the sum, we can use Remark~\ref{rem:unique divisor} and lift $E'$ to some $E$ as a sum of  $D(U_a \backslash U_{a-1}, U_b \backslash U_{b-1})$'s. Then the same cancellations as in \eqref{our claim 22} occur in the left-hand side of \eqref{our claim2} and we get zero.

\medskip

$\bullet$ If $\Wc'_1=\Wc'_2$, then $\theta(\Uc',\Wc'_1)=\theta(\Uc',\Wc'_2)$ and it follows from \eqref{E} that $\theta(\Wc'_1,\Xc'_1)=\theta(\Wc'_2,\Xc'_2)$. By looking at $D'(u_i,u_j)$ we can recognize $u_i$. By looking at the unique part in $\Wc'_1=\Wc'_2$ that contains two elements, we recognize $u_j$. By looking at the vertices where $\theta({\Wc'_1},\Xc'_1)=\theta({\Wc'_2},\Xc'_2)$ is nonzero we can recognize $u_r$ since $\Xc'_1=\Xc'_2$.

\medskip

$\bullet$ If $\Wc'_1 \ne \Wc'_2$, then we consider the following cases:

\medskip

\begin{itemize}
\item[(1)] $\theta({\Wc'_1},\Xc'_1)=D'(u_r,u_s)$: we have  $E'=D'(u_i,u_j)+D'(u_r,u_s)$. The places where $E'$ is nonzero determine $\{u_i, u_r\}$. By looking at the two parts in $\Xc'_1=\Xc'_2$ which contain precisely two vertices, we can distinguish $\{\{u_i,u_j\},\{u_r,u_s\}\}$.

\item[ ]

\item [(2)] $\theta({\Wc'_1},\Xc'_1)=D'(\{u_i,u_j\},u_r)$: we have $E'=D'(u_i,u_j)+D'(u_i,u_r)+D'(u_j,u_r)$. Since we know $\Xc'_1=\Xc'_2$ we know $\{u_i,u_j,u_r\}$.
\begin{itemize}
\item if $u_i$ and $u_r$ are not adjacent: then $u_r$ is the vertex where $E'$ is zero. The vertex $u$ such that $E'(u)$ is equal to the number of edges between $u$ and $\{u_i,u_j,u_r\} \backslash u$ is $u_i$. The other vertex where $E'$ is nonzero is $u_j$.

\item if $u_j$ and $u_r$ are not adjacent: then $u_i$ is the unique vertex where $E'$ is nonzero. We do not need to distinguish between $u_j$ and $u_r$ because $E'$ is of the form $E'=D'(u_i,u_j)+D'(u_i,u_r)$.

\end{itemize}

\item[ ]

\item [(3)] $\theta({\Wc'_1},\Xc'_1)=D'(u_r,\{u_i,u_j\})$: we have $E'=D'(u_i,u_j)+D'(u_r,u_i)+D'(u_r,u_j)$. Since we know $\Xc'_1=\Xc'_2$ we know $\{u_i,u_j,u_r\}$. This case reduces to (2) by permuting the indices.

\end{itemize}

%%%%%%%%%%%%%%%%%%%%%%%%%%%%%%%%%%%%%%%%%%

Therefore \eqref{our claim2} holds. {Note that $\Uc^{(1)}$ and $\Uc'^{(1)}$ (respectively, $\Uc^{(2)}$ and $\Uc'^{(2)}$) are related by the map $\phi^{\ast}$ \eqref{phistar}.} On the other hand by Remark~\ref{rem:contract} for each $\ell$ the total ordering $\prec_\ell$ corresponding to $G$ and the total ordering $\prec'_\ell$ corresponding to $G/\Uc$ (and so the term orderings $<_\ell$ and $\ell'$) are compatible.
Hence it follows from the discussion above and Remark~\ref{rem:contract} that \eqref{our claim2} is precisely coming from the $s$-polynomial computation.

\medskip

$\bullet$ $k=n-2$.
Let $U_i\backslash U_{i-1}=\{v_{i}\}$ and for simplicity, $x_{v_i}=x_i$.
By Theorem~\ref{thm:GB} $\psi(\Uc)=s(\psi(\Uc^{(1)}),\psi(\Uc^{(2)}))$. We directly apply the division algorithm to describe $s(\psi(\Uc^{(1)}),\psi(\Uc^{(2)}))$.

From the proof of Lemma~\ref{lem:KU1U2} the coefficient of $[\psi(\Uc^{(1)})]$ is $\xb^{\theta(\Uc,\Uc^{(1)})}=\xb^{D(v_2,v_1)}$ and the coefficient of $[\psi(\Uc^{(2)})]$ is $\xb^{\theta(\Uc,\Uc^{(2)})}$, where
\[
\theta(\Uc,\Uc^{(2)})=\begin{cases}
D(v_3,v_2), &\text{if $v_2$ and $v_3$ are adjacent;}\\
D(v_3,v_1), &\text{if $v_2$ and $v_3$ are not adjacent.}
\end{cases}
\]

Now assume that
\[
M:=\LM(\spoly(\psi(\Uc^{(1)}),\psi(\Uc^{(2)})))\ .
\]
Since $\Image(\psi)$ forms a minimal Gr\"obner bases of $(\syz_{n-3}(\Gb(G,q)),<_{n-3})$ there exists an element
$\Vc\in \Sf_{n-1}(G, q)$ such that $\LM(\psi(\Vc))$ divides $M$.
We know from \eqref{eq:LMpsi} that $\LM(\psi(\Vc))=\xb^{D(V_{2}\backslash V_1,V_1)}[\psi(\Vc^{(1)})]$. Hence

\[
M=\xb^{\theta(\Uc,\Vc)+D(V_{2}\backslash V_1,V_1)}[\psi(\Vc^{(1)})] \quad \text{for some} \quad \theta(\Uc,\Vc) \geq 0 \ .
\]

In the first step of the division algorithm we obtain
\[
\xb^{\theta(\Uc,\Uc^{(1)})}\psi(\Uc^{(1)})-\xb^{\theta(\Uc,\Uc^{(2)})}\psi(\Uc^{(2)}) + \epsilon(\Uc,\Vc) \xb^{\theta(\Uc,\Vc)}\psi(\Vc)\ .
\]
The division algorithm proceeds by finding the leading monomial of the above expression and continuing similarly.
To show \eqref{eq:inmap} we show that in each step $\Vc$ belongs to $\Bf(\Uc)$ and $\theta(\Uc,\Vc)$ is of the form $D(v,v')$ for some vertices $v$ and $v'$ ({and that all such terms appear}). This is clearly correct for $\Vc = \Uc^{(1)}$ and $\Vc = \Uc^{(2)}$ as discussed above. Assume we are in $i^{\rm th}$ step. Recursively we may assume the leading monomial of the existing expression is a monomial of $\xb^{D(v_i,v_j)}[\psi(\Wc)]$.

Since $\Wc \in \Sf_{n-1}(G,q)$ we can use the result of the previous case for $k=n-3<n-2$ to write

\[
\psi(\Wc)=\sum_{\Xc \in \Bf(\Wc)} \epsilon(\Wc,\Xc)\xb^{\theta(\Wc,\Xc)}[\psi(\Xc)]\ ,
\]
where $\theta(\Wc,\Xc)=D(W_r\backslash W_{r-1},W_{s}\backslash W_{s-1})$ for some $r,s$.

\medskip

Now let $M$ be the term $\xb^{D(v_i,v_j)}\xb^{\theta(\Wc,\Xc)}[\psi(\Xc)]$ which is divisible by
\[
\LM(\psi(\Vc))=\xb^{D(V_2\backslash V_1,V_1)}[\psi(\Vc^{(1)})]\ .
\]
Set $E':=D(v_i,v_j)+\theta(\Wc,\Xc)=D(v_{i}, v_j)+D(W_r\backslash W_{r-1},W_s\backslash W_{s-1})$. Let $\{v,v'\}$ be the unique part of $\Vc$ which contains two vertices. It is enough to show that
$E'=D(V_2\backslash V_1,V_1)+D(v, v')$ and $\Vc\in \Bf(\Uc)$.
Depending on which parts of $\Wc$ are merged to get $\Xc$ we have the following cases:

\medskip

Case 1. $\theta(\Wc,\Xc)=D(v_{r},v_s)$: {note that $\{v_r,v_s\}$ and $\{v_i,v_j\}$ are two disjoint parts of $\Xc$.}
Since $V_2$ has at least two elements we have $V_2=\{v_i, v_j\}$ or $V_2=\{v_r, v_s\}$. Now by noting that $v_1\in V_2$ we can recognize $V_2$. With no loss of generality assume that
$V_2=\{v_i, v_j\}$. Then we must have
$v_j=v_1$; since $\xb^{D(V_{2}\backslash V_1,V_1)}$ divides $\xb^{E'}$. Therefore $E'-D(V_2\backslash V_1,V_1)=D(v_r,v_s)$. Note that $\Xc=\Vc^{(1)}$ implies that
$\{v_r, v_s\}$ is the unique part of $\Vc$ containing two vertices and so  $D(v_r,v_s)=\theta(\Uc,\Vc)$.

\medskip

%%%%%%%%%%%%%%%%%%%%%%%%%%%%%%%%%%%%%

Case 2. $\theta(\Wc,\Xc)=D(v_r,\{v_i,v_j\})$: {the ordered collection of $(X_i\backslash X_{i-1})_{i=1}^{n-2}$ is a permutation of the sets $\{v_t\}_{t\neq i,j,r}$ and $\{v_i,v_j,v_r\}$, since $\Wc$ is obtained from $\Uc$ by merging $v_i,v_j$ and $\Xc$ is obtained from $\Wc$ by merging $v_r,\{v_i,v_j\}$ where $v_r$ appears before $\{v_i,v_j\}$ in $\Wc$.}
So we have $V_2=\{v_i,v_j,v_r\}$ and $v_1\in V_2$. Therefore the following cases may occur:
\begin{itemize}
\item $V_1=\{v_i,v_j\}$: then $D(V_2\backslash V_1,V_1)={D(v_r,\{v_i,v_j\})}$ and $E'-D(V_2\backslash V_1,V_1)={D(v_{i}, v_j)}$. Therefore $\theta(\Uc,\Vc)=D(v_{i}, v_j)$.

\item $V_1=\{v_j,v_r\}$: then $D(V_2\backslash V_1,V_1)={D(v_{i},\{v_j,v_r\})}$ and $v_j$ is adjacent to $v_r$.
On the other hand, since $\LM(\psi(\Vc))$ divides $M$ we should have no edge between $v_i$ and $v_r$. Therefore $E'-D(V_2\backslash V_1,V_1)={D(v_r, v_j)}$ which is equal to $\theta(\Uc,\Vc)$.

\item $V_1=\{v_j\}$: then $D(V_2\backslash V_1,V_1)={D(\{v_i,v_r\},v_j)}$ and $v_i$ is adjacent to $v_r$.
Therefore $\theta(\Uc,\Vc)=D(v_r,v_i)$.

\item $v_j\in V_2\backslash V_1$ and $v_i\in V_1$: the fact that $v_i$ and $v_j$ are adjacent implies that $x_j$ divides $M$ which is impossible since $\theta(\Uc,\Wc)(v_j)=0$.

\item $V_1=\{v_r\}=\{v_1\}$: then $D(V_2\backslash V_1,V_1)=D(\{v_j,v_{i}\},v_r)$ and two vertices $v_i$ and $v_r$ are adjacent.
The number of edges between $v_i$ and $v_r$ is less than the number of edges between $v_i$ and $v_j$.
On the other hand, since $\LM(\psi(\Vc))$ divides $M$ we should have no edge between $v_j$ and $v_r$.

Now we will show that this case cannot happen. First we note that $M$ is equal to the term $\xb^{D(v_1,v_i)}\xb^{\theta(\Wc',\Xc)}[\psi(\Xc)]$ for some $\Wc'\in\Bf(\Uc)$ since $M$ is a term corresponding to a summand of an element which was added in the previous steps of the division algorithm. This term is obtained by merging the parts $\{v_1\}$ and $\{v_i\}$, i.e., removing the orientations on the edges between $v_1$ and $v_i$. Thus
$\{v_1,v_i\}$ is
the unique part of $\Wc'$ which contains two vertices. This implies that $\xb^{\theta(\Uc,\Wc')}[\psi(\Wc')]$ has not been added in the previous steps of the division algorithm. Otherwise this term has been canceled with its corresponding dual term coming from Proposition~\ref{prop:sign}. On the other hand $M'=\xb^{\theta(\Uc,\Wc')}\LM([\psi(\Wc')])$ is not among the previous terms added in the division algorithm till this step (since otherwise $M'$ could be the leading term). Note that $W'_1=\{v_1,v_i\}$. Now assume that $W'_2\backslash W'_1=\{v_s\}$. If $v_1$ and $v_s$ are adjacent then the dual element of $M'$ coming from Proposition~\ref{prop:sign}, is obtained by merging the parts $\{v_1\}$ and $\{v_s\}$ (in order to get $\Wc''\in \mathcal{A}(\Uc)$) and then merging the parts $\{v_i\}$ and $\{v_1,v_s\}$ in $\Wc''$. Note that by Lemma~\ref{lem:contractible} we know that $\{v_1\}$ and $\{v_s\}$ are mergeable in $G(\Uc^{(1)})$ or $G(\Uc^{(2)})$ which implies that $\xb^{\theta(\Uc,\Wc'')}[\psi(\Wc'')]$ has been already added in the previous steps of the division algorithm. However the term $M'$ of this element has not been canceled with its dual term which is a summand of $\xb^{\theta(\Uc,\Wc')}[\psi(\Wc')]$ which is a contradiction since $M'>_{\rm revlex}M$.
\end{itemize}

\medskip

Case 3. $\theta(\Wc,\Xc)=D(\{v_i,v_j\},v_r)$: we have $E'=D(v_i,v_j)+D(v_i,v_r)+D(v_j,v_r)$. This case reduces to (2) by changing the indices.

\medskip

\medskip

As we see in all cases we add the term $\xb^{\theta(\Uc,\Vc)}\psi(\Vc)$ to the $s$-polynomial which has the desired properties.

By Proposition~\ref{prop:sign} corresponding to each summand
$\xb^{\theta(\Uc,\Wc)}\xb^{\theta(\Wc,\Xc)}[\psi(\Xc)]$ where $\Wc \in \Bf(\Uc)$ and $\Xc \in \Bf(\Wc)$ there exists a unique $\Wc' \in \Bf(\Uc)$ such that $\Xc \in \Bf(\Wc')$ and so there exists a unique
element $\xb^{\theta(\Uc,\Wc')}\xb^{\theta(\Wc',\Xc)}[\psi(\Xc)]$ with different sign. On the other hand, by Lemma~\ref{lem:U1,U2}
all terms coming from elements of $\Ifr(\Uc)$ are appearing in $s(\psi(\Uc^{(1)}),\psi(\Uc^{(2)}))$. By Lemma~\ref{lem:contractible}(b)
corresponding to each element $\Merge(o_j(\Uc);\{v_i\},\{v_j\})$ of $\Bf(\Uc)$ there exists a $\Wc\in \Ifr(\Uc)$ with a mergeable edge corresponding to $E(v_i,v_j)$. Our previous argument implies that the term associated to $\Wc$ has been already added in the division algorithm.
Therefore all terms corresponding to elements of $\Bf(\Uc)$
are added in the division algorithm as well, which completes the proof.
\end{proof}

\begin{Corollary}
The Betti numbers of the ideals $I_G$ and $\ini(I_G)$ are independent of the characteristic of the base field $K$.
\end{Corollary}

%%%%%%%%%%%%%%%%%%%%%%%%%%%%%%%%%%%%%%%%%%%%%%%%%%%%%%%%%%%%%%%%

%%%%%%%%%%%%%%%%%%%%%%%%%%%%%%%%%%%%%%%%%%%%%%%%%%%%%%%%%%%%%%%%%%%%%%%%%

\begin{Remark} \label{thm:betti}
Note that in the proof of Theorem~\ref{thm:GB} we actually construct a free resolution for $I_G$ according to Algorithm~\ref{alg:schr}. Once we know that this resolution is minimal it follows that $\Gb_k(G,q):=\Image(\psi_k)$ indeed forms a {\em minimal generating set} of $\syz_{k}(\Gb(G,q))$ by Remark~\ref{minfreegen}. Then Theorem~\ref{thm:GBini} implies that the same statement is true for $\ini(I_G)$ as well.
\end{Remark}

\begin{Remark}
As described in \S\ref{sec:Introduction}, the constructed minimal free resolutions are in fact supported on a cellular complex. In \cite{FarbodFatemeh2} we describe this geometric picture in detail.
\end{Remark}

The following example sums up all our notions by giving the explicit minimal free resolution for our running example, the $4$-cycle graph.
\begin{Example}

Returning to Example~\ref{exam:C4}, (and Examples \ref{exam:I(U)}, \ref{exam:running example}) by Theorems~\ref{thm:GB} and \ref{thm:betti} we have that

\[
\begin{aligned}
\varphi_{2}([\psi(\Uc)])&=x_2[\psi(\Merge(\Uc; 1, 2))]-x_3[\psi(\Merge(\Uc; 1,3))]+x_4[\psi(\Merge(\Uc; 3, 4))]\\&-x_4[\psi(\Merge(\Uc; 2,4))]
+x_1[\psi(\Merge(\mathfrak{c}(\Uc); 2, 1))]-x_1[\psi(\Merge(\mathfrak{c}(\Uc); 3, 1))]\\&-x_2[\psi(\Merge(\mathfrak{c}(\Uc); 4, 2))]+
x_3[\psi(\Merge(\mathfrak{c}(\Uc); 4, 3))] \ .
\end{aligned}
\]
Moreover for the ideal $I_{C_4}$ the minimal free resolution $R/I_{C_4}$ is
\[
0\rightarrow R(-4)^3 \xrightarrow{\varphi_2} R(-3)^8 \xrightarrow{\varphi_1} R(-2)^6\xrightarrow{\varphi_0} R\ .
\] 

The matrix for the first differential map is
\[
\varphi_0:\ \ \left( \begin{array}{cccccc}
  x_3x_4-x_1x_2 & x_2x_4-x_1x_3 & x_{2}x_{3}-x_1^2 & x_{4}^{2}-x_2x_3 & x_{3}^{2}-x_1x_4 & x_{2}^{2}-x_1x_4 \end{array}\right)
\]
in which the columns correspond to the generators of $I_{C_4}$ listed in the same order $x_3x_4-x_1x_2, x_2x_4-x_1x_3,\ldots,x_{2}^{2}-x_1x_4$.
\end{Example}
The second differential map is presented by the matrix
\[
\varphi_1:\ \ \left( \begin{array}{cccccccc}
   - x_{4} & 0 & x_{2} & 0 &  - x_{3} & 0 & 0 &  - x_{1} \\
  0 &  - x_{4} & 0 & x_{3} & 0 &  - x_{2} &  - x_{1} & 0 \\
  0 & 0 &  - x_{4} &  - x_{4} & 0 & 0 &  - x_{3} &  - x_{2} \\
  x_{3} & x_{2} & 0 & 0 &  - x_{1} &  - x_{1} & 0 & 0 \\
   - x_{2} & 0 & 0 &  - x_{1} & x_{4} & 0 & x_{2} & 0 \\
  0 &  - x_{3} &  - x_{1} & 0 & 0 & x_{4} & 0 & x_{3} \end{array}\right)
\]
where the columns of the above matrix correspond to the bases elements associated to connected flags
\[
\Merge(\Uc; 1, 2), \Merge(\Uc; 1,3), \Merge(\Uc; 3, 4), \Merge(\Uc; 2,4)
\] 
listed in Example~\ref{exam:I(U)} and the connected flags
\[ \Merge(\mathfrak{c}(\Uc);2, 1), \Merge(\mathfrak{c}(\Uc); 3, 1), \Merge(\mathfrak{c}(\Uc); 4, 2), \Merge(\mathfrak{c}(\Uc); 4, 3)\]
listed in Example~\ref{exam:running example}.

The last differential map is presented by the matrix
\[
\varphi_2:\ \ \left( \begin{array}{ccc}
  x_{2} & 0 & -x_1 \\
   - x_{3} &  - x_{3} & 0 \\
  x_4 & 0 & 0 \\
   - x_{4} &  0 & 0 \\
  x_1 &   x_{2} & 0 \\
  -x_1 & 0 &  x_{3} \\
  -x_2 &  - x_{4} & x_2 \\
  x_3 & x_3 & -x_{4} \end{array}\right)
\]
in which the first column corresponds to \[\Uc: \{1\}\subset\{1,2\}\subset\{1,2,3\}\subset\{1,2,3,4\}, \]
the second and third columns correspond to
\[\Uc_2:\ \{1\}\subset\{1,2\}\subset\{1,2,4\}\subset\{1,2,3,4\}\ {\rm and }\ \Uc_3:\ \{1\}\subset\{1,3\}\subset\{1,3,4\}\subset\{1,2,3,4\}.\] Their corresponding acyclic orientations are listed in Figure~3.

%%%%%%%%%%%%%%%%%%%%%%%%%%%%%%%%%%%%%%%%%%%%%%%%%%%%%%%%%%%%%%%%%%%%%%%%%%%%%%%%%%%%%%%%%%%%%%%%%%%%%%%%%%%%%%%

%%%%%%%%%%%%%%%%%%%%%%%%%%%%%%%%%%%%%%%%%%%%%%%%%%%%%%%%%%%%%%%%%%%%%%%%%%%%%%%%%%%%%%%%%%%%%%%%%%%%%%%%%%%%%%%
%%%%%%%%%%%%%%%%%%%%%%%%%%%%%%%%%%%%%%%%%%%%%%%%%%%%%%%%

\section{Betti numbers} \label{sec:betti}

Let $\As=\ZZ$ or $\As=\Pic(G)$. We let $\beta_{i}$ and $\beta_{i, \js}$ denote $\beta_{i}(R/I_G)$ and $\beta_{i,\js}(R/I_G)$ respectively
(for $ i \geq 0$ and $\js \in \As$). Note that by Remark~\ref{thm:betti} one might replace $I_G$ with $\ini(I_G)$.

\begin{Proposition}\label{prop:Betti_I}
For all $i \geq 0$,
$
\beta_{i}=|\Sf_{i+1}(G, q)|=|\Ef_{i+1}(G, q)| \ .
$
\end{Proposition}

\begin{proof}
The assertion follows by Theorem~\ref{thm:GB}, Remark~\ref{rmk:inithm}, and Theorem~\ref{thm:betti_binom} and the fact that $\beta_{i}(R/\ini(I_G))=\beta_{i-1}(\ini(I_G))$.
\end{proof}

\begin{Remark}
It follows from Proposition~\ref{prop:Betti_I} that $|\Sf_{i+1}(G, q)|$ is independent of $q$. It is a nice combinatorial exercise to show this directly.
\end{Remark}

Recall from \S\ref{sec:grade} that for $D\in \Div(G)$ we define
\[
\deg_{\As}(D)=\deg_{\As}(\xb^D)=
\begin{cases}
\deg(D), &\text{ if $\As=\ZZ$;}\\
D, &\text{ if $\As=\Div(G)$;}\\
[D], &\text{ if $\As=\Pic(G)$.}
\end{cases}
\]

\begin{Definition}
 For $k \geq 1$ and $\js \in \As$ define
\[
\Sf_{k,\js}(G, q)= \{ \Uc \in \Sf_{k}(G, q) : \, \deg_{\As}(D(\Uc))=\js\}
\]
where $D(\Uc)$ is defined in Definition~\ref{def:Du}.
\end{Definition}

\medskip

We now strengthen Proposition~\ref{prop:Betti_I} as follows.
\begin{Proposition}\label{prop:GBetti_I}
For $\As=\ZZ$ or $\As=\Pic(G)$
\[
\beta_{i,\js}=|\Sf_{i+1,\js}(G, q)|
\]
for all $i \geq 0$ and $\js \in \As$.
\end{Proposition}

\begin{proof}

By Theorem~\ref{thm:GB}, Theorem~\ref{thm:betti_binom}, and Remark~\ref{minfreegen} the set $\psi_i(\Sf_{i+2}(G, q))$ minimally generates the module $\syz_i(\ini(I_G))$ for each $i \geq 0$, and we have
\[
\beta_{i,\js}=\beta_{i-1,\js}(I_G)=|\{\psi_{i-1}(\Uc) :\, \deg(\psi_{i-1}(\Uc))=\js\, \text{ for }\, \Uc \in \Sf_{i+1,\js}(G, q)\}|\ .
\]

We first note that for $\Uc \in \Sf_{i+1,\js}(G, q)$ we have
\[D(\Uc)=\sum_{\ell=1}^{i+1}{D(U_{\ell} \backslash U_{\ell-1}, U_{\ell-1})} \quad , \quad \xb^{D(\Uc)}=\prod_{\ell=1}^{i+1}\xb^{D(U_{\ell}\backslash U_{\ell-1},U_{\ell-1})}\]
and
\begin{equation}\label{deg(D)}
\deg_{\As}(D(\Uc))=\deg_{\As}(\xb^{D(\Uc)})=\sum_{i=1}^{k}\deg_{\As}(\xb^{D(U_{i}\backslash U_{i-1},U_{i-1})}) \ .
\end{equation}

We need to show that $\deg(\psi_{i-1}(\Uc))=\deg_{\As}(D(\Uc))$.
The proof is by induction on $i\geq 0$. For $i=0$ there is nothing to prove. Since $\psi_{i-1}(\Uc)$ is homogeneous, by \eqref{eq:LMpsi} and \eqref{deg(D)} we obtain
\[
\begin{aligned}
\deg_\As(\psi_{i-1}(\Uc))&=\deg_\As(\LM(\psi_{i-1}(\Uc)))\\
&=\deg(\xb^{D(U_2\backslash U_1, U_1)}[\psi_{k-1}(\Uc^{(1)})])\\
&=\deg_\As(\xb^{D(U_2\backslash U_1, U_1)})+\deg_\As(\psi_{k-1}(\Uc^{(1)}))\\
&=\deg_\As(\xb^{D(U_2\backslash U_1, U_1)})+\deg_\As(D(\Uc^{(1)}))\\
&=\deg_\As(D(\Uc))\ .
\end{aligned}
\]\qed

\end{proof}

\begin{Example}\label{ex:acyc}
It follows from above descriptions that for the $\ZZ$-grading of $R$, $\beta_{i,j}$ can take nonzero values only if $0 \leq i \leq n-1$ and $0 \leq j \leq m$ where $n=|V(G)|$ and $m=|E(G)|$. Clearly $\beta_0=1$. Moreover $\beta_{n-1}=\beta_{n-1,m}$ which is equal to the number of acyclic orientations of $G$ with unique source at $q$ (see also Lemma~\ref{lem:deltaD}). Since both $R/I_G$ and $R/\ini(I_G)$ are Cohen-Macaulay (see, e.g., \cite[Section~11.1]{FarbodFatemeh2}), it follows that the {\em Castelnuovo-Mumford regularity} of both $R/I_G$ and $R/\ini(I_G)$ is equal to $g=m-n+1$ (see, e.g., \cite[page 69]{EisenbudSyz}). 
\end{Example}
\begin{Example}\label{exam:complete}
Let $G=K_n$ be the complete graph on $n$ vertices. Let $\{A_1, A_2, \cdots , A_k\}$ be any $k$-partition of $V(G)$ with $q\in A_1$.
Then corresponding to each permutation $\delta=(i_1,i_2,\ldots,i_{k-1})$ of $(2,3,\ldots,n)$ the strictly increasing $k$-flag
\[
\Uc_\delta:\ \, U_1\subsetneq U_2\subsetneq \cdots\subsetneq U_k=V(G)
\]
is an element of $\Sf_{k}(G,q)$ where $U_j=A_1\cup A_{i_1}\cup\cdots\cup A_{i_{j-1}}$ for each $j$.
Therefore $|\Sf_{k}(G,q)|=(k-1)! \, S(n,k)$ where $S(n,k)$ denotes the Stirling number of the second kind (i.e. the number of ways to partition a set of $n$ elements into $k$ nonempty subsets). In other words $\beta_i= i! \, S(n,i+1)$. See http://oeis.org/A028246 for other interpretations of these numbers.
\end{Example}

\begin{Example}\label{exam:tree}
Let $G$ be a tree on $n$ vertices. Let $\Uc\in \Sf_{k}(G,q)$. For each $i>1$ the part $U_i\backslash U_{i-1}$ is connected by exactly one edge to only one part $U_j\backslash U_{j-1}$ with $j<i$; otherwise we get a cycle in the graph. Therefore each element $\Uc \in \Sf_{k}(G,q)$ is determined by the $k-1$ edges (of $n-1$ edges of $G$) between the partitions $(U_i\backslash U_{i-1})$'s of $\Uc$ and $|\Sf_{k}(G,q)|={n-1\choose k-1}$. The fact that each edge contributes $1$ to the degree of $\psi(\Uc)$ means that
$\beta_{i}=\beta_{i,i}={n-1 \choose i}$.
\end{Example}

\begin{Example}\label{prop:cycle}
Let $G=C_n$ be the cycle on $n$ vertices. For simplicity of notation let $V(G)=[n]$. Then we will show, by induction on $n$, that for $k \geq 2$
\[|\Sf_{k}(C_n,q)|=(k-1)\times{n\choose k}.\]
One can easily check the formula for $k=2$ and $k=3$. So we may assume that $k\geq 4$. Let $1\leq i_1<i_2<\cdots<i_k\leq n$. Then
we consider the parts
\[
A_1:=\{i_k,i_{k}+1,\ldots,n\}\cup\{1,\ldots,i_1-1\}\ \text{ and }\  A_{t+1}:=\{i_t,i_t+1,\ldots,i_{t+1}-1\}\ \text{ for }\ 1\leq t\leq k-1
\]
of the graph. Then there are three types of elements $\Uc$ of $\Sf_{k}(G,q)$ such that $(U_i\backslash U_{i-1})_{i=2}^k$ is a permutation of $A_2,\ldots,A_k$:
\begin{itemize}
\item[(1)] $U_1=A_1$ and $U_i=U_{i-1}\cup A_i$ for each $1<i\leq k$.

\item[(2)] $U_1=A_1$ and $U_i=U_{i-1}\cup A_{k-i+2}$ for each $1<i\leq k$.

\item[(3)] $U_1=A_1$, $U_2\backslash U_1=A_2$, $U_3\backslash U_2=A_k$: Then the number of $k$-connected flags of $C_n$ with this partition set is equal to the number of
$(k-3)$-connected flags of $C_{k-2}$ on the vertex set $u_1,u_2,\ldots,u_{k-2}$ where
$q=u_{k-3}$ is associated to $A_1\cup A_2\cup A_k$ and $u_i$ is associated to the part $A_i$ for each $3\leq i\leq k-1$. This number equals to $(k-3)\times {k-3\choose k-3}$ by the induction hypothesis.

\item[(4)] $U_1=A_1$, $U_2\backslash U_1=A_k$, $U_3\backslash U_2=A_2$: Similar to the previous case the number of $k$-connected flags of $C_n$ with this partition set is equal to $(k-3)\times {k-3\choose k-3}$.
\end{itemize}

\medskip

Now note that since $A_2$ and $A_k$ are not adjacent just one of the two elements
\[\Uc:\ \, A_1\subsetneq A_1\cup A_2\subsetneq A_1\cup A_2\cup A_k\subsetneq U_4\subsetneq \cdots\subsetneq U_k\]
and
\[\Uc':\ \, A_1\subsetneq A_1\cup A_k\subsetneq A_1\cup A_2\cup A_k\subsetneq U_4\subsetneq \cdots\subsetneq U_k\]
will be in $\Sf_{k}(G,q)$.
This implies that $|\Sf_{k}(G,q)|=(1+1+(k-3))\times {n\choose k}=(k-1)\times{n\choose k}$. We get $\beta_i=\beta_{i,i+1} = i \, {n \choose i+1}$ for $i \geq 1$.

\medskip

For example
for $G=C_5$ we have
\[
\beta_0= 1, \ \beta_1=10, \ \beta_2=20, \ \beta_3=15, \ \beta_4=4\ .
\]
\end{Example}

\begin{Example}
It follows from Proposition~\ref{prop:GBetti_I} that adding or removing parallel edges will not change $\beta_{i}$, since this process does not add/remove any element to/from the set $\Sf_{i+1}(G, q)$. However, the graded Betti numbers $\beta_{i,\js}$ do change by adding or removing parallel edges.
For example, consider the theta graph $G$ with two vertices $u$ and $v$ connected by $m$ edges. Then $\Sf_2(G,u)$ has the unique element $\{v\}\subsetneq\{u,v\}$ which implies that $\beta_{1}=\beta_{1,m}=1$.
\end{Example}

%%%%%%%%%%%%%%%%%%%%%%%%%%%%%%%%%%%%%%%%%%%%%%%%%%%%%%%%%%%%%%%%%%%%%%%%%%%%%%%%%%%%%%%%%%%%%%%%%%%%%%%%%%%%%%%
%%%%%%%%%%%%%%%%%%%%%%%%%%%%%%%%%%%%%%%%%%%%%%%%%%%%%%%%%%%%%%%%%%%%%%%%%%%%%%%%%%%%%%%%%%%%%%%%%%%%%%%%%%%%%%%

\subsection{Relation to maximal reduced divisors}
\label{sec:Connection}

Recall the definition of reduced divisors.

\begin{Definition}
Let $(\Gamma, v_0)$ be a pointed graph. A divisor $D \in \Div(\Gamma)$ is called {\em $v_0$-reduced} if it satisfies the following two conditions:
\begin{itemize}
\item[(i)] $D(v) \geq 0$ for all $v \in V(\Gamma)\backslash \{v_0\}$.
\item[(ii)] For every nonempty subset $A \subseteq V(\Gamma)\backslash \{v_0\}$, there exists a vertex $v \in A$ such that $D(v) < \outdeg_A(v)$.
\end{itemize}
\end{Definition}
These divisors arise precisely from the normal forms with respect to the Gr\"obner bases given in Theorem~\ref{thm:Cori}. There is a well-known algorithm due to Dhar for checking whether a given divisor is reduced (see, e.g., \cite{FarbodMatt12} and references therein).

\medskip

Recall from Definition~\ref{def:U^k} that given $\Uc \in \Sf_{k}(G, q)$ we obtain a graph $G_{/\Uc}$ from $G$ by contracting all the unoriented edges of $G(\Uc)$. The contraction map $\phi: G\rightarrow G_{/\Uc}$ induces the map
\[
\phi_{\ast}:\Div(G)\rightarrow \Div(G_{/\Uc})\ \ \text{ with }\ \
\phi_{\ast}(\sum_{v\in V(G)}a_v (v))=
\sum_{v\in V(G)}a_v (\phi(v)).
\]

Assume $U_1$ is the part of $\Uc$ containing $q$ and let $q'=\phi(U_1)\in V(G_{/\Uc})$.
\begin{Lemma}
$\phi_{\ast}(D(\Uc))=E+ \mathbf{1}$, where $E$ is a maximal $q'$-reduced divisor and $\mathbf{1}$ is the all-one divisor.
\end{Lemma}

\begin{proof}
This follows from the well-known fact that Dhar's algorithm gives a one-to-one correspondence between acyclic orientations with unique source at $v_0$ and maximal $v_0$-reduced divisors; given such an acyclic orientation the corresponding $v_0$-reduced divisor is $\sum_{v \in V(\Gamma)}{(\indeg(v)-1)}(v)$ (see, e.g., \cite{BensonTetali08}). The result now follows from Remark~\ref{rmk:indeg}.
\end{proof}

Since different acyclic orientations with unique source at $q'$ give rise to inequivalent $q'$-reduced divisors we deduce that if $\Uc,\Vc \in \Sf_{k}(G, q)$ and the graphs $G_{/\Uc}$ and $G_{/\Vc}$ coincide, then $\phi_{\ast}(D(\Uc))- \mathbf{1}$ and $\phi_{\ast}(D(\Vc))- \mathbf{1}$ are two inequivalent maximal reduced divisors. These observations lead to the following formula for Betti numbers which, in an equivalent form, was conjectured in \cite{perkinson} for $I_G$:
\[
\begin{aligned}
\beta_i &= \sum_{G_{/\Uc}}{ |\{D: \, D \text{ is a maximal $v_0$-reduced divisor on } G_{/\Uc}\}|}\\
&= \sum_{G_{/\Uc}}{ | \{\text{acyclic orientations of $G_{/\Uc}$ with unique source at $v_0$} \}|}
\end{aligned}
\]
where the sum is over all {\em distinct} contracted graphs $G_{/\Uc}$ as $\Uc$ varies in $\Sf_{i+1}(G, q)$, and $v_0$ is an arbitrary vertex of $G_{/\Uc}$.

\medskip

Here is another connection with reduced divisors. Hochster's formula for computing the Betti numbers topologically (see, e.g., \cite[Theorem 9.2]{MillerSturmfels}), when applied to $I_G$ and the ``nice'' grading by $\Pic(G)$, says that for each $\js \in \Pic(G)$ the graded Betti number $\beta_{i,\js}(R / I_G)$ is the dimension of the $i ^{\rm th}$ reduced homology of the simplicial complex
\[
\Delta_{\js} = \{\supp(E): \, 0 \leq E \leq D' \in |\, \js \,|\}
\]
where $|\, \js \,|$ denotes the linear system of $\js \in \Pic(G)$. One can use this to give an alternate proof for the highest graded Betti numbers. The following is a simplification of the proof of \cite[Theorem 7.7]{perkinson} (see also Example~\ref{ex:acyc}).

\begin{Lemma} \label{lem:deltaD}
For $\js \in \Pic(G)$, we have $\beta_{n-1,\js}(R/I_G)=1$ if and only if \[\js \sim E + \mathbf{1}\] where $E$ is a maximal $q$-reduced divisor.
\end{Lemma}
\begin{proof}
$\beta_{n-1,\js}(R/I_G)=1$ if and only if $\Delta_{\js}$ is homotopy equivalent to an $(n-1)$-sphere. This is equivalent to the following two conditions.
\begin{itemize}
\item[(1)] $|\, \js -\mathbf{1}\,| = \emptyset$,
\item[(2)] $|\, \js -\mathbf{1} + (v)\,| \ne \emptyset$ for any $v \in V(G)$.
\end{itemize}
Let $E$ be the unique $q$-reduced divisor equivalent to $\js - \mathbf{1}$. Then (1) is equivalent to saying $E(q) \leq -1$. But (2) for $v=q$ would require $E(q)=-1$, and for $v \ne q$ would require that $E$ be a maximal $q$-reduced divisor. This is because for all maximal reduced divisors the values of vertices $v\ne q$ add up to the same number $g= |E(G)|-|V(G)|+1$.
\end{proof}

\begin{Remark}\label{rmk:mania}
\begin{itemize}
\item[]
\item[(i)] By Remark~\ref{thm:betti} one can use Proposition~\ref{prop:GBetti_I} to read all dimensions of the reduced homologies of $\Delta_{\js}$. Although we now know all the Betti numbers, giving an explicit bijection between connected flags and the bases of the reduced homologies of $\Delta_{\js}$ is an
intriguing problem.

\item[(ii)] In a recent work, Mania \cite{Horia} studies the number of
connected components of $\Delta_{\js}$. This gives an alternate proof that 
$\beta_1=|\Sf_2(G,q)|$.
\end{itemize}
\end{Remark}

%%%%%%%%%%%%%%%%%%%%%%%%%%%%%%%%%%%%%%%%%%%%%%%%%%%%%%%%%%%%%%%%%%%%%%%%%%%%%%%%%%%%%%%%%%%%%%%%%%%%%%%%%%%%%%%
%%%%%%%%%%%%%%%%%%%%%%%%%%%%%%%%%%%%%%%%%%%%%%%%%%%%%%%%%%%%%%%%%%%%%%%%%%%%%%%%%%%%%%%%%%%%%%%%%%%%%%%%%%%%%%%

\bibliographystyle{plain}
\bibliography{Betti2012}

\begin{thebibliography}{10}

\bibitem{BN1}
Matthew Baker and Serguei Norine.
\newblock Riemann-{R}och and {A}bel-{J}acobi theory on a finite graph.
\newblock {\em Adv. Math.}, 215(2):766--788, 2007.

\bibitem{FarbodMatt12}
Matthew Baker and Farbod Shokrieh.
\newblock Chip-firing games, potential theory on graphs, and spanning trees.
\newblock {\em J. Combin. Theory Ser. A}, 120(1):164--182, 2013.

\bibitem{popescu}
Dave Bayer, Sorin Popescu, and Bernd Sturmfels.
\newblock Syzygies of unimodular {L}awrence ideals.
\newblock {\em J. Reine Angew. Math.}, 534:169--186, 2001.

\bibitem{BensonTetali08}
Brian Benson, Deeparnab Chakrabarty, and Prasad Tetali.
\newblock {$G$}-parking functions, acyclic orientations and spanning trees.
\newblock {\em Discrete Math.}, 310(8):1340--1353, 2010.

\bibitem{Biggs97}
Norman Biggs.
\newblock Algebraic potential theory on graphs.
\newblock {\em Bull. London Math. Soc.}, 29(6):641--682, 1997.

\bibitem{conca}
Aldo Conca, Serkan Ho{\c{s}}ten, and Rekha~R. Thomas.
\newblock Nice initial complexes of some classical ideals.
\newblock In {\em Algebraic and geometric combinatorics}, volume 423 of {\em
  Contemp. Math.}, pages 11--42. Amer. Math. Soc., Providence, RI, 2006.

\bibitem{CoriRossinSalvy02}
Robert Cori, Dominique Rossin, and Bruno Salvy.
\newblock Polynomial ideals for sandpiles and their {G}r\"obner bases.
\newblock {\em Theoret. Comput. Sci.}, 276(1-2):1--15, 2002.

\bibitem{Dhar90}
Deepak Dhar.
\newblock Self-organized critical state of sandpile automaton models.
\newblock {\em Phys. Rev. Lett.}, 64(14):1613--1616, 1990.

\bibitem{Anton}
Anton Dochtermann and Raman Sanyal.
\newblock Laplacian ideals, arrangements, and resolutions.
\newblock Preprint available at \href{http://arxiv.org/abs/1212.6244}{{\tt
  ar{X}iv:1212.6244}}, 2012.

\bibitem{Eisenbud}
David Eisenbud.
\newblock {\em Commutative algebra}, volume 150 of {\em Graduate Texts in
  Mathematics}.
\newblock Springer-Verlag, New York, 1995.
\newblock With a view toward algebraic geometry.

\bibitem{EisenbudSyz}
David Eisenbud.
\newblock {\em The geometry of syzygies}, volume 229 of {\em Graduate Texts in
  Mathematics}.
\newblock Springer-Verlag, New York, 2005.
\newblock A second course in commutative algebra and algebraic geometry.

\bibitem{Gabrielov93}
Andrei Gabrielov.
\newblock Abelian avalanches and {T}utte polynomials.
\newblock {\em Phys. A}, 195(1-2):253--274, 1993.

\bibitem{GathmannKerber}
Andreas Gathmann and Michael Kerber.
\newblock A {R}iemann-{R}och theorem in tropical geometry.
\newblock {\em Math. Z.}, 259(1):217--230, 2008.

\bibitem{Singular}
Gert-Martin Greuel and Gerhard Pfister.
\newblock {\em A \textbf{{S}ingular} introduction to commutative algebra}.
\newblock Springer, Berlin, extended edition edition, 2008.
\newblock With contributions by Olaf Bachmann, Christoph Lossen and Hans
  Sch{\"o}nemann.

\bibitem{Kamoi}
Yuji Kamoi.
\newblock Noetherian rings graded by an abelian group.
\newblock {\em Tokyo J. Math.}, 18(1):31--48, 1995.

\bibitem{Lorenzini89}
Dino~J. Lorenzini.
\newblock Arithmetical graphs.
\newblock {\em Math. Ann.}, 285(3):481--501, 1989.

\bibitem{Horia}
Horia Mania.
\newblock Wilmes' conjecture and boundary divisors.
\newblock Preprint available at \href{http://arxiv.org/abs/1210.8109}{{\tt
  ar{X}iv:1210.8109}}, 2012.

\bibitem{Madhu}
Madhusudan Manjunath, Frank-Olaf Schreyer, and John Wilmes.
\newblock Minimal free resolutions of the {$G$}-parking function ideal and the
  toppling ideal.
\newblock To appear in Trans. Amer. Math. Soc., Preprint available at
  \href{http://arxiv.org/abs/1210.7569}{{\tt ar{X}iv:1210.7569}},, 2012.

\bibitem{MadhuBernd}
Madhusudan Manjunath and Bernd Sturmfels.
\newblock Monomials, binomials and {R}iemann--{R}och.
\newblock {\em J. Algebraic Combin.}, 37(4):737--756, 2013.

\bibitem{MK08}
Grigory Mikhalkin and Ilia Zharkov.
\newblock Tropical curves, their {J}acobians and theta functions.
\newblock In {\em Curves and abelian varieties}, volume 465 of {\em Contemp.
  Math.}, pages 203--230. Amer. Math. Soc., Providence, RI, 2008.

\bibitem{MillerSturmfels}
Ezra Miller and Bernd Sturmfels.
\newblock {\em Combinatorial commutative algebra}, volume 227 of {\em Graduate
  Texts in Mathematics}.
\newblock Springer-Verlag, New York, 2005.

\bibitem{Fatemeh}
Fatemeh Mohammadi.
\newblock Prime splittings of determinantal ideals.
\newblock Preprint available at \href{http://arxiv.org/abs/1208.2930}{{\tt
  ar{X}iv:1208.2930}}, 2012.

\bibitem{FarbodFatemeh2}
Fatemeh Mohammadi and Farbod Shokrieh.
\newblock Divisors on graphs, binomial and monomial ideals, and cellular
  resolutions.
\newblock Preprint available at \href{http://arxiv.org/abs/1306.5351}{{\tt
  ar{X}iv:1306.5351}}, 2013.

\bibitem{novik}
Isabella Novik, Alexander Postnikov, and Bernd Sturmfels.
\newblock Syzygies of oriented matroids.
\newblock {\em Duke Math. J.}, 111(2):287--317, 2002.

\bibitem{peeva}
Irena Peeva.
\newblock {\em Graded syzygies}, volume~14 of {\em Algebra and Applications}.
\newblock Springer-Verlag London Ltd., London, 2011.

\bibitem{perkinson}
David Perkinson, Jacob Perlman, and John Wilmes.
\newblock Primer for the algebraic geometry of sandpiles.
\newblock Preprint available at \href{http://arxiv.org/abs/1112.6163}{{\tt
  ar{X}iv:1112.6163}}, 2011.

\bibitem{PostnikovShapiro04}
Alexander Postnikov and Boris Shapiro.
\newblock Trees, parking functions, syzygies, and deformations of monomial
  ideals.
\newblock {\em Trans. Amer. Math. Soc.}, 356(8):3109--3142, 2004.

\bibitem{RaynaudPicard}
Michel Raynaud.
\newblock Sp\'ecialisation du foncteur de {P}icard.
\newblock {\em Inst. Hautes \'Etudes Sci. Publ. Math.}, (38):27--76, 1970.

\bibitem{Schreyer}
Frank Schreyer.
\newblock Die {B}erechnung von {S}yzygien mit dem verallgemeinerten
  {W}eierstrass'schen {D}ivisionssatz.
\newblock Diplom Thesis, University of Hamburg, Germany., 1980.

\bibitem{Spear}
D.~A. Spear.
\newblock A constructive approach to commutative ring theory.
\newblock In {\em Proceedings of the the 1977 MACSYMA Users' Conference, NASA
  CP-2012}, pages 369--376, 1977.

\bibitem{Wilmes}
John Wilmes.
\newblock Algebraic invariants of sandpile graphs.
\newblock Bachelor�s Thesis, Reed College, Portland, OR, 2010.

\end{thebibliography}

\end{document}